\pdfoutput=1
\documentclass[hidelinks,onefignum,onetabnum]{siamart220329}

\usepackage{scalerel,stackengine}
\stackMath
\newcommand\reallywidehat[1]{%
\savestack{\tmpbox}{\stretchto{%
\scaleto{%
\scalerel*[\widthof{\ensuremath{#1}}]{\kern-.6pt\bigwedge\kern-.6pt}%
{\rule[-\textheight/2]{1ex}{\textheight}}
}{\textheight}%
}{0.5ex}}%
\stackon[1pt]{#1}{\tmpbox}%
}
\usepackage[margin=1.5in]{geometry}
\usepackage{lipsum}
\usepackage{amsfonts}
\usepackage{graphicx,tabularx,multirow}
\usepackage{epstopdf}
\usepackage{algorithm}
\usepackage{enumitem}
\usepackage{amsmath, amssymb, mathtools,amsfonts,enumitem}
\usepackage{algpseudocode}
\usepackage{caption, subcaption}
\usepackage{amsopn}
\usepackage[title]{appendix}
\usepackage{color}
\usepackage{math}
\usepackage[sort]{cite}
\usepackage{hyperref}
\usepackage{arydshln}
\definecolor{ao(english)}{rgb}{0.0, 0.5, 0.0}
\definecolor{applegreen}{rgb}{0.55, 0.71, 0.0}
\hypersetup{
	colorlinks=true,
	linkcolor=ao(english),
	citecolor=blue}
\ifpdf
\DeclareGraphicsExtensions{.eps,.pdf,.png,.jpg}
\else
\DeclareGraphicsExtensions{.eps}
\fi

\title{{\red Parallel} Sequential Quadratic Programming with Overlapping Graph Decomposition and Exact Augmented Lagrangian 
}

\author{
Runxin Ni$^{\P}$\thanks{Department of Statistics, University of Chicago, Chicago, IL (\email{runxin.ni@uchicago.edu}).}
\and
Haoxuan Wang$^{\P}$\thanks{School of Industrial and Systems Engineering, Georgia Tech, Atlanta, GA (\email{hwang3111@gatech.edu}, \email{senna@gatech.edu}).}
\and
Sen Na\footnotemark[2]
\and
Sungho Shin\thanks{Department of Chemical Engineering, Massachusetts Institute of Technology, Cambridge, MA (\email{sushin@mit.edu})}
\and
Mihai Anitescu\thanks{Mathematics and Computer Science Division, Argonne National Laboratory, Lemont, IL (\email{anitescu@mcs.anl.gov}).}
}

\headers{Parallel SQP with Graph Decomposition and Augmented Lagrangian}{Runxin Ni, Haoxuan Wang, Sen Na, Sungho Shin, and Mihai Anitescu}

\newsiamremark{remark}{Remark}
\newsiamremark{hypothesis}{Hypothesis}
\newsiamremark{assumption}{Assumption}
\crefname{hypothesis}{Hypothesis}{Hypotheses}
\newsiamthm{claim}{Claim}
\allowdisplaybreaks

\makeatletter
\AtBeginDocument{%
  \def\refstepcounter@optarg[#1]#2{%
    \cref@old@refstepcounter{#2}%
    \cref@constructprefix{#2}{\cref@result}%
    \@ifundefined{cref@#1@alias}%
      {\def\@tempa{#1}}%
      {\def\@tempa{\csname cref@#1@alias\endcsname}}%
    \protected@edef\cref@currentlabel{%
      [\@tempa][\arabic{#2}][\cref@result]%
      \csname p@#2\endcsname\csname the#2\endcsname}%
  }%
}
\makeatother

\begin{document}

\makeatletter
\def\EqualContributionNote{%
  \begingroup
  \endgroup
}
\makeatother

\renewcommand{\thefootnote}{}
\footnotetext{\hspace*{-2.4em}
\hbox to1.8em{\hss$^{\P}$}These authors contributed equally to this work.}%
\maketitle
\footnotetext{\hspace*{-2.1em}\hbox to1.8em{\hss$^1$}\relax
This material is based upon work supported by the U.S. Department of Energy, Office of Science, Office of Advanced Scientific Computing Research (ASCR) under Contract DE-AC02-06CH11347.}
\def\thefootnote{\arabic{footnote}}

\begin{abstract}
In this paper, we address the challenge of solving large-scale graph-structured nonlinear programs (gsNLPs) in a scalable manner. GsNLPs are problems in which the objective and constraint functions are associated with nodes on a graph and depend on the variables of adjacent nodes. This graph-structured formulation encompasses various specific instances, such as dynamic optimization, PDE-constrained optimization, multi-stage stochastic optimization, and general network optimization. By leveraging the sequential quadratic programming (SQP) framework, we propose a globally convergent overlapping graph decomposition method to solve large-scale gsNLPs under standard mild regularity conditions on the graph topology. In each iteration, we perform an overlapping graph decomposition to compute an approximate Newton direction in a parallel environment. Then, we select a suitable stepsize and update the primal-dual iterate by performing a backtracking line search on an exact augmented Lagrangian merit function. Built on the exponential decay of sensitivity of gsNLPs, we show that the approximate Newton direction is a descent direction of the augmented Lagrangian, which leads to global convergence with local linear convergence rate. In particular, global convergence is achieved for sufficiently large overlaps, and the local linear convergence rate improves exponentially in terms of the overlap size. Our results match existing state-of-the-art guarantees established for dynamic programs (which simply correspond to linear graphs). We validate the theory on~{\red two} PDE-constrained problems: a semilinear elliptic problem and {\red a boundary heating problem}.

\end{abstract}


\section{Introduction}\label{sec1}

We consider {\it graph-structured nonlinear programs} (gsNLPs) \cite{Shin2022Exponential}:
\begin{subequations}\label{eq21}
\begin{align}
\min_{\{\bx_i\in \mR^{r_i}\}_{i\in\mV}} \quad & \sum_{i\in\mV} f_i(\{\bx_j\}_{j\in N_{\mG}[i]}), \label{eq21a}\\
\text{s.t. }\quad\;\; &  \bc_i(\{\bx_{j}\}_{j\in N_{\mG}[i]})=\textbf{0},\; \; i\in\mathcal{\mV}, \;\; (\bl_i). \label{eq21b}
\end{align}
\end{subequations}
Here, $\mG = (\mV, \mE)$ is an undirected graph, $\mV$ is the strictly ordered node set, $\mE\subseteq\{\{i,j\}:i,j\in\mV, i\neq j\}$ is the edge set, and $N_{\mG}[i] \coloneqq\{j\in\mV:\{i,j\}\in\mE\}\cup\{i\}$ is the closed neighborhood of $i\in\mV$ on $\mG$. For each node $i$, we have the following associated quantities: $\bx_i\in \mR^{r_i}$ is the primal variable, $f_i:\Pi_{j\in N_{\mG}[i]} \mR^{r_j}\rightarrow \mR$ is the objective function, $\bc_i:\Pi_{j\in N_{\mG}[i]} \mR^{r_j}\rightarrow \mR^{m_i}$ is the equality constraint function, and $\bl_i\in \mR^{m_i}$ is the dual variable associated with the constraint $\bc_i$.
Problem \eqref{eq21} arises in various applications, including optimal control \cite{Biegler2010Nonlinear}, partial differential equation (PDE)-constrained problems \cite{Biegler2003Large}, multi-stage stochastic optimization \cite{Pereira1991Multi}, and network optimization \cite{Zlotnik2015Optimal}.

In this paper, we develop a scalable method for solving large-scale problems of form \cref{eq21}. The complexity of these problems stems from the large size of the underlying graph $\mG$, which can make many centralized methods intractable. To address this issue, decomposition-based methods are of particular interest. In these methods, the intractable full problem is decomposed into a set of smaller subproblems, each designed to be individually tractable on the local machine. Decomposition methods apply an iterative manner to ensure the convergence of the iterates toward the full solution; the convergence is achieved by exchanging solution information among (adjacent) subproblems in each iteration step. Examples of decomposition methods include Lagrangian decomposition \cite{Jackson2003Temporal, Beccuti2004Temporal}, temporal decomposition \cite{Jackson2003Temporal, Xu2018Exponentially}, Jacobi or Gauss--Seidel methods \cite{Zavala2016New}, and alternating direction method of multipliers (ADMM) \cite{Ghadimi2015Optimal}.$\hskip0.2cm$

Although decomposition methods offer several advantages over centralized methods, such as computational scalability and flexibility in implementation, they may exhibit slower convergence rates compared to centralized methods (e.g., Ipopt) \cite{Kozma2014Benchmarking}. The slow convergence can be attributed to their iterative manner and limited amount of information shared across the subproblems. As investigated in the benchmark study \cite{Kozma2014Benchmarking}, the practically observed convergence rates of existing decomposition schemes are often too slow to make those schemes practically useful. Such limitations motivate the design of a new decomposition paradigm, which aims to resolve the issue of slow convergence while preserving the desired scalability.

Unlike the majority of decomposition methods that exclusively decompose the full problem, there is a fast growing literature studying the {\it overlapping decomposition} \cite{Shin2020Decentralized, Shin2020Overlapping, Na2023Fast, Na2022Convergence}. This type of methods decomposes the full problem into a set of \textit{overlapping} subproblems. The overlap can facilitate exchanging more information between (adjacent) subproblems in each iteration, hereby significantly improving the convergence behavior compared to the methods that perform exclusive decomposition. One particular instance of overlapping decomposition is the {\it overlapping Schwarz method}, whose origin can be traced back to classical domain decomposition methods proposed for solving second-order elliptic PDEs \cite{Chan1994Additive, chan1994domain}. The overlapping Schwarz method has been applied to various structured problems, and its convergence properties have been investigated under different settings, including unconstrained quadratic programs \cite{Shin2020Decentralized}, constrained quadratic programs \cite{Shin2020Overlapping}, and nonlinear optimal control \cite{Na2022Convergence}. This series of literature has established not only the \textit{local convergence} of the Schwarz method, but also the relationship between its local convergence rate and the overlap size. Specifically, the authors have shown that \textit{the local linear convergence rate of the Schwarz method improves exponentially with respect to the overlap size}, meaning that one can address the issue of slow convergence in classical (non-overlapping) decomposition methods by slightly increasing the subproblem complexity.

While the overlapping Schwarz method may enjoy some benefits over classical methods, it faces two major concerns. First, the subproblems in each iteration of the Schwarz method are nonlinear and nonconvex, and the method solves \textit{all} subproblems to \textit{optimality}, albeit in parallel, leading to significant per-iteration computational costs on each local machine. Second, the Schwarz method lacks the utilization of sophisticated globalization strategies, such as backtracking line search and trust region, to enforce global convergence. In particular, the aforementioned works on the Schwarz method only provided a local convergence guarantee, meaning that the convergence is guaranteed to occur only when the initial guess is sufficiently close to the local solution. The above concerns have motivated a recent enhanced algorithm design to reduce per-iteration computational costs and incorporate a globalization strategy into the Schwarz method, while still preserving its desired fast local linear convergence rate \cite{Na2023Fast}. 
In that work, the authors designed a {\it Fast Overlapping Temporal Decomposition} (FOTD) procedure for solving nonlinear dynamic programs. The method embeds overlapping temporal decomposition into the Sequential Quadratic Programming (SQP) framework, where the \mbox{decomposition} is applied to the (large-scale) quadratic program in each SQP step, making all subproblems easily-solvable quadratic programs. Additionally, FOTD performs backtracking line search on a differentiable exact augmented Lagrangian merit function. Although the exact augmented Lagrangian serves as an alternative to non-differentiable $\ell_1/\ell_2$ merit functions in SQP methods (see \cite{Na2022adaptive, Hong2023Constrained} and references therein), recent research has demonstrated that it may offer particular benefits when applied to dynamic programs under different settings \cite{Na2021Global, Na2023Fast}. Specifically, with such a globalization strategy, \cite{Na2023Fast} showed that FOTD enjoys global convergence while preserving the local linear convergence rate of the Schwarz method.

Inspired by the promising results of FOTD in dynamic programs, this paper generalizes that method to a much broader class of nonlinear programs. In particular, we design the \textit{Fast Overlapping Graph Decomposition} (FOGD) method to handle gsNLPs of the form \cref{eq21}, so that dynamic programs simply correspond to the special case of linear graphs. 
We embed overlapping graph decomposition within an SQP framework (in the spirit of \cite{Na2023Fast}). At each SQP step, we compute approximate Newton directions \textit{in parallel} by solving small overlapping subproblems, while carefully specifying their boundary losses and constraints. The stepsize is then chosen via line search on an augmented Lagrangian merit function.

{\red Our main technical contribution lies in demonstrating the global and local convergence of the FOGD method. Under a mild graph growth condition on $\mG$, we establish error bounds between the approximate Newton direction (computed in parallel) and the exact Newton direction. We show that the approximate Newton direction is a descent direction of the augmented Lagrangian, provided that we specify the boundary loss and overlap size properly. This ultimately leads to global convergence with an explicit local \textit{linear, overlap-dependent} rate. We mention that the growth condition excludes tree graphs, which grow exponentially as seen in multi-stage stochastic optimization. This is simply because, even if the errors introduced by graph decomposition decay exponentially with the distance to the graph boundary, accumulating such errors over an exponential number of boundary nodes may still be uncontrollable.
That said, to the best of our knowledge, FOGD is the first overlapping decomposition scheme that solves gsNLPs with a global convergence guarantee while matching the state-of-the-art local linear convergence rate of the Schwarz method. Below, we highlight some key differences relative to two closely related lines of work: FOTD for dynamic programs and distributed optimization over graphs.
}

{\red
On the one hand, compared with FOTD \cite{Na2023Fast}, our setting and analysis extend much beyond line graphs and the modifications of the quadratic subproblem objective at the endpoints (cf. their (3.9)). Our subproblem formulation \eqref{eq37_sec3} applies to general graphs and treats boundary terms and coupling constraints consistently across different node sets. Moreover, Theorem 4.1 in \cite{Na2023Fast} strictly leverages the line-graph decay of sensitivity property, while our result in \Cref{thm42} extends the decomposition error analysis to broader graph families under the growth condition. Our local convergence analysis in \Cref{lm62} similarly goes beyond the boundary-node analysis in Theorem 6.2 of \cite{Na2023Fast}. On the other hand, large-scale optimization over graphs has been studied extensively, including distributed methods \cite{Notarstefano2019Distributed, Notarnicola2020Distributed, Scutari2019Distributed, Sun2022Distributed} and asynchronous variants \cite{Cannelli2020Asynchronousa}, most of which are first-order or based on successive convex approximation (SCA). 
We compare these methods with ours from three aspects:
\begin{enumerate}[wide, labelindent=0pt, label=\textbf{(\alph*)},topsep=0pt]
\setlength\itemsep{0.0em}
\item \textbf{Problem class.}
Many distributed methods focus on unconstrained or convexly constrained formulations and often allow composite objectives consisting of a smooth (possibly nonconvex) term plus a convex (possibly nonsmooth) regularizer \cite{Notarstefano2019Distributed, Scutari2019Distributed, Sun2022Distributed}. In parallel computing settings, asynchronous parallel frameworks extend this composite model to broad classes of nonconvex constrained problems \cite{Cannelli2020Asynchronous}. Over graphs, \cite{Cannelli2020Asynchronousa} studies a partially separable sum-cost model with block-separable closed convex constraints and convex (possibly nonsmooth) local terms: each agent controls a block of variables and each smooth cost term depends only on a small neighborhood of that block. 
By comparison, we study gsNLPs with smooth nonlinear objectives and nonlinear equality constraints that couple neighboring variables \eqref{eq21}, motivated by PDE-constrained problems and large-scale optimal control. Such nonlinear constraint coupling is typically not the focus in unconstrained or separably constrained formulations (e.g., aggregative feedback optimization) \cite{Carnevale2024Nonconvex}. Moreover, methods focusing on affine coupling constraints over networks \cite{Yarmoshik2024Decentralized} or on linear-and-conic constrained nonconvex programs \cite{Dvurechensky2025Hessian} do not directly cover our nonlinear equality constraints that couple neighboring variables; generalized Nash-equilibrium models with coupling constraints \cite{Scarabaggio2024Local}~also~address a different multi-agent problem class.

\vskip2pt
\item \textbf{Assumptions and convergence guarantees.}
For distributed first-order and SCA methods, convergence to stationary points is typically proved under standard smoothness assumptions with mild regularity and boundedness conditions \cite{Notarstefano2019Distributed, Scutari2019Distributed, Sun2022Distributed, Notarnicola2020Distributed}. 
However, the resulting convergence rates are often sublinear, while linear rates usually require stronger regularity such as strong convexity or error-bound-type conditions.
Asynchronous frameworks further rely on explicit asynchrony models and assumptions (e.g. probabilistic models or bounded delays ensuring that information does not become infinitely old) \cite{Cannelli2020Asynchronous, Cannelli2020Asynchronousa}; in particular, \cite{Cannelli2020Asynchronousa} establishes sublinear convergence to stationarity for nonconvex objectives and an $R$-linear rate under the Luo--Tseng error bound condition. Related distributed optimal-control works prove convergence to trajectories satisfying first-order necessary conditions \cite{Sforni2024Distributed}; stronger rate-type guarantees, when available, are obtained under additional regularity such as strong convexity or PL conditions.
All above results differ significantly from our analysis for nonlinear nonconvex problems. As discussed above, our SQP-type approach provides global convergence and a local linear rate, which improves exponentially with respect to~the~overlap~size.~In~smooth unconstrained case, our method reduces to a damped inexact~Newton~scheme.

\vskip2pt

\item \textbf{Decomposition structure, implementation, and graph classes.}
Many distributed algorithms are designed for directed and possibly time-varying communication graphs \cite{Scutari2019Distributed, Sun2022Distributed}. Asynchrony has been studied both in parallel multi-worker architectures, where block updates may use delayed information under a probabilistic delay model \cite{Cannelli2020Asynchronous}, and in networked settings, where agents update without global coordination and may rely on delayed, out-of-sync information from their neighbors \cite{Cannelli2020Asynchronousa}. A recent closely related work is the sensitivity-based distributed programming framework (SBDP+) \cite{Esch2025Enforcing}, which also targets graph-structured nonlinear programs and solves local nonlinear subproblems with neighbor-to-neighbor communication. 
However, in addition to imposing stronger assumptions on the Lagrangian Hessian matrices (e.g., their Assumptions 4.2 and 5.3) and lacking global convergence analysis, its decomposition strategy is fundamentally different from ours. SBDP+ employs first-order sensitivities and a \textit{non-overlapping} primal decomposition, in which each node $i$ optimizes only its own local block while neighboring variables are fixed at the current iterate; their influence is incorporated through sensitivity terms and a transformed primal-dual update. By comparison, our method FOGD targets static undirected graphs and adopts a synchronous but highly parallel SQP-based overlapping decomposition. Unlike SBDP+, FOGD enlarges each local subproblem with overlap and incorporates neighboring effects through carefully designed boundary losses and constraints. This construction is motivated by the exponential decay of sensitivity in gsNLPs: even when boundary data are only approximate, their influence on the interior decays rapidly with the distance to the subdomain boundary. At each SQP iteration, we solve small overlapping subproblems locally and exchange only boundary information without solving a global KKT system.

\end{enumerate}
}

\vskip5pt
\paragraph{Structure of the paper}
The remainder of this paper is organized as follows. In \Cref{sec2}, we introduce SQP and OGD and propose the FOGD procedure. In \Cref{sec3}, we establish the error bound of the approximate Newton direction. Then, we prove the global convergence of FOGD in \Cref{sec4}, and prove the local linear convergence and show the relationship between the linear rate and the overlap size in \Cref{sec5}. We present numerical experiments to demonstrate~our theoretical findings in \Cref{sec6}, and conclude with future directions in \Cref{sec7}. Some of~the proofs are deferred to the appendix to make the main paper compact.

\vskip5pt
\paragraph{Notation} 
Throughout the paper, all vectors $\bx_i$ are assumed to be column vectors. For any integer $n$, we let $[n] \coloneqq \{1, \cdots, n\}$. For any set $\mI$, we denote $|\mI|$ as the number of elements in $\mI$. For any two sets $\mI$ and $\mJ$, we let $\mI \backslash \mJ \coloneqq \{i:i\in \mI,\;i\notin \mJ\}$. For a strictly ordered node set $\mI = \{i_1, i_2, \cdots, i_{|\mI|}\}$ with $i_1 < i_2 < \cdots < i_{|\mI|}$, we define $\{\bx_i\}_{i\in \mI} \coloneqq (\bx_{i_1}; \cdots; \bx_{i_{|\mI|}})$ as the vector that stacks all variables of $\mI$. For a (one- or) two-way  indexed array $M$, $M[\mI][\mJ]$ represents the subarray where the first index lies in the (strictly ordered node) set $\mI$, and the second index lies in the set $\mJ$. For the undirected graph $\mG=(\mV,\mE)$ and a node set $\mW\subseteq \mV$, we let $N^b_{\mG}[\mW] \coloneqq \{i\in \mV:d_{\mG}(i,\mW)\leq b\}$ be the $b$-hop neighborhood of $\mW$, where $d_{\mG}(i,\mW)$ is the number of edges in the shortest path connecting the node $i$ and the set $\mW$ on graph $\mG$. We omit the superscript $b$ when $b=1$, i.e., $N_{\mG}[\mW] = N^1_{\mG}[\mW]$. We also let $N_{\mG}(\mW) \coloneqq N_{\mG}[\mW]\backslash \mW$ be the neighborhood of $\mW$ excluding itself (i.e., the open neighborhood).

\section{Fast Overlapping Graph Decomposition}\label{sec2}

We present FOGD in three steps: (1) introduce the Newton system in each SQP iteration, (2) apply OGD to approximately solve the Newton system, and (3) select a proper stepsize by performing line search on the exact~augmented Lagrangian merit function. We summarize the full method in \cref{algo2}.

\vskip4pt
\paragraph{Step 1: the Newton system in each SQP iteration}

We define the Lagrangian function of \cref{eq21} as $\mL(\bx, \bl) = f(\bx) + \bl^T\bc(\bx)$, where $f(\bx)=\sum_{i\in \mV} f_i(\{\bx_j\}_{j\in N_{\mG}[i]})$ is the objective function, $\bc(\bx)=\{\bc_i(\{\bx_{j}\}_{j\in N_\mG[i]})\}_{i\in \mV}$ are the equality constraints, and $\bl = \{\bl_i\}_{i\in\mV}$ are the dual variables associated with the constraints. We apply Newton's method to the Karush-Kuhn-Tucker (KKT) condition $\nabla\mL(\bxs,\bls) = \bnz$. In particular, given the $\tau$-th iterate $(\bxt, \blt)$, the search direction $(\Delta \bxt, \Delta \blt)$ is computed by solving the following Newton system:
\begin{equation}\label{eq31}
\begin{bmatrix}
\hHt & (\Gt)^T\\
\Gt & \bnz 
\end{bmatrix}\begin{bmatrix}
\Delta \bxt\\
\Delta \blt
\end{bmatrix} = - \begin{bmatrix}
\nbx \mLt\\
\nbl \mLt
\end{bmatrix},
\end{equation}
where $\nabla \mLt = \nabla\mL(\bxt, \blt)$ is the Lagrangian gradient, $\Gt = \nabla\bc(\bxt)$ is the constraints Jacobian, and $\hHt$ is a modification of the Lagrangian Hessian $H^\tau = \nabla^2_{\bx}\mL^\tau$. The modification is to ensure that $\hHt$ is positive definite in the null space $\{\bo: G^\tau\bo = \bnz\}$, which, together~with the~full row-rankness of $G^\tau$, ensures the well-definedness of the system \eqref{eq31} \cite[Lemma 16.1]{Nocedal2006Numerical}.$\hskip1cm$

Solving the Newton system \cref{eq31} is equivalent to solving the quadratic program:
\begin{equation}\label{eq36}
\min_{\bo}\;\; \frac{1}{2}\bo^T \hHt \bo + \bo^T \nbx \mLt, \quad\;\text{s.t. }\;\;  \bct + \Gt \bo = \bnz, \;\; (\bze),
\end{equation}
where $\bct = \bc(\bx^\tau)$ and $\bze$ is the dual vector. Since the solution of \cref{eq31} is the unique global solution of Problem \cref{eq36}, we let $((\bo^\tau)^{\boldsymbol{*}}, (\bze^\tau)^{\boldsymbol{*}})$ be the solution of \cref{eq36} and have $(\Delta \bx^\tau, \Delta \bl^\tau) = ((\bo^\tau)^{\boldsymbol{*}}, (\bze^\tau)^{\boldsymbol{*}})$. To ease the notation, we may occasionally suppress the iteration index $\tau$ when referring to the variables in a generic iteration.

\vskip4pt
\paragraph{Step 2: OGD on the Newton system} 

For large-scale gsNLPs, solving \cref{eq36} exactly can be computationally intractable. To address this issue, we apply OGD to \cref{eq36} to compute an approximate solution in a parallel environment. The OGD procedure is summarized as follows. First, we decompose the entire problem domain $\mV$ into a set of overlapping subdomains. Next, we formulate a subproblem for each subdomain by considering only the associated objectives and constraints within the subdomain. Then, we solve the subproblems in parallel; the solutions from the exclusive parts of each subdomain are composed to form the final approximate Newton direction.
{\red
Under our computation scheme, we do not need to store the full matrices $\hHt$ and $\Gt$ (or the full vectors $\nabla \mLt$ and $\bct$), which would be prohibitive for large-scale problems. Instead, each subproblem requires only the corresponding local blocks in \eqref{eq37}. As a result,~the per-iteration memory usage and computational cost scale with the subdomain~size.
}

In particular, let us consider a disjoint partition $\{\mV_{\ell}\}_{\ell\in[M]}$ of the full node set $\mV$, and extend each set $\mV_{\ell}$ to its $b$-hop neighborhood, where $b\geq 1$ is a user-specified parameter (called the overlap size):
\begin{equation*}
\mV = \bigcup_{\ell=1}^M\mV_{\ell},\quad\quad \quad \mV_\ell \;\text{ disjoint with each other},\quad\quad \quad N^b_{\mG}[\mV_{\ell}] \subseteq \mW_{\ell},\;\; \forall \ell\in[M].
\end{equation*}
We may just let $\mW_{\ell}=N^b_{\mG}[\mV_{\ell}]$ in the simple case. With the above overlapping decomposition $\{\mW_\ell\}_{\ell\in[M]}$, we define the subproblem $\mS\mP_{\ell}^{\tau, \mu}$ for the subdomain $\mW_{\ell}$ in the $\tau$-th iteration as:
\begin{subequations}\label{eq37}
\begin{align}
\min_{\bto_{\ell}}\;\; & \frac{1}{2}\bto_{\ell}^T\hHt[\mW_{\ell}][\mW_{\ell}] \bto_{\ell} + \bto_{\ell}^T \nabla_{\bx} \mLt[\mW_{\ell}] + \frac{\mu}{2}\nbr{\bct[\hat{\mW}_{\ell}] + G^\tau[\hat{\mW}_{\ell}][\mW_{\ell}] \bto_{\ell} }^2, \label{eq37a}\\
\text{s.t.}\;\; &\bct[\mW_{\ell}\backslash \tmW_{\ell}]  + \Gt[\mW_{\ell}\backslash \tmW_{\ell}][\mW_{\ell}] \bto_{\ell} = \bnz,\quad (\btze_{\ell}), \label{eq37b}
\end{align}
\end{subequations}
where $\tmW_{\ell} \coloneqq {N_{\mG}\left(N_{\mG}(\mW_{\ell})\right)}\cap \mW_{\ell}$ is the set of internal boundary nodes of $\mW_{\ell}$, $\hat{\mW}_{\ell} \coloneqq N_{\mG}(\mW_{\ell})\cup \tmW_{\ell}$ is the set of internal and external boundary nodes of $\mW_{\ell}$, and $\bto_{\ell} = \bo[\mW_{\ell}]$ and $\btze_{\ell} = \bze[\mW_{\ell}\backslash \tmW_{\ell}]$ are the primal and dual (sub-)variables. The solution of Problem \eqref{eq37} is denoted as $((\bto_{\ell}^\tau)^{\boldsymbol{*}}, (\btze_{\ell}^\tau)^{\boldsymbol{*}})$, though again we may omit the index $\tau$ to ease the notation. An illustration of all the sets $\mV_{\ell}$, $\mW_{\ell}$, $N_{\mG}(\mW_{\ell})$, $\tmW_{\ell}$, and $\hat{\mW}_{\ell}$ is shown in \cref{figset}.

The subproblem \eqref{eq37} is essentially the truncation of the full problem \eqref{eq36} onto the set $\mW_{\ell}$. However, there is a key difference: only the \textit{internal constraints} --- those that fully depend on the variables within $\mathcal{W}_{\ell}$ --- are enforced in \eqref{eq37b}, while the \textit{coupling constraints} --- those that partially depend on the variables outside $\mathcal{W}_{\ell}$ --- are incorporated as an augmented Lagrangian-like regularization term in \eqref{eq37a}. As will be shown in \cref{lm41}, this formulation ensures the subproblems to satisfy the desired regularity conditions, provided that the conditions are satisfied by the full problem \eqref{eq36} and the parameter $\mu>0$ is large enough.

In each SQP iteration, we solve $M$ overlapping subproblems $\{\mS\mP_{\ell}^{\mu}\}_{\ell\in[M]}$ in parallel, and for each subproblem $\ell$, we retain only the solution $(\btosi[\mV_\ell], \btzesi[\mV_\ell])$ within the disjoint set $\mV_{\ell}$. Then, we compose all solution pieces to obtain the approximate Newton direction $(\tD\bx, \tD\bl)$ (in the $\tau$-th iteration). In particular, we introduce the following definition. For a subdomain variable $\bto_{\ell}$ (similar for $\btze_{\ell}$), we let $\bto_{\ell, k} \coloneqq \bto_{\ell}[\{k\}]$ be the variable associated with the node $k\in \mW_{\ell}$. This notation aims to distinguish $\bto_{\ell_1, k}$ and $\bto_{\ell_2, k}$ when both $k\in \mW_{\ell_1}$ and $k\in \mW_{\ell_2}$.

\begin{definition}[\textbf{Composition and Decomposition}]\label{df22}
Given $M$ subdomain variables $\{(\bto_{\ell},\btze_{\ell})\}_{\ell=1}^{M}$, we define the composition operator as $\mC(\{(\bto_{\ell},\btze_{\ell})\}_{\ell=1}^{M}) = (\bo,\bze)$, where for any node $k\in \mV$, $(\bo_k,\bze_k)  \coloneqq (\bto_{\ell,k}, \btze_{\ell,k})$ for (the unique) $\ell$ such that $k\in \mV_{\ell}$. Conversely, given a full-domain variable $(\bo,\bze)$, we define the decomposition operator as $\mD(\bo,\bze)\hskip-0.5pt =\hskip-0.5pt\{\mD_{\ell}(\bo,\bze)\}_{\ell=1}^{M}$ with $\mD_{\ell}(\bo,\bze)= (\bto_{\ell}, \btze_{\ell}) = (\bo[\mW_{\ell}], \bze[\mW_{\ell}\backslash\tilde{\mW}_{\ell}])$.
\end{definition}

By \cref{df22}, at the $\tau$-th iteration, we have $(\tD \bx^{\tau}, \tD \bl^{\tau}) = \mC(\{((\bto_{\ell}^\tau)^{\boldsymbol{*}}, (\btze_{\ell}^\tau)^{\boldsymbol{*}} )\}_{\ell=1}^M)$.

\begin{figure}
\setlength{\belowcaptionskip}{-6pt}
 \centering
\begin{subfigure}[b]{0.45\textwidth}
\centering
\includegraphics[height=3.4cm]{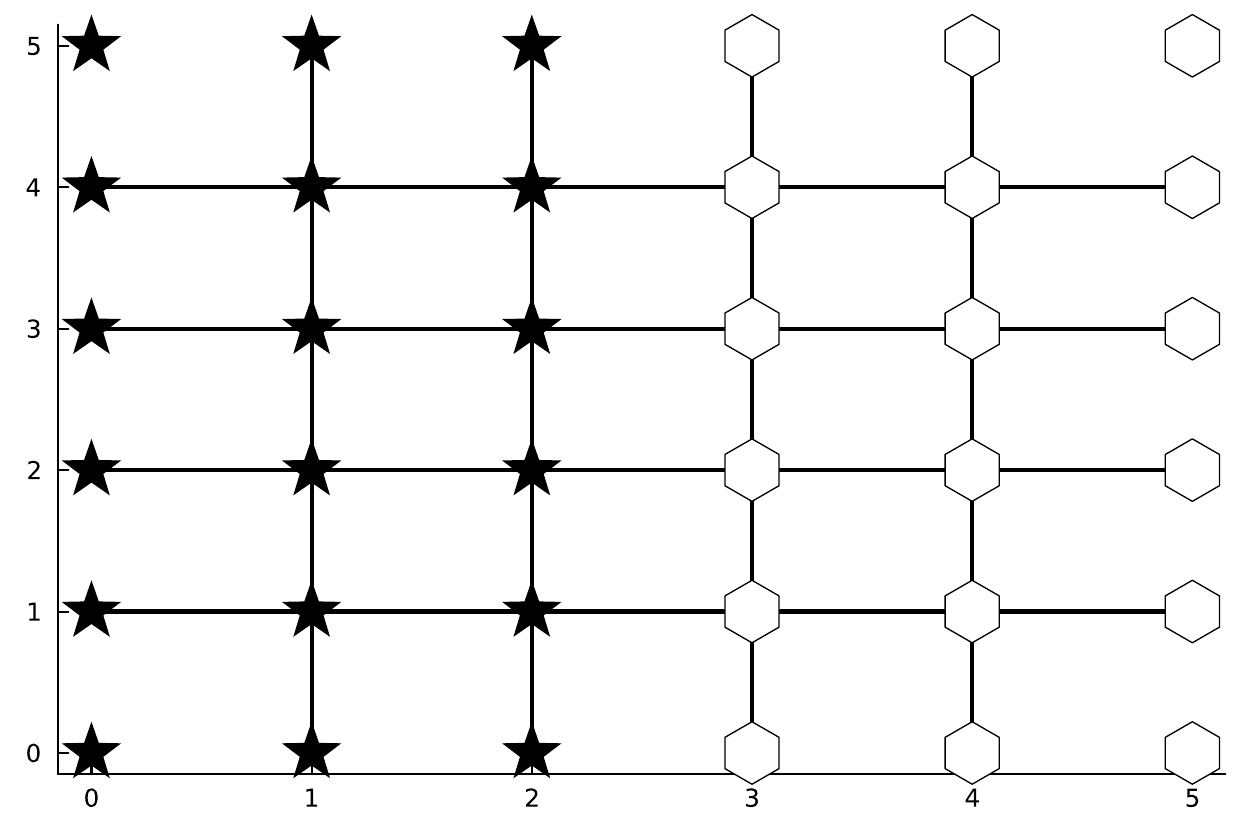}
\caption{$\mV_1$, $\mV_2$, and $\mV = \mV_1\cup\mV_2$}
\label{figset.a}
\end{subfigure}
\begin{subfigure}[b]{0.45\textwidth}
\centering
\includegraphics[height=3.4cm]{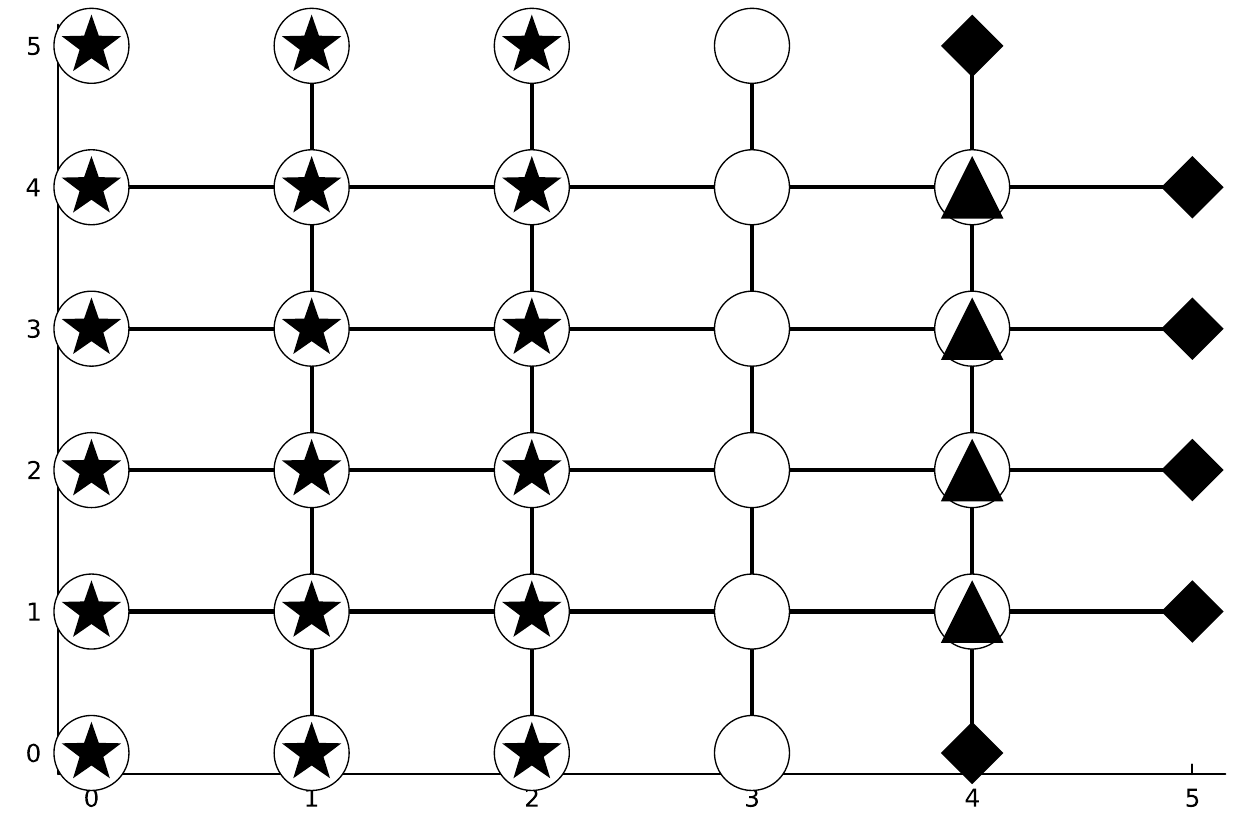}
\caption{$\mV_1$, $\mW_1$, $N_{\mG}(\mW_1)$, $\tilde{\mW}_1$, and $\hat{\mW}_1$}
\label{figset.b}
\end{subfigure}
\caption{Illustration of $\mV_i$, $\mW_i$, $N_{\mG}(\mW_i)$, $\tilde{\mW}_i$, and $\hat{\mW}_i$. We consider a two-dimensional graph divided into two subgraphs with the overlap size $b=2$. 
In \cref{figset.a}, star markers represent $\mV_1$ and hexagon markers represent $\mV_2$. In \cref{figset.b}, star markers represent $\mV_1$, circle markers represent $\mW_1$, diamond markers represent $N_{\mG}(\mW_1)$, and triangle markers represent $\tilde{\mW}_1$.~{\red The~union of the diamond- and triangle-marked nodes represents $\hat{\mW}_1=N_{\mG}(\mW_1)\cup\tilde{\mW}_1$.}}\label{figset}
\end{figure}

\vskip5pt

\paragraph{Step 3: stepsize selection with exact augmented Lagrangian}

After obtaining the approximate Newton direction $(\tD \bx^{\tau}, \tD \bl^{\tau})$, we then select a proper stepsize $\at$ by performing backtracking line search on a differentiable exact augmented Lagrangian merit function. Particularly, we update the primal-dual iterate as
\begin{equation}\label{eq32}
\begin{bmatrix}
\bx^{\tau+1}\\ \bl^{\tau+1}
\end{bmatrix} = \begin{bmatrix}
\bxt\\ \blt
\end{bmatrix} + \at\begin{bmatrix}
\tD \bxt\\ \tD \blt
\end{bmatrix},
\end{equation}
where $\at$ is selected to satisfy the Armijo condition for a parameter $\beta\in(0,0.5)$ as
\begin{equation}\label{eq34}
\mLe^{\tau+1} \leq \mLe^{\tau} + \beta \at\, \begin{bmatrix}
\nbx\mLe^{\tau}\\ \nbl\mLe^{\tau}
\end{bmatrix}^T\begin{bmatrix}
\tD \bxt\\ \tD \blt
\end{bmatrix},
\end{equation}
and $\mLe^{\tau} = \mLe(\bxt, \blt)$ (similar for $\mLe^{\tau+1}$, $\nabla \mLe^{\tau}$) is the exact augmented Lagrangian, defined as
\begin{equation}\label{eq33}
\mLe(\bx, \bl) = \mL(\bx, \bl) + \frac{\eta_1}{2}\|\nabla_{\bl}\mL(\bx, \bl)\|^2 + \frac{\eta_2}{2}\|\nabla_{\bx}\mL(\bx, \bl)\|^2
\end{equation}
with the penalty parameters $\eta = (\eta_1, \eta_2)$. Here, the first penalty biases the feasibility residual while the second penalty biases the optimality residual.
{\red
The quantities needed in the line search can also be evaluated in parallel without assembling global objects. In particular, $\mLe^{\tau}$ and $\mLe^{\tau+1}$ can be expressed as sums of local contributions over the disjoint partition $\{\mV_{\ell}\}_{\ell\in[M]}$, and the directional derivative term in \eqref{eq34} can be computed in the same way by summing local inner products restricted to the corresponding node sets. Therefore, each subdomain computes its local contributions independently, and the global quantities in \eqref{eq34} are obtained by aggregating these local values across subdomains.
}

Although there are many alternative non-differentiable $\ell_1/\ell_2$ merit functions \cite{Nocedal2006Numerical}, recent research has shown that the exact augmented Lagrangian \eqref{eq33} may offer certain benefits when applied to optimal control problems \cite{Na2023Fast, Na2021Global}, as well as to more general graph-structured problems as demonstrated in this paper. Roughly speaking, we desire a merit function that involves the dual variables to endure an approximation error in the dual component. Furthermore, we also desire a differentiable merit function that can overcome the Maratos effect and achieve fast local convergence (by accepting the unit stepsize) without the need of advanced enhancements of the method (e.g., second-order correction), as required by non-differentiable merit functions. The enhancements like second-order correction have to be decentralized in the parallel setting, significantly complicating the algorithm applicability. The detailed discussion regarding the choice of the merit function is deferred to \cref{thm51} and \cref{thm61}.

\begin{algorithm}
\caption{A Fast Overlapping Graph Decomposition Method}\label{algo2}
\begin{algorithmic}[1]
\State{\textbf{Input:} initialization $(\bx^0, \boldsymbol{\lambda}^0)$, number of subproblems $M$, scalars $\mu,\eta> 0$, $\beta\in(0,0.5)$;}
\For{$\tau=0, 1, 2, \cdots$}
\State Compute $\hHt, \Gt,$ and $\nabla \mLt$;
\For{$\ell=1, 2, \cdots, M$ \textbf{(in parallel)}}
\State Solve $\mS\mP_{\ell}^{\tau, \mu}$ in (\ref{eq37}) to obtain the solution $((\bto_{\ell}^\tau)^{\boldsymbol{*}}, (\btze_{\ell}^\tau)^{\boldsymbol{*}})$;
\EndFor
\State Let $(\tD \bx^{\tau}, \tD \bl^{\tau}) = \mC(\{((\bto_{\ell}^\tau)^{\boldsymbol{*}}, (\btze_{\ell}^\tau)^{\boldsymbol{*}} )\}_{\ell=1}^M)$;
\State Update the iterate as \eqref{eq32} by selecting $\at$ via \eqref{eq34};
\EndFor
\end{algorithmic}
\end{algorithm}

\section{Error Analysis of OGD}\label{sec3}

We have introduced how to approximately solve the Newton system in each SQP iteration by applying overlapping graph decomposition. Then, a natural question is how close the approximate direction is to the exact direction. In this section, we answer this question by providing an upper bound for the approximation error $(\tD \bxt-\Delta \bxt, \tD \blt-\Delta \blt)$. We will finally show in \cref{thm42} that, under standard mild regularity conditions, the error decays exponentially fast in terms of the overlap size $b$.

To study the approximation error, we leverage the sensitivity analysis of graph-structured linear-quadratic programs. In particular, as shown in \cref{lm41i}, the truncated exact direction on $\mW_{\ell}$, i.e., $\mD_{\ell}(\Delta \bxt, \Delta \blt)=\mD_{\ell}((\bo^\tau)^{\boldsymbol{*}}, (\bze^\tau)^{\boldsymbol{*}})$, is the solution of a linear-quadratic program that is closely related to the full problem \eqref{eq37}. 
{\red
Intuitively, once we restrict the full problem to the node set $\mW_{\ell}$, the influence of variables outside $\mW_{\ell}$ enters only through the boundary of $\mW_{\ell}$. This motivates introducing the boundary parameter $\bd_{\ell}$ and the parameterized subproblem $\mS\mP_{\ell}^{\tau,\mu}(\bd_{\ell})$ in \eqref{eq37_sec3}. For a particular choice $\bd_{\ell} = \bd_{\ell}^{\star}$ (given in \cref{lm41i}), the solution of $\mS\mP_{\ell}^{\tau,\mu}(\bd_{\ell}^{\star})$ coincides with the truncated exact direction $\mD_{\ell}(\Delta \bxt, \Delta \blt)$. On the other hand, the setup of $\bd_{\ell} = \bd_{\ell}^{\star}$ relies on knowledge of the true solution $((\bo^\tau)^{\boldsymbol{*}}, (\bze^\tau)^{\boldsymbol{*}})$ itself, which is unavailable to us. Our actual subproblem \eqref{eq37_sec3} instead corresponds to the choice of $\bd_{\ell}=\bnz$.~Therefore, the approximation error can be viewed as the effect of a boundary perturbation. By sensitivity analysis together with the exponential decay of sensitivity \cite{Na2020Exponential, Shin2022Exponential}, this boundary effect decreases rapidly as we move away from the boundary, which explains why increasing the overlap size improves the accuracy on the interior set $\mV_{\ell}$. 
}

Motivated by this discussion, we now introduce a parameterized version of Problem \eqref{eq37}, denoted as $\mS\mP_{\ell}^{\tau,\mu}(\bd_{\ell})$:
\begin{subequations}\label{eq37_sec3}
\begin{align}
\min_{\bto_{\ell}}\;\; & \frac{1}{2}\begin{bmatrix}
\bto_{\ell} \\ \bbo_{\ell}
\end{bmatrix}^T
\begin{bmatrix}
\hHt[\mW_{\ell}][\mW_{\ell}] & \hHt[\mW_{\ell}][\bar{\mW}_{\ell}]\\
\hHt[\bar{\mW}_{\ell}][\mW_{\ell}] & \hHt[\bar{\mW}_{\ell}][\bar{\mW}_{\ell}]
\end{bmatrix}\begin{bmatrix}
\bto_{\ell} \\ \bbo_{\ell}
\end{bmatrix} + \bto_{\ell}^T \nabla_{\bx} \mLt[\mW_{\ell}] \\
& \hskip0.5cm + \bbze_{\ell}^TG^\tau[\hat{\mW}_{\ell}][\mW_{\ell}]\bto_{\ell} + \frac{\mu}{2}\nbr{\bct[\hat{\mW}_{\ell}] + G^\tau[\hat{\mW}_{\ell}][\mW_{\ell}] \bto_{\ell}  + G^\tau[\hat{\mW}_{\ell}][\bar{\mW}_{\ell}] \bbo_{\ell}}^2, \notag\\
\text{s.t.}\;\; & \bct[\mW_{\ell}\backslash \tmW_{\ell}]  + \Gt[\mW_{\ell}\backslash \tmW_{\ell}][\mW_{\ell}] \bto_{\ell} = \bnz,\quad (\btze_{\ell}),
\end{align}
\end{subequations}
where $\tmW_{\ell}$, $\hat{\mW}_{\ell}$ are defined in \eqref{eq37}, $\bar{\mW}_{\ell} \coloneqq N_{\mG}(\mW_{\ell})\cup N_{\mG}\left(N_{\mG}[\mW_{\ell}]\right)$ is the set of external boundary nodes of $\mW_\ell$ within depth two, and $\bd_{\ell} = (\bbo_{\ell}, \bbze_{\ell}) = (\{\bbo_{\ell,j}\}_{j\in \bar{\mW}_{\ell}}, \{\bbze_{\ell,j}\}_{j\in \hat{\mW}_{\ell}})$ are primal-dual boundary parameters specified to the problem. Clearly, the solution of $\mS\mP_{\ell}^{\mu}(\bd_{\ell})$ depends on the parameters $\bd_{\ell}$, which we denote as $(\towt_{\ell}(\bd_{\ell}), \tzewt_{\ell}(\bd_{\ell}))$ (the iteration index $\tau$ is omitted). It is worth noting that when $\bd_{\ell}=\bnz$, $\mS\mP_{\ell}^{\mu} (\bnz)$ reduces to $\mS\mP_{\ell}^{\mu}$ in \eqref{eq37}.

We then introduce some preliminary assumptions.

\begin{assumption}[\textbf{Hessian modification}]\label{as41}
For any $\tau \ge 0$, we assume the modified Hessian $\hH^\tau$ satisfies $\bo^T\hHt\bo \geq \gamma_{H}\|\bo\|^2$ for any $\bo\in\{\bo: G^\tau\bo = \bnz\}$ for a constant $\gamma_{H}\in(0,1]$.
\end{assumption}

\begin{assumption}[\textbf{Linear independence constraint qualification}]\label{as42}
For any $\tau\geq 0$, we assume $\Gt(\Gt)^T \succeq \gamma_G I$ for a constant $\gamma_G\in(0,1]$.
\end{assumption}

\begin{assumption}[\textbf{Uniform boundedness}]\label{as44}
For any $\tau\geq 0$ and any nodes $i, j\in\mV$, we assume $\max\{\|\hHt_{i,j}\|, \|H^\tau_{i,j}\|, \|G^\tau_{i,j}\|\} \leq \Upsilon$ for a constant $\Upsilon \geq 1$, where $\hHt_{i,j}$, $H^\tau_{i,j}$, and $G^\tau_{i,j}$ are the $(i,j)$-block of $\hHt$, $H^\tau$, and $G^\tau$, respectively.
\end{assumption}

\begin{assumption}[\textbf{Polynomial graph}]\label{as45}
There are two polynomial functions $\texttt{poly}_{\mW}(\cdot)$ and $\texttt{poly}_{\mG}(\cdot)$ such that
\begin{subequations}
\begin{alignat}{2}
d_{\mG}(i,j) & \leq  \texttt{poly}_{\mW}(b),\quad && \forall i,j \in \mW_{\ell},\; \ell\in [M], \label{eqn:graph-1}\\
|\{j\in \mV:d_{\mG}(i,j) \leq d\}| &\le \texttt{poly}_{\mG}(d),\quad && \forall i\in \mV,\label{eqn:graph-2}
\end{alignat}
\end{subequations}
where $b$ is the overlap size (cf. \Cref{sec2}) and $d$ is any integer.
\end{assumption}

Assumptions \ref{as41}, \ref{as42}, and \ref{as44} are standard assumptions in SQP literature \cite{Bertsekas1996Constrained, Gould2000SQP} and are crucial for sensitivity analysis of graph-structured problems \cite{Shin2022Exponential}. In particular, Assumptions \ref{as41} and \ref{as42} ensure that the quadratic program \eqref{eq36} has a unique global solution \cite[Lemma 16.1]{Nocedal2006Numerical}, and Assumption \ref{as44} implies a boundedness condition on the full matrices $(\hHt\hskip-0.45pt, H^\tau\hskip-0.45pt, G^\tau)$. We note that the conditions on the upper and lower bound constants, $\Upsilon\geq 1\geq \max(\gamma_H, \gamma_G)$, are inessential (similar for other constants later), which are only used to simplify the presentation. Without such conditions, our results still hold by replacing $\Upsilon, \gamma_H, \gamma_G$ with $\max(\Upsilon,1)$, $\min(\gamma_H,1)$, $\min(\gamma_G,1)$, respectively.

To satisfy Assumption \ref{as41}, we employ a structure-preserving modification procedure for the Lagrangian Hessian $H^\tau = \nabla_{\bx}^2\mL^\tau$. One simple example is the Levenberg-type regularization \cite{Dunn1989Efficient}: $\hHt = H^\tau + (\gamma_H + \|H^\tau\|) I$. Assumption \ref{as42} is equivalent to assuming $G^\tau$ to have full row rank with the least singular value (uniformly) bounded away from zero. Furthermore, we should note that Assumption \ref{as44} actually only needs to be imposed on blocks with adjacent nodes, as both the Hessian $H^\tau$ and the Jacobian $G^\tau$ from Problem \eqref{eq21} exhibit a nice banded structure (so does $\hHt$):
\begin{equation*}
H^\tau_{i,j} = \bnz \;\;\;\; \text{if }\;\; d_\mG(i,j)>2, \quad \quad\quad G^\tau_{i,j} = \bnz \;\;\;\; \text{if }\;\; d_\mG(i,j)>1.
\end{equation*}
With the above banded structure, we know that the block-wise boundedness implies the full matrix boundedness (with a bound independent of the graph size). Specifically, we apply Assumptions \ref{as44}, \ref{as45} and \cite[Lemma 5.10]{Shin2022Exponential}, and have for any $\tau\geq 0$,
\begin{equation}\label{as44-e}
\max\{\|\hHt\|, \|H^\tau\|, \|G^\tau\|\} \leq \Upsilon\cdot \polyG(1)^2 \eqqcolon \hU.
\end{equation}

Assumption \ref{as45} imposes mild conditions on the graph topology. In particular, \eqref{eqn:graph-1} ensures that the diameter of each overlapping subdomain $\mW_\ell$ is bounded by a polynomial of the overlap size, and \eqref{eqn:graph-2} ensures that the graph $\mG$ can only grow polynomially. Combining \eqref{eqn:graph-1} and \eqref{eqn:graph-2}, we know for any $\ell\in[M]$,
\begin{equation}\label{nequ:8}
|\mW_{\ell}| \leq \texttt{poly}_{\mG}(\texttt{poly}_{\mW}(b)) \eqqcolon \texttt{poly}_{\mG\circ\mW}(b). 
\end{equation}
We mention that \eqref{eqn:graph-2} excludes tree graphs, which grow exponentially as seen in multi-stage stochastic programs. Decomposing (scenario) tree structures that can break their curse of dimensionality is a challenging topic. The difficulty lies in the fact that the error induced by decomposition perturbations decays only linearly along the tree branches, which cannot govern the speed of their growth rate. A similar bottleneck has been observed in \cite{Shin2023Optimal} and we defer the inclusion of tree structures to future work.
{\red That being said, \eqref{eqn:graph-1}--\eqref{eqn:graph-2} are satisfied by a broad range of practically important graphs, including: (1) chain or path graphs arising in sequential decision-making and dynamic programs; (2) grid graphs (and other bounded-degree meshes) commonly encountered in discretizations of PDE-constrained problems (e.g., Figure \ref{fig31}); and (3) sparse power-grid networks, where buses are modeled as nodes and transmission lines as edges, so that node degrees are modest and neighborhood sizes grow at most~\mbox{polynomially}~with distance.}
Also, it is important to note that all our established bounds are~independent~of the graph size $|\mV|$ and the number of subproblems $M$.


Given Assumptions \ref{as41}--\ref{as45}, we are able to establish the existence and uniqueness of the solution to subproblem $\mS\mP_{\ell}^{\mu}(\bd_{\ell})$ in the following lemma, which holds for any iteration $\tau\geq 0$. (The iteration index $\tau$ is omitted here and there.)

\begin{lemma}\label{lm41}
If Assumptions \ref{as41}--\ref{as45} hold for the full Problem \eqref{eq36}, then for any $\ell\in[M]$ and any boundary variables $\bd_{\ell}$, the subproblem $\mS\mP_{\ell}^{\mu}(\bd_{\ell})$ with $\mu \geq \hat{\mu} \coloneqq4\hU^2/(\gamma_G\,\gRH)$ will also satisfy these assumptions. In particular,
\begin{enumerate}[label=(\alph*),topsep=0pt]
\setlength\itemsep{0.0em}
\item the Jacobian $G_\ell \coloneqq \Gt[\mW_{\ell}\backslash \tmW_{\ell}][\mW_{\ell}]$ satisfies $G_\ell G_\ell^T\succeq \gamma_G I$.
\item the Hessian $\hHm_\ell \coloneqq \hHt[\mW_\ell][\mW_\ell] + \mu\cdot (G^\tau[\hat{\mW}_{\ell}][\mW_{\ell}])^TG^\tau[\hat{\mW}_{\ell}][\mW_{\ell}]$ satisfies 
\begin{equation*}
\bto_{\ell}^T\hHm_\ell\bto_\ell \geq \gamma_{H}\|\bto_\ell\|^2/2, \quad \text{ for any } \quad \bto_{\ell}\in\{\bto_{\ell}: G_\ell \bto_{\ell} = \bnz\}.
\end{equation*}
\end{enumerate}
This implies that $\mS\mP_{\ell}^{\mu}(\bd_{\ell})$ has a unique global solution.
\end{lemma}

\begin{proof}
See \cref{alm41}.
\end{proof}

\cref{lm41} shows that the subproblems enjoy the same regularity conditions as those of the full problem. In fact, Assumptions \ref{as41}--\ref{as45} are critical since they imply an exponential decay structure for the KKT matrix inverse, which plays a key role when analyzing the approximation error of Newton directions. We note that \cite[Lemma 2]{Na2023Superconvergence} established the decay result for optimal control problems (i.e., linear graphs), while we study here a more general graph. We first introduce the KKT matrix of $\mS\mP_\ell^\mu(\bd_\ell)$ and its inverse partition as follows:
\begin{equation}\label{eq64}
K_{\ell}^\mu \coloneqq \begin{bmatrix}
\hHm_{\ell} & G_{\ell}^T\\
G_{\ell} & \bnz
\end{bmatrix}, \quad\quad \quad (K_\ell^\mu)^{-1}\coloneqq \begin{bmatrix}
(K_\ell^\mu)^{-1}_1 & \{(K_\ell^\mu)^{-1}_2\}^T\\
(K_\ell^\mu)^{-1}_2 & (K_\ell^\mu)^{-1}_3
\end{bmatrix}.
\end{equation}
For any nodes $i, j\in\mW_\ell$ and $k, v\in \mW_{\ell}\backslash \tmW_{\ell}$, we let $[(K_\ell^\mu)^{-1}_1]_{i,j}$, $[(K_\ell^\mu)^{-1}_2]_{k,j}$, and $[(K_\ell^\mu)^{-1}_3]_{k,v}$ denote the $(i,j)$, $(k,j)$, and $(k,v)$-blocks of the submatrices $(K_\ell^\mu)^{-1}_1$, $(K_\ell^\mu)^{-1}_2$, and $(K_\ell^\mu)^{-1}_3$, respectively. With the above partition, we present the structure of $(K_\ell^\mu)^{-1}$ in the next lemma.

\begin{lemma}[\textbf{Structure of the KKT inverse}]\label{lm61}
Suppose Assumptions \ref{as41}--\ref{as45} hold for the full Problem \eqref{eq36} and $\mu \geq \hat{\mu}$, then for any nodes $i, j\in\mW_\ell$ and $k, v\in \mW_{\ell}\backslash \tmW_{\ell}$, we have
\begin{equation}\label{eq66}
\small \|[(K_\ell^\mu)^{-1}_1]_{i,j}\| \le \Upsilon_\mu \rho_\mu^{d_{\mG}(i,j)}, \quad \;\; \|[(K_\ell^\mu)^{-1}_2]_{k,j}\| \le \Upsilon_\mu \rho_\mu^{d_{\mG}(k,j)},\quad\;\; \|[(K_\ell^\mu)^{-1}_3]_{k,v}\| \le \Upsilon_\mu \rho_\mu^{d_{\mG}(k,v)},
\end{equation}
where $\Upsilon_\mu = \frac{\Upsilon_{K}}{\gamma_{K}^2\rho_\mu^2}$ and $\rho_\mu = \left(\frac{\Upsilon_{K}^2 - \gamma_{K}^2}{\Upsilon_{K}^2 + \gamma_{K}^2}\right)^{\frac{1}{4}}\in(0, 1)$ with $\Upsilon_{K} = 1.5\mu\hU^2$ and $\gamma_{K} = \frac{\gamma_{H}^3\gamma_G^2}{75\Upsilon_{K}^6}$.

\end{lemma}

\begin{proof}
	
By \eqref{as44-e}, \cref{lm41}, and the facts that $\mu\geq 4$ and $\hU\geq 1$, we know
\begin{equation}\label{nequ:3}
\|K_\ell^\mu\| \leq \|\hH_\ell^\mu\| + \|G_\ell\| \leq 2\hU + \mu\hU^2 \leq 1.5\mu\hU^2 \eqqcolon \Upsilon_{K}.
\end{equation}
By \cite[Theorem A.3]{Shin2023Optimal}, we know the least singular value of $K_\ell^\mu$, denoted by $\sigma_{\min}(K_\ell^\mu)$, satisfies
\begin{equation}\label{nequ:2}
\sigma_{\min}(K_\ell^\mu)\geq \rbr{\frac{4}{\gamma_{H}} + \rbr{1+\frac{8\Upsilon_{K}}{\gamma_{H}} + \frac{16\Upsilon_{K}^2}{\gamma_{H}^2}}\frac{\Upsilon_{K}(1+ \nu\Upsilon_{K})}{\gamma_G} + \nu}^{-1}
\end{equation}
with $\nu = (8\Upsilon_{K}^2/\gamma_{H} + \gamma_{H} + 2\Upsilon_{K})/(2\gamma_G)$. Since $\gamma_{H}\leq 1\leq \Upsilon_{K}^2/6^2$ and $\Upsilon_{K} \leq \Upsilon_{K}^2/6$, we know $\nu \leq 4.2\Upsilon_{K}^2/(\gamma_{H}\gamma_G)$. Plugging it into \eqref{nequ:2} and using the facts that 
\begin{align*}
1+\frac{8\Upsilon_{K}}{\gamma_{H}} + \frac{16\Upsilon_{K}^2}{\gamma_{H}^2} & \leq \rbr{\frac{1}{36}+ \frac{4}{3} + 16} \frac{\Upsilon_{K}^2}{\gamma_{H}^2} \leq 17.4 \frac{\Upsilon_{K}^2}{\gamma_{H}^2},\\
\frac{\Upsilon_{K}(1+ \nu\Upsilon_{K})}{\gamma_G} & \leq \frac{\Upsilon_{K}}{\gamma_G} + \frac{4.2\Upsilon_{K}^4}{\gamma_{H}\gamma_G^2} \leq \rbr{\frac{1}{6^3} + 4.2}\frac{\Upsilon_{K}^4}{\gamma_{H}\gamma_G^2}\leq 4.21\frac{\Upsilon_{K}^4}{\gamma_{H}\gamma_G^2},
\end{align*}
we can obtain
\begin{multline}\label{nequ:4}
\sigma_{\min}(K_\ell^\mu) \geq \rbr{\frac{4}{\gamma_{H}} + 17.4\cdot4.21\frac{\Upsilon_{K}^6}{\gamma_{H}^3\gamma_G^2} + \frac{4.2\Upsilon_{K}^2}{\gamma_{H}\gamma_G}}^{-1}\\
\geq \cbr{\rbr{\frac{4}{6^6} + 17.4\cdot4.21 + \frac{4.2}{6^4}} \frac{\Upsilon_{K}^6}{\gamma_{H}^3\gamma_G^2} }^{-1} \geq \frac{\gamma_{H}^3\gamma_G^2}{75\Upsilon_{K}^6} \eqqcolon \gamma_{K}.
\end{multline}
With the bounds in \eqref{nequ:3} and \eqref{nequ:4}, the result is obtained from \cite[Theorem 3.6]{Shin2022Exponential} by noting that $K_\ell^\mu$ has a bandwidth $2$. We complete the proof.
\end{proof}

\Cref{lm61} suggests that for any nodes $i,j\in \mW_{\ell}$, the $(i,j)$-block of the KKT matrix inverse (of the subproblem) has an exponentially decaying magnitude in terms of the distance $d_{\mG}(i,j)$ on the graph.
We will rely on this property to study the difference between two solutions with different boundary variables, $(\towt_{\ell}(\bd_{\ell}), \tzewt_{\ell}(\bd_{\ell})) - (\towt_{\ell}(\bd_{\ell}'), \tzewt_{\ell}(\bd_{\ell}'))$. Before that, we show in the next lemma that if the boundary variables are set to the components of the exact Newton direction $(\Delta\bx, \Delta\bl)$, i.e., $\bd_\ell^{\boldsymbol{*}} \coloneqq (\{\Delta\bx_j\}_{j\in\bar{\mW}_\ell},  \{\Delta\bl_j\}_{j\in \hat{\mW}_\ell})$, then the subproblem $\mS\mP_\ell^\mu(\bd_\ell^{\boldsymbol{*}})$ will recover the truncated exact Newton direction $\mD_{\ell}(\Delta\bx, \Delta\bl)$ .

\begin{lemma}\label{lm41i}
Suppose Assumptions \ref{as41}--\ref{as45} hold for the full Problem \eqref{eq36} and $\mu\geq \hat{\mu}$. Let
$(\bos,\bzes) = (\Delta\bx, \Delta\bl)$ be the solution of \eqref{eq36}. Then, for any $\ell\in[M]$, $\mD_{\ell}(\bos, \bzes)$ is the solution of $\mS\mP_\ell^\mu(\bd_\ell^{\boldsymbol{*}})$ in \eqref{eq37_sec3} with $\bd_\ell^{\boldsymbol{*}} = (\bos[\bar{\mW}_\ell], \bzes[\hat{\mW}_\ell])$.

\end{lemma}

\begin{proof}

By \cref{lm41}, $\mS\mP_\ell^\mu(\bd_\ell^{\boldsymbol{*}})$ has a unique global solution. Thus, it suffices to show that $\mD_{\ell}(\bos, \bzes)$ satisfies the KKT conditions of $\mS\mP_\ell^\mu(\bd_\ell^{\boldsymbol{*}})$. Noting that \eqref{eq31} is the KKT conditions of Problem \eqref{eq36}, we can analogously obtain the KKT conditions of  \eqref{eq37_sec3} as follows:
\begin{equation}\label{eq24}
\hspace{1cm}\begin{cases}
\hH[\mW_{\ell}][\mW_{\ell}]\bto_{\ell} + \hH[\mW_{\ell}][\bar{\mW}_{\ell}] \bbo_{\ell} + (G[\hat{\mW}_{\ell}][\mW_{\ell}])^T \bbze_{\ell} + (G[\mW_{\ell}\backslash \tmW_{\ell}][\mW_{\ell}])^T \btze_{\ell}\\
\hskip1.2cm +\mu (G[\hat{\mW}_{\ell}][\mW_{\ell}])^T(\bc[\hat{\mW}_{\ell}] + G[\hat{\mW}_{\ell}][\mW_{\ell}] \bto_{\ell}  + G[\hat{\mW}_{\ell}][\bar{\mW}_{\ell}] \bbo_{\ell}) = -\nabla_{\bx} \mL[\mW_{\ell}], \\
G[\mW_{\ell}\backslash \tmW_{\ell}][\mW_{\ell}] \bto_{\ell} =- \bc[\mW_{\ell}\backslash \tmW_{\ell}],
\end{cases}
\end{equation}
where $\bbo_{\ell} = \bos[\bar{\mW}_\ell]$ and $\bbze_{\ell} = \bzes[\hat{\mW}_\ell]$. Since $\bc + G\bos = \bnz$, we have for the rows of $\mW_{\ell}\backslash \tmW_{\ell}$ that
\begin{align*}
\bnz & = \bc[\mW_{\ell}\backslash \tmW_{\ell}] + G[\mW_{\ell}\backslash \tmW_{\ell}][\mV]\bos \\
& = \bc[\mW_{\ell}\backslash \tmW_{\ell}] + G[\mW_{\ell}\backslash \tmW_{\ell}][\mW_{\ell}]\bos[\mW_{\ell}] + G[\mW_{\ell}\backslash \tmW_{\ell}][\mV\backslash\mW_{\ell}]\bos[\mV\backslash\mW_{\ell}]\\
& = \bc[\mW_{\ell}\backslash \tmW_{\ell}] + G[\mW_{\ell}\backslash \tmW_{\ell}][\mW_{\ell}]\bos[\mW_{\ell}],
\end{align*}
where the last equality is due to the facts that $N_\mG[\mW_{\ell}\backslash \tmW_{\ell}]\subseteq \mW_\ell$ and for any nodes $i,j\in\mV$, $G_{i,j} = \bnz$ if $j\notin N_{\mG}[i]$ (i.e., $d_\mG(i,j)>1$), so that $G[\mW_{\ell}\backslash \tmW_{\ell}][\mV\backslash\mW_{\ell}]=\bnz$. This verifies the second condition in \eqref{eq24}. Following the same reasoning, we consider the rows of $\hat{\mW}_\ell$ and have\;\;\;\;
\begin{align}\label{nequ:5}
\bnz & = \bc[\hat{\mW}_{\ell}] + G[\hat{\mW}_{\ell}][\mV]\bos = \bc[\hat{\mW}_{\ell}] + G[\hat{\mW}_{\ell}][\mW_\ell]\bos[\mW_\ell] + G[\hat{\mW}_{\ell}][\bar{\mW}_\ell]\bos[\bar{\mW}_\ell]  \nonumber\\
&\quad + G[\hat{\mW}_{\ell}][\mV\backslash(\mW_{\ell}\cup \bar{\mW}_\ell)]\bos[\mV\backslash(\mW_{\ell}\cup \bar{\mW}_\ell)]\nonumber\\
& = \bc[\hat{\mW}_{\ell}] + G[\hat{\mW}_{\ell}][\mW_\ell]\bos[\mW_\ell] + G[\hat{\mW}_{\ell}][\bar{\mW}_\ell]\bos[\bar{\mW}_\ell],
\end{align}
where the last equality is due to $G[\hat{\mW}_{\ell}][\mV\backslash(\mW_{\ell}\cup \bar{\mW}_\ell)] = \bnz$. Furthermore, we consider the rows of $\mW_\ell$ of the equality $\hH\bos + G^T\bzes = -\nabla_{\bx}\mL$, and have
\begin{align}\label{nequ:6}
-\nabla_{\bx}\mL[\mW_{\ell}] & = \hH[\mW_\ell][\mV]\bos + (G[\mV][\mW_\ell])^T\bzes \nonumber\\
& = \hH[\mW_\ell][\mW_\ell]\bos[\mW_{\ell}] + \hH[\mW_\ell][\bar{\mW}_\ell]\bos[\bar{\mW}_\ell] + (G[\hat{\mW}_\ell][\mW_{\ell}])^T\bzes[\hat{\mW}_\ell] \nonumber\\
& \quad + (G[\mW_{\ell}\backslash\tilde{\mW}_\ell][\mW_{\ell}])^T\bzes[\mW_{\ell}\backslash\tilde{\mW}_\ell].
\end{align}
The second equality here is due to $\hH[\mW_{\ell}][\mV\backslash(\mW_{\ell}\cup\bar{\mW}_\ell)] = \bnz$ and $G[\mV\backslash(\hat{\mW}_{\ell}\cup (\mW_{\ell}\backslash\tilde{\mW}_\ell))][\mW_{\ell}]=\bnz$. Combining \eqref{nequ:5} and \eqref{nequ:6}, we see that the first condition in \eqref{eq24} holds. Thus, we know $\mD_{\ell}(\bos, \bzes) = (\bos[\mW_{\ell}], \bzes[\mW_{\ell}\backslash\tilde{\mW}_{\ell}])$ satisfies \eqref{eq24} and we complete the proof.
\end{proof}

We now provide an upper bound of the solution difference with different boundary variables, $\|(\towt_{\ell}(\bd_{\ell}), \tzewt_{\ell}(\bd_{\ell})) - (\towt_{\ell}(\bd_{\ell}'), \tzewt_{\ell}(\bd_{\ell}'))\|$. Recall (see above \cref{df22}) that for any node $k\in \mW_{\ell}$, $\towt_{\ell,k}$ denotes the variable associated with this node (similar for $\tzewt_{\ell,k}$).

\begin{lemma}\label{lem41}
Suppose Assumptions \ref{as41}--\ref{as45} hold for the full Problem \eqref{eq36} and $\mu\geq \hat{\mu}$. Then, for any $\ell\in[M]$ and any nodes $k\in \mW_{\ell}$ and $j\in\mW_{\ell}\backslash\tilde{\mW}_\ell$, we have
\begin{subequations}\label{eq42}
\begin{align}
\|\towt_{\ell,k}(\bd_{\ell}) - \towt_{\ell,k}(\bd_{\ell}')\| & \leq \Upsilon_\mu\Upsilon_{K}\cdot \{\texttt{poly}_{\mG\circ\mW}(b)\}^{1/2}\cdot\rho_\mu^{d_\mG(k, N_{\mG}(\mW_{\ell}) )- 2 }\|\bd_{\ell} - \bd_{\ell}'\|, \label{eq42a}\\
\|\tzewt_{\ell,j}(\bd_{\ell}) - \tzewt_{\ell,j}(\bd_{\ell}')\| & \leq \Upsilon_\mu\Upsilon_{K}\cdot \{\texttt{poly}_{\mG\circ\mW}(b)\}^{1/2}\cdot\rho_\mu^{d_\mG(j, N_{\mG}(\mW_{\ell}) )- 2 }\|\bd_{\ell} - \bd_{\ell}'\|, \label{eq42b}
\end{align}
\end{subequations}
where $\Upsilon_\mu, \Upsilon_{K}, \rho_\mu$ are defined in \cref{lm61}, and $\texttt{poly}_{\mG\circ\mW}(b)$ is defined in \eqref{nequ:8}.

\end{lemma}

\begin{proof}

The main idea of the proof is utilizing the exponential decay structure of the KKT matrix inverse in \cref{lm61}. First, we note that the KKT conditions of \eqref{eq37_sec3} are given by
\begin{equation}\label{nequ:11}
\begin{bmatrix}
\hHm_{\ell} & G_{\ell}^T\\
G_{\ell} & \bnz\\
\end{bmatrix} \begin{bmatrix}
\towt_{\ell}(\bd_{\ell})\\
\tzewt_{\ell}(\bd_{\ell})
\end{bmatrix} = \begin{bmatrix}
-F_{\ell}\bd_{\ell}-\mu (G[\hat{\mW}_{\ell}][\mW_{\ell}])^T\bc[\hat{\mW}_{\ell}]-\nabla_{\bx} \mL[\mW_{\ell}]\\
-\bc[\mW_{\ell}\backslash \tmW_{\ell}]
\end{bmatrix},
\end{equation}
where $F_{\ell}$ is given as follows
\begin{equation*}
F_{\ell} = \begin{pmatrix}
\hH[\mW_{\ell}][\bar{\mW}_{\ell}] + \mu (G[\hat{\mW}_{\ell}][\mW_{\ell}])^T G[\hat{\mW}_{\ell}][\bar{\mW}_{\ell}] & (G[\hat{\mW}_{\ell}][\mW_{\ell}])^T
\end{pmatrix}.
\end{equation*}
Therefore, we obtain
\begin{equation*}
\begin{bmatrix}
\towt_{\ell}(\bd_{\ell}) - \towt_{\ell}(\bd_{\ell}')\\
\tzewt_{\ell}(\bd_{\ell}) - \tzewt_{\ell}(\bd_{\ell}')
\end{bmatrix} = \begin{bmatrix}
(K_\ell^\mu)^{-1}_1 & \{(K_\ell^\mu)^{-1}_2\}^T\\
(K_\ell^\mu)^{-1}_2 & (K_\ell^\mu)^{-1}_3
\end{bmatrix} \begin{bmatrix}
-F_{\ell} (\bd_{\ell}-\bd_{\ell}') \\
\bnz
\end{bmatrix}.
\end{equation*}
With the above equality, we define $\cmW_{\ell}\coloneqq (N_{\mG}(\tmW_{\ell})\cap \mW_{\ell})\cup \tmW_{\ell}$ to be the set of internal boundary nodes of $\mW_{\ell}$ within depth two. Then, for any $k\in \mW_{\ell}$, we apply \cref{lm61} and have
\begin{align}\label{nequ:7}
\| \towt_{\ell,k}(\bd_{\ell}) & - \towt_{\ell,k}(\bd_{\ell}') \|  = \big\| \sum_{j\in \mW_\ell} [(K_\ell^\mu)_1^{-1}]_{k,j}[F_\ell(\bd_\ell - \bd_\ell')]_j\big\| \nonumber\\
& = \big\| \sum_{j\in \cmW_\ell} [(K_\ell^\mu)_1^{-1}]_{k,j}[F_\ell(\bd_\ell - \bd_\ell')]_j\big\| \leq \sum_{j\in \cmW_\ell} \Upsilon_\mu\rho_\mu^{d_\mG(k,j)}\cdot\|[F_\ell(\bd_\ell - \bd_\ell')]_j\| \nonumber\\
& \leq \Upsilon_\mu\sqrt{|\cmW_{\ell}|}\bigg\{\sum_{j\in \cmW_\ell}\rho_\mu^{2d_\mG(k,j)}\cdot \|[F_\ell(\bd_\ell - \bd_\ell')]_j\|^2\bigg\}^{1/2} \nonumber\\
& \leq \Upsilon_\mu\sqrt{|\cmW_{\ell}|}\bigg\{\rho_\mu^{2(d_\mG(k,N_\mG(\mW_{\ell}))-2)} \sum_{j\in \cmW_\ell} \|[F_\ell(\bd_\ell - \bd_\ell')]_j\|^2\bigg\}^{1/2} \nonumber\\
& =\Upsilon_\mu\sqrt{|\cmW_{\ell}|} \rho_\mu^{d_\mG(k,N_\mG(\mW_{\ell})) -2}\cdot\|F_\ell(\bd_\ell - \bd_\ell')\|  \nonumber\\
& \leq \Upsilon_\mu\sqrt{|\cmW_{\ell}|} \rho_\mu^{d_\mG(k,N_\mG(\mW_{\ell})) -2 } \cdot(2\hU+\mu\hU^2)\|(\bd_\ell - \bd_\ell')\|,
\end{align}
where the second equality is due to the facts that $\hH[j][\bar{\mW}_\ell] = \bnz$ and $G[\hat{\mW}_\ell][j] = \bnz$ for any $j\in \mW_{\ell}\backslash\cmW_{\ell}$, so that $[F_\ell(\bd_\ell - \bd_\ell')]_j = \bnz$ for $j\in \mW_{\ell}\backslash\cmW_{\ell}$; the fourth inequality is by Cauchy-Schwarz inequality; the fifth inequality is due to the fact that $d_\mG(k,N_\mG(\mW_{\ell})) \leq d_\mG(k, j) + 2$ for any $j\in \cmW_{\ell}$; the last inequality is due to \eqref{as44-e}. Noting that $\mu\geq 4$ and $|\cmW_{\ell}|\leq \texttt{poly}_{\mG\circ\mW}(b)$ by \eqref{nequ:8}, we complete the proof of \eqref{eq42a}. For \eqref{eq42b}, we follow the derivation of \eqref{nequ:7} by replacing $(K_\ell^\mu)_1^{-1}$ with $(K_\ell^\mu)_2^{-1}$, and obtain the same result. This completes the proof.$\quad\;\;$
\end{proof}

We now establish the approximation error of the Newton direction computed by overlapping graph decomposition. We note that the approximate direction $(\tD \bx, \tD \bl)$ is composed by $(\towt_{\ell}(\bnz), \tzewt_{\ell}(\bnz))$, while the exact direction $(\Delta \bx, \Delta \bl)$ is composed by $(\towt_{\ell}(\bds_\ell), \tzewt_{\ell}(\bds_\ell))$ (refer to \cref{lm41i}). We remind the reader that all our results hold for any iteration $\tau\geq 0$.

\begin{theorem}[\textbf{Approximation error of OGD}]\label{thm42} 
Suppose Assumptions \ref{as41}--\ref{as45} hold for the full Problem \eqref{eq36} and $\mu\geq \hat{\mu}$. Then, the approximate direction $(\tD \bx, \tD \bl)$ computed by OGD with the overlap size $b$ satisfies
\begin{equation}\label{eq44}
\|(\tD\bx - \Delta \bx,\; \tD\bl - \Delta \bl)\|\le \Upsilon_{\mu,b}\cdot \rho_\mu^{b}\|(\Delta \bx, \;\Delta \bl)\|,
\end{equation}
where $\Upsilon_\mu, \Upsilon_{K}, \rho_\mu$ are defined in \cref{lm61} and 
\begin{equation*}
\Upsilon_{\mu,b}\coloneqq \frac{\sqrt{2}\Upsilon_\mu\Upsilon_{K}}{\rho_\mu\sqrt{1-\rho_\mu^2}}\cdot\texttt{poly}_{\mG\circ\mW}(b)\{\texttt{poly}_{\mG}(b+2)\}^{1/2}.
\end{equation*}
\end{theorem}

\begin{proof}

By \cref{df22}, \cref{lm41i}, and \cref{lem41}, we know that for any node $k\in \mV_\ell$,
\begin{equation*}
\|\tD \bx_k - \Delta \bx_k\|^2 = \| \towt_{\ell,k}(\bnz) - \towt_{\ell,k}(\bds_\ell) \|^2 \leq \Upsilon_1 \rho_\mu^{2d_\mG(k, N_\mG(\mW_\ell))-4}\|\bds_\ell\|^2,
\end{equation*}
where $\Upsilon_1 \coloneqq \Upsilon_\mu^2\Upsilon_{K}^2\texttt{poly}_{\mG\circ\mW}(b)$, and the above bound holds for $\|\tD \bl_k - \Delta \bl_k\|^2$ as well. With the above bound, we have
\begin{align*}
\| \tD \bx - \Delta \bx\|^2 & = \sum_{\ell=1}^{M} \sum_{k\in\mV_{\ell}} \|\tD \bx_k - \Delta \bx_k\|^2\le\Upsilon_1\sum_{\ell=1}^{M} \sum_{k\in\mV_{\ell}} \rho_\mu^{2d_\mG(k, N_\mG(\mW_\ell))-4}\|\bds_\ell\|^2\\
& = \Upsilon_1\sum_{\ell=1}^{M}\|\bds_\ell\|^2\sum_{d=b+1}^{\infty}|\{k\in\mV_\ell: d_\mG(k, N_\mG(\mW_\ell))=d\}|
\cdot\rho_\mu^{2d-4}\\
& \leq \Upsilon_1\sum_{\ell=1}^{M}\|\bds_\ell\|^2 |\mV_\ell| \sum_{d=b+1}^{\infty}\rho_\mu^{2d-4} \stackrel{\eqref{nequ:8}}{\leq}\Upsilon_1\texttt{poly}_{\mG\circ\mW}(b)\frac{\rho_\mu^{2b-2}}{1-\rho_\mu^2}\sum_{\ell=1}^{M}\|\bds_\ell\|^2.
\end{align*}
We note that the boundary variables of two subproblems, $\bds_{\ell_1}$ and $\bds_{\ell_2}$ with $\ell_1\neq \ell_2$, may have overlaps so that they share the same variables of $(\Delta\bx, \Delta\bl)$. Below we give an upper bound for the number of subproblems (i.e., subgraphs) that can share one boundary node. For any $\ell\in[M]$, we know $\bds_\ell$ is associated with the boundary nodes $\bar{\mW}_{\ell}\cup\hat{\mW}_{\ell}$. Thus, for any node $k\in \bar{\mW}_{\ell}\cup\hat{\mW}_{\ell}$, if it is the boundary of the $\ell'$-th subgraph, then $d_\mG(k, \mV_{\ell'}) \leq b+2$. Since $\{\mV_\ell\}_{\ell\in[M]}$ are mutually disjoint and by Assumption \ref{as45}, $|\{j\in\mV: d(k,j)\leq b+2\}| = \texttt{poly}_\mG(b+2)$, one boundary node is shared by at most $\texttt{poly}_\mG(b+2)$ subgraphs. With this fact, we have \begin{equation*}
\sum_{\ell=1}^{M}\|\bds_{\ell}\|^2 \leq \sum_{j\in \mV} \polyG(b+2)\|(\Delta \bx_j, \Delta\bl_j)\|^2 = \polyG(b+2)\|(\Delta \bx,\Delta \bl)\|^2.
\end{equation*}
Combining the above two displays, we obtain \vskip-0.6cm
\begin{equation*}
\| \tD \bx - \Delta \bx\|^2 \leq \Upsilon_1\texttt{poly}_{\mG\circ\mW}(b)\cdot\texttt{poly}_{\mG}(b+2)\frac{\rho_\mu^{2b-2}}{1-\rho_\mu^2}\|(\Delta \bx,\Delta \bl)\|^2.
\end{equation*}
The result of $\| \tD \bl - \Delta \bl\|^2 $ follows the same derivations above. This completes the proof.\;\;
\end{proof}

\cref{thm42} indicates that the approximate Newton direction $(\tD\bx, \tD\bl)$ approaches the exact Newton direction $(\Delta\bx, \Delta\bl)$ exponentially fast with the increasing overlap size $b$. We note that the scaling factor $\Upsilon_{\mu,b}$ grows only polynomially in $b$; thus, the approximation error is eventually dominated by the exponential rate in $b$ and can be arbitrarily small as long as $b$ is large enough. It is worth mentioning that the constant $\Upsilon_{\mu,b}$ is conservative; in our experiments (cf. \cref{sec6}), we observe that the method may still work even for small $b$ (e.g., $1$ or $2$), although larger $b$ can improve the convergence rate.

\section{Global Convergence of FOGD}\label{sec4}

In this section, we prove the global convergence of FOGD. A key step is to show that the approximate Newton direction is a descent direction of the exact augmented Lagrangian merit function \eqref{eq33}, so that FOGD can keep decreasing the merit function over iterates. To achieve this, we rely on the approximation error established in \cref{thm42}. For simplicity of the presentation, let us define $\delta_{\mu,b} \coloneqq\Upsilon_{\mu,b}\rho_\mu^b$, and we always suppose a fixed $\mu\geq \hat{\mu}$ (cf. \cref{lm41}) without repeating it.

We need the following assumption in the global convergence analysis, which is commonly used in the SQP literature \cite[Proposition 4.15]{Bertsekas1996Constrained}. It enhances Assumption \ref{as44} by imposing it on a compact set.

\begin{assumption}[\textbf{Compactness}]\label{as51}
There exists a compact set $\mX\times\Lambda$ such that $(\bxt+\alpha \tD \bxt, \blt+\alpha \tD \blt)\in\mX\times\Lambda$, $\forall \tau \ge 0$, $\alpha\in[0,1]$; and we assume that $\{f_i, \bc_i\}_{i\in\mV}$ are thrice continuously differentiable in $\mX\times \Lambda$, and for a constant $\Upsilon\geq 1$,
\begin{equation}\label{eq52}
 \underset{\bx\in \mX,\, \bl\in\Lambda}{\sup}\|H_{i,j}(\bx, \bl)\|\leq \Upsilon \quad \text{ and }\quad  \underset{\bx\in \mX}{\sup}\|G_{i,j}(\bx)\| \le \Upsilon,\quad \forall\,i,j\in\mV.
\end{equation}
\end{assumption}

Here, we abuse the notation $\Upsilon$ in Assumption \ref{as44} to denote the upper bound of the Hessian and Jacobian matrices in the set $\mX\times \Lambda$. Again, $H_{i,j}$ and $G_{i,j}$ actually have banded structures with distance measured on the graph $\mG$; thus, \eqref{eq52} is only required for $d_\mG(i, j)\leq 2$. $\hskip1.5cm$

The next lemma is a direct result of Assumptions \ref{as41}--\ref{as45} and \ref{as51}.

\begin{lemma}\label{lm51}
Under Assumptions \ref{as45} and \ref{as51}, we have for the constant $\hU$ in \eqref{as44-e} that \vskip-0.4cm
\begin{equation}\label{eq53}
\underset{\bx\in \mX,\, \bl\in\Lambda}{\sup} \|H(\bx,\bl)\|\leq \hU \quad \text{ and }\quad  \underset{\bx\in\mX}{\sup}\|G(\bx)\| \le \hU.
\end{equation} \vskip-0.2cm
\noindent Furthermore, if Assumptions \ref{as41}--\ref{as45} hold, then the KKT inverse satisfies \vskip-0.4cm
\begin{equation}\label{eq54}
\Bigg\|\begin{bmatrix}
\hHt & (\Gt)^T\\
\Gt & \bnz 
\end{bmatrix}^{-1}\Bigg\|\le \frac{4\hU^2}{\gamma_{H}\gamma_G},\quad\forall \tau\ge 0.
\end{equation}
\end{lemma}

\begin{proof}
See \cref{alm51}.
\end{proof}

We are now able to introduce the sufficient conditions on the merit function parameters $\eta_1, \eta_2$ and the approximation error $\delta_{\mu,b}$ so that $(\tD \bx, \tD \bl)$ is a descent direction of the exact augmented Lagrangian \eqref{eq33}. We omit the iteration index $\tau$.

\begin{lemma}\label{thm51}
Suppose Assumptions \ref{as41}--\ref{as45} and \ref{as51} hold. If \vskip-0.4cm
\begin{equation}\label{eq55}
\eta_1\ge  \frac{9}{\eta_2\gG},\quad\quad\quad\eta_2\le\frac{\gH}{7\hU^2},\quad\quad\quad \delta_{\mu,b} \le \frac{\eta_2\gG}{17\eta_1\hU^2},
\end{equation}\vskip-0.2cm
\noindent then we have \vskip-0.4cm
\begin{equation*}
\begin{bmatrix}
\nbx \mLe\\
\nbl \mLe\\
\end{bmatrix}^T
\begin{bmatrix}
\tD \bx\\
\tD \bl\\
\end{bmatrix}\le -\frac{\eta_2}{8}\Bigg\|\begin{bmatrix}
\nbx \mL\\
\nbl \mL\\
\end{bmatrix}\Bigg\|^2 - \frac{\eta_2\gamma_G}{16}\nbr{\begin{bmatrix}
\Delta\bx\\
\Delta\bl
\end{bmatrix}}^2.
\end{equation*}
\end{lemma}

\begin{proof}

We note that
\begin{equation}\label{eq57}
\begin{bmatrix}
\nbx \mLe\\
\nbl \mLe\\
\end{bmatrix}^T\begin{bmatrix}
\tD \bx\\
\tD \bl\\
\end{bmatrix} = \begin{bmatrix}
\nbx \mLe\\
\nbl \mLe\\
\end{bmatrix}^T\begin{bmatrix}
\Delta \bx\\
\Delta \bl\\
\end{bmatrix} + \begin{bmatrix}
\nbx \mLe\\
\nbl \mLe\\
\end{bmatrix}^T\begin{bmatrix}
\tD \bx - \Delta \bx\\
\tD \bl - \Delta \bl\\
\end{bmatrix} \coloneqq\mI_1 +\mI_2,
\end{equation}
where by direct calculation, 
\begin{equation}\label{eq35}
\begin{bmatrix}
\nbx \mLe\\
\nbl \mLe
\end{bmatrix} = 
\begin{bmatrix}
I + \eta_2 H  & \eta_1 G^T \\
\eta_2 G & I 
\end{bmatrix}\begin{bmatrix}
\nbx\mL\\
\nbl\mL
\end{bmatrix}.
\end{equation}
For the term $\mI_1$, we have
\begin{align*}
\mI_1
& \overset{\mathclap{\eqref{eq35}}}{=}\; \begin{bmatrix}
\Delta \bx\\
\Delta \bl\\
\end{bmatrix}^T\begin{bmatrix}
I+\eta_2 H & \eta_1 G^T\\
\eta_2 G & I\\
\end{bmatrix}\begin{bmatrix}
\nbx \mL\\
\nbl \mL\\
\end{bmatrix}\\
&\overset{\mathclap{\eqref{eq31}}}{=} -\begin{bmatrix}
\Delta \bx\\
\Delta \bl\\
\end{bmatrix}^T\begin{bmatrix}
I+\eta_2 H & \eta_1 G^T\\
\eta_2 G & I\\
\end{bmatrix}\begin{bmatrix}
\hH & G^T\\
G   & \bnz\\
\end{bmatrix}\begin{bmatrix}
\Delta \bx\\
\Delta \bl\\
\end{bmatrix}\\
& = - \begin{bmatrix}
\Delta \bx\\
\Delta \bl\\
\end{bmatrix}^T \begin{bmatrix}
\hH +\eta_2 H\hH+\eta_1 G^T G & G^T + \eta_2 H G^T\\
G + \eta_2 G \hH & \eta_2 G G^T\\
\end{bmatrix}\begin{bmatrix}
\Delta \bx\\
\Delta \bl\\
\end{bmatrix}\\
& \stackrel{\mathclap{\eqref{as44-e}}}{\leq} - \Delta\bx^T\hH\Delta\bx + \eta_2\hU^2\|\Delta\bx\|^2 - \eta_1\|G\Delta\bx\|^2 -2\Delta\bl^TG\Delta\bx + 2\eta_2\hU\|\Delta\bx\|\|G^T\Delta\bl\|-\eta_2\|G^T\Delta\bl\|^2\\
& \stackrel{\mathclap{\eqref{eq31}}}{\leq} - \Delta\bx^T\hH\Delta\bx + \eta_2\hU^2\|\Delta\bx\|^2 - \eta_1\|\bc\|^2 + 2\bc^T\Delta\bl + 2\eta_2\hU^2\|\Delta\bx\|^2- \frac{\eta_2}{2}\|G^T\Delta\bl\|^2\\
& \leq - \Delta\bx^T\hH\Delta\bx + 3\eta_2\hU^2\|\Delta\bx\|^2 - \eta_1\|\bc\|^2 + \frac{8}{\eta_2\gamma_G}\|\bc\|^2 + \frac{\eta_2\gamma_G}{8}\|\Delta\bl\|^2 - \frac{\eta_2\gamma_G}{4}\|\Delta\bl\|^2 - \frac{\eta_2}{4}\|G^T\Delta\bl\|^2\\
& \stackrel{\mathclap{\eqref{eq31}}}{=} - \Delta\bx^T\hH\Delta\bx + 3\eta_2\hU^2\|\Delta\bx\|^2 - \rbr{\eta_1 - \frac{8}{\eta_2\gamma_G}}\|\bc\|^2 - \frac{\eta_2\gamma_G}{8}\|\Delta\bl\|^2 - \frac{\eta_2}{4}\|\hH\Delta\bx + \nabla_{\bx}\mL\|^2\\
& \stackrel{\mathclap{\eqref{as44-e}}}{\leq} - \Delta\bx^T\hH\Delta\bx + 3\eta_2\hU^2\|\Delta\bx\|^2 - \rbr{\eta_1 - \frac{8}{\eta_2\gamma_G}}\|\bc\|^2 - \frac{\eta_2\gamma_G}{8}\|\Delta\bl\|^2 - \frac{\eta_2}{8}\|\nabla_{\bx}\mL\|^2 + \frac{\eta_2\hU^2}{4}\|\Delta\bx\|^2\\
& = - \Delta\bx^T\hH\Delta\bx + 3.25\eta_2\hU^2\|\Delta\bx\|^2 - \rbr{\eta_1 - \frac{8}{\eta_2\gamma_G}}\|\bc\|^2 - \frac{\eta_2\gamma_G}{8}\|\Delta\bl\|^2 - \frac{\eta_2}{8}\|\nabla_{\bx}\mL\|^2,
\end{align*}
where the sixth inequality is due to Assumption \ref{as42}. To simplify the above bound, we decompose the step $\Delta\bx$ as $\Delta\bx = \Delta \bx_1 + \Delta\bx_2$ with $\Delta\bx_1\in \text{span}(G^T)$ and $G\Delta\bx_2=\bnz$. Then, we~have
\begin{align*}
- \Delta\bx^T\hH\Delta\bx & + 3.25\eta_2\hU^2\|\Delta\bx\|^2 \\
& = - \Delta\bx_1^T\hH\Delta\bx_1 - 2
\Delta\bx_1^T\hH\Delta\bx_2 - \Delta\bx_2^T\hH\Delta\bx_2 + 3.25\eta_2\hU^2\|\Delta\bx\|^2\\
& \stackrel{\mathclap{\eqref{as44-e}}}{\leq} \hU\|\Delta\bx_1\|^2 + 2\hU\|\Delta\bx_1\|\|\Delta\bx_2\| - \gamma_H\|\Delta\bx_2\|^2 + 3.25\eta_2\hU^2\|\Delta\bx\|^2\\
& \leq \rbr{\hU + \frac{2\hU^2}{\gamma_{H}}}\|\Delta\bx_1\|^2 - \frac{\gamma_{H}}{2}\|\Delta\bx_2\|^2 + 3.25\eta_2\hU^2\|\Delta\bx\|^2\\
& = \rbr{\hU + \frac{2\hU^2}{\gamma_{H}}+\frac{\gamma_{H}}{2}}\|\Delta\bx_1\|^2 - \rbr{\frac{\gamma_{H}}{2} - 3.25\eta_2\hU^2}\|\Delta\bx\|^2\\
& \leq \rbr{\hU + \frac{2\hU^2}{\gamma_{H}}+\frac{\gamma_{H}}{2}}\frac{1}{\gamma_G}\|\bc\|^2 - \rbr{\frac{\gamma_{H}}{2} - 3.25\eta_2\hU^2}\|\Delta\bx\|^2\\
& \leq \frac{3.5\hU^2}{\gamma_{H}\gamma_G}\|\bc\|^2 - \rbr{\frac{\gamma_{H}}{2} - 3.25\eta_2\hU^2}\|\Delta\bx\|^2,
\end{align*}
where the fifth inequality is due to the fact that $\|\bc\|^2 = \|G\Delta\bx\|^2 = \|G\Delta\bx_1\|^2\geq \gamma_G\|\Delta\bx_1\|^2$, and the sixth inequality is due to $\hU\geq 1\geq \gamma_{H}$. Combining the above two displays, we obtain 
\begin{align*}
\mI_1 & \leq - \frac{\eta_2}{8}\nbr{\begin{bmatrix}
\nabla_{\bx}\mL\\
\nabla_{\bl}\mL
\end{bmatrix}}^2 - \frac{\eta_2\gamma_G}{8}\nbr{\begin{bmatrix}
\Delta\bx\\
\Delta\bl
\end{bmatrix}}^2 - \rbr{\eta_1 - \frac{8}{\eta_2\gamma_G} - \frac{3.5\hU^2}{\gamma_{H}\gamma_G} - \frac{\eta_2}{8}}\|\bc\|^2\\
& \quad  - \rbr{\frac{\gamma_{H}}{2} - 3.25\eta_2\hU^2 - \frac{\eta_2\gamma_G}{8}}\|\Delta\bx\|^2.
\end{align*}
Noting that
\begin{align*}
\frac{\gamma_{H}}{2} - 3.25\eta_2\hU^2 - \frac{\eta_2\gamma_G}{8} & \geq 0\Longleftarrow \frac{\gamma_{H}}{2}-3.375\eta_2\hU^2\geq 0 \Longleftarrow \eta_2\leq \frac{\gamma_{H}}{7\hU^2},\\
\eta_1 - \frac{8}{\eta_2\gamma_G} - \frac{3.5\hU^2}{\gamma_{H}\gamma_G} - \frac{\eta_2}{8} &\geq 0 \Longleftarrow \eta_1 - \frac{8.5}{\eta_2\gamma_G} - \frac{\eta_2}{8}\geq 0\Longleftarrow \eta_1\geq \frac{9}{\eta_2\gamma_G},
\end{align*}
we know under the condition \eqref{eq55},
\begin{equation}\label{nequ:9}
\mI_1 \leq - \frac{\eta_2}{8}\nbr{\begin{bmatrix}
\nabla_{\bx}\mL\\
\nabla_{\bl}\mL
\end{bmatrix}}^2 - \frac{\eta_2\gamma_G}{8}\nbr{\begin{bmatrix}
\Delta\bx\\
\Delta\bl
\end{bmatrix}}^2.
\end{equation}
For the term $\mI_2$, we have
\begin{align}\label{nequ:10}
\mI_2  & \overset{\substack{\eqref{eq31}, \eqref{eq35}}}{=} -\begin{bmatrix}
\tD \bx - \Delta \bx\\
\tD \bl - \Delta \bl\\
\end{bmatrix}^T\begin{bmatrix}
\hH + \eta_2 H\hH+\eta_1 G^T G & G^T + \eta_2 H G^T\\
G + \eta_2 G\hH & \eta_2 G G^T\\
\end{bmatrix}\begin{bmatrix}
\Delta \bx\\
\Delta \bl\\
\end{bmatrix} \nonumber\\
& \overset{\substack{\eqref{eq44}, \eqref{as44-e}}}{\le}\; \delta_{\mu,b} \|(\Delta \bx, \Delta \bl)\|^2 \cbr{2\hU + 3\eta_2\hU^2 + \eta_1\hU^2} \leq 1.05\eta_1\hU^2\delta_{\mu,b} \|(\Delta \bx, \Delta \bl)\|^2,
\end{align}
where the last inequality uses the fact that $2\hU + 3\eta_2\hU^2 \leq 2\hU + 3/7 \leq 2.5\hU^2 \leq 0.05\eta_1\hU^2$ (since $\eta_1\geq 63$ by $\hU\geq 1\geq \max\{\gamma_{H}, \gamma_G\}$). Combining \eqref{nequ:10} with \eqref{nequ:9}, \eqref{eq57}, and \eqref{eq55}, we complete the proof.
\end{proof}

\cref{thm51} shows that the approximate Newton direction $(\tD \bx, \tD \bl)$ is a descent direction of the exact augmented Lagrangian merit function \eqref{eq33}, ensuring sufficient reduction in each step, provided $\eta_1$ is large enough and $\eta_2$ and $\delta_{\mu,b}$ are small enough. The choices of $\eta_1$, $\eta_2$, and $\delta_{\mu,b}$ are independent of $\tau$ and $M$, and the negative inner product between $(\nabla_{\bx}\mL_\eta, \nabla_{\bl}\mL_\eta)$ and~$(\tD \bx, \tD \bl)$ guarantees the existence of a stepsize $\alpha_{\tau}$ to satisfy the Armijo condition \eqref{eq34} (by the mean value theorem).

The result in \cref{thm51} illustrates the significance of utilizing the exact augmented Lagrangian merit function. We observe that the left-hand side of \eqref{eq57} consists of two terms: $\mI_1$ that measures the reduction of the exact Newton direction on the augmented Lagrangian, and $\mI_2$ that characterizes the approximation error of OGD. By \eqref{nequ:9}, $\mI_1$ is further upper bounded by two terms. The first term is proportional to the KKT residual $\|\nabla\mL\|^2$, which is crucial for accumulating reductions across steps and achieving global convergence. The second term is proportional to the squared magnitude of the exact Newton direction $\|(\Delta\bx, \Delta\bl)\|^2$, which is the key to tolerate the approximation error of OGD in $\mI_2$. As elucidated in \cref{lem41} and \cref{thm42}, the approximation error (even for the primal variable $\|\tD\bx-\Delta\bx\|$) depends not only on $\Delta \bx$ but also on $\Delta \bl$ in the boundary variables. Therefore, a merit function that depends on both primal and dual variables, capable of tolerating the dual approximation error, is particularly preferable for our algorithm and analysis.

With \cref{thm42} and \cref{thm51}, we are able to prove the global convergence of FOGD in the next theorem.

\begin{theorem}[\textbf{Global convergence of FOGD}]\label{thm52}
Suppose Assumptions \ref{as41}--\ref{as45}, \ref{as51} hold and the parameters $(\eta_1,\eta_2, \delta_{\mu,b})$ satisfy \eqref{eq55}. Then, \cref{algo2} satisfies $\|\nabla \mLt\|\rightarrow 0$ as $\tau\rightarrow\infty$.
\end{theorem}

\begin{proof}

By twice differentiability of $\{f_i, \bc_i\}_{i\in \mV}$, we know $\nabla\mL_\eta$ is Lipschitz continuous in the set $\mX\times\Lambda$. Let us denote its Lipschitz constant as $\Upsilon_\eta$. Then, we have
\begin{align*}
\mLe^{\tau+1} &\le \mLe^{\tau} + \at\begin{bmatrix}
\nbx \mLe^{\tau} \\
\nbl \mLe^{\tau} \\
\end{bmatrix}^T
\begin{bmatrix}
\tD \bxt\\
\tD \blt\\
\end{bmatrix}
+ \frac{\Upsilon_{\eta}\at^2}{2} \Bigg\|\begin{bmatrix}
\tD \bxt\\
\tD \blt\\
\end{bmatrix} \Bigg\|^2\\
&\stackrel{\mathclap{\eqref{eq44}}}{\leq} \mLe^{\tau} + \at\begin{bmatrix}
\nbx \mLe^{\tau} \\
\nbl \mLe^{\tau} \\
\end{bmatrix}^T
\begin{bmatrix}
\tD \bxt\\
\tD \blt\\
\end{bmatrix}
+ \frac{\Upsilon_{\eta}\at^2(1+\delta_{\mu,b})^2}{2} \Bigg\|\begin{bmatrix}
\Delta \bxt\\
\Delta \blt\\
\end{bmatrix} \Bigg\|^2\\
& \stackrel{\mathclap{\eqref{eq31}}}{\leq} \mLe^{\tau} + \at\begin{bmatrix}
\nbx \mLe^{\tau} \\
\nbl \mLe^{\tau} \\
\end{bmatrix}^T
\begin{bmatrix}
\tD \bxt\\
\tD \blt\\
\end{bmatrix}
+ \frac{32\hU^4\Upsilon_{\eta}\at^2}{\gamma_{H}^2\gamma_G^2}
\|\nabla \mLt \|^2 \quad \text{(by \cref{lm51} and $\delta_{\mu,b}\leq 1$)}\\
& \le \mLe^{\tau} + \at\begin{bmatrix}
\nbx \mLe^{\tau} \\
\nbl \mLe^{\tau} \\
\end{bmatrix}^T
\begin{bmatrix}
\tD \bxt\\
\tD \blt\\
\end{bmatrix}
- \frac{256\hU^4\Upsilon_{\eta}\at^2}{\gamma_{H}^2\gamma_G^2\eta_2} \begin{bmatrix}
\nbx \mLe^{\tau} \\
\nbl \mLe^{\tau} \\
\end{bmatrix}^T\begin{bmatrix}
\tD \bxt\\
\tD \blt\\
\end{bmatrix}\quad \text{(by \cref{thm51})}\\
& = \mLe^{\tau} + \at\left\{1- \frac{256\hU^4\Upsilon_{\eta}}{\gamma_{H}^2\gamma_G^2\eta_2} \at\right \}\begin{bmatrix}
\nbx \mLe^{\tau} \\
\nbl \mLe^{\tau} \\
\end{bmatrix}^T
\begin{bmatrix}
\tD \bxt\\
\tD \blt
\end{bmatrix}.
\end{align*}
Thus, the Armijo condition \eqref{eq34} is satisfied if $\at \leq (1-\beta)\gamma_{H}^2\gamma_G^2\eta_2/(256\hU^4\Upsilon_{\eta})$, which implies the existence of a lower bound of $\at$, i.e., $\at\ge\bar{\alpha}>0$, $\forall \tau\geq 0$. Applying \cref{thm51} to \eqref{eq34}, we obtain \vskip-0.45cm
\begin{equation*}
\mLe^{\tau+1} \le \mLe^{\tau} -  \frac{\eta_2\at\beta}{8}\| \nabla \mLt\|^2 \le \mLe^{\tau} -  \frac{\eta_2\bar{\alpha}\beta}{8}\| \nabla \mLt\|^2 .
\end{equation*}\vskip-0.12cm
\noindent Accumulating the above reduction from $\tau=0$ to $\infty$, we obtain $\sum_{\tau=0}^{\infty} \| \nabla \mLt\|^2 \le \frac{8}{\eta_2\bar{\alpha}\beta}(\mL^0_{\eta}-\min_{\mX\times\Lambda}\mLe(\bx,\bl))<\infty$. This completes the proof.
\end{proof}

In the next section, we will show that FOGD exhibits local linear convergence. Enjoying global convergence and preserving fast local linear rate of the Schwarz method makes FOGD an appealing algorithm compared to other decomposition methods.

\section{Local Convergence of FOGD}\label{sec5}

To segue into local analysis, we assume in this section that the iterate $(\bxt, \blt)$ converges to a strict local solution $(\bxs,\bls)$ as $\tau\rightarrow \infty$. We rely on the following assumptions regarding the vanishing modification and the local Lipschitz continuity of the Lagrangian Hessian in local analysis.

\begin{assumption}[\textbf{Vanishing modification}]\label{as61}
We assume $\hHt \rightarrow \Ht$ as $\tau\rightarrow \infty$, where $\hHt$ is a modification of the Lagrangian Hessian $\Ht=\nbx^2\mLt$.
\end{assumption}

We mention that the Hessian modification is intended to enforce Assumption \ref{as41} to have a well-defined quadratic program \eqref{eq36}. Since $(\bxs,\bls)$ is a strict local solution, by the second-order sufficient condition, Assumption \ref{as41} directly applies to the exact Hessian $\Ht$ for large $\tau$, suggesting that no modifications may be needed locally and, thus, Assumption \ref{as61} holds naturally. We also need the local Lipschitz continuity condition.

\begin{assumption}[\textbf{Local Lipschitz continuity}]\label{as62}
There exists a constant $\Upsilon_L\geq 1$ such that for any two points $(\bx, \bl)$ and $(\bx', \bl')$ sufficiently close to $(\bxs,\bls)$ and any nodes $i,j\in \mV$, \vskip-0.4cm
\begin{align*}
\|H_{i,j}(\bx, \bl) - H_{i,j}(\bx', \bl')\| & \leq \Upsilon_L\max_{k\in N_\mG^2[i]}\|(\bx_k,\bl_k) - (\bx_k',\bl_k')\|, \\
\|G_{i,j}(\bx, \bl) - G_{i,j}(\bx', \bl')\| & \leq \Upsilon_L\max_{k\in N_\mG[i]}\|(\bx_k,\bl_k) - (\bx_k',\bl_k')\|.
\end{align*}
\end{assumption}

\vskip-0.1cm
We note that $H_{i,j}$ only depends on variables $\{(\bx_k, \bl_k)\}_{k\in N_\mG^2[i]}$ and $G_{i,j}$ only depends on variables $\{(\bx_k, \bl_k)\}_{k\in N_\mG[i]}$. The next result shows that the unit stepsize is locally admissible to satisfy the Armijo condition \eqref{eq34}.

\begin{lemma}\label{thm61}
\hskip-3pt Suppose Assumptions \ref{as41}--\ref{as45}, \ref{as51}, \ref{as61}--\ref{as62} hold and the parameters $(\eta_1,\eta_2,\\ \delta_{\mu,b})$ satisfy \vskip-0.47cm
\begin{equation}\label{eq61}
\eta_1\ge \frac{9}{\eta_2\gG},\quad\quad\quad\eta_2\le\frac{\gH}{7\hU^2},\quad\quad\quad \delta_{\mu,b} \leq \frac{0.5-\beta}{1.5-\beta}\cdot \frac{\eta_2\gamma_G}{8.5\eta_1\hU^2},
\end{equation}  \vskip-0.07cm
\noindent then $\alpha_{\tau}=1$ for all sufficiently large $\tau$.
\end{lemma}

\begin{proof}
See \cref{athm61}.
\end{proof}

\cref{thm61} retains the same conditions on $\eta_1$ and $\eta_2$ as in \cref{thm51}. However, since $(0.5-\beta)/(1.5-\beta)\leq 1/3$, it imposes a stronger condition on $\delta_{\mu,b}$. 
\cref{thm61} suggests that the exact augmented Lagrangian merit function effectively overcomes the Maratos effect and admits the unit stepsize locally, provided that $\eta_2$ and $\delta_{\mu,b}$ are sufficiently small and $\eta_1$ is sufficiently large. Admitting the unit stepsize is crucial for achieving a fast local convergence rate.~We may need to enhance the design by employing other techniques, such as second-order correction, to overcome the Maratos effect when using other non-differentiable merit functions. Those techniques would need to be decentralized in our parallel setting, significantly \mbox{complicating}~the applicability of the method. 

We then show the local linear convergence of FOGD in the next theorem.

\begin{theorem}[\textbf{Uniform local linear convergence of FOGD}]\label{lm62}
Let us define $\Psi^{\tau}\coloneqq \max_{k\in\mV} \Psi^{\tau}_k \coloneqq  \max_{k\in\mV} \|(\bx_k^{\tau}, \bl_k^{\tau}) - (\bxs_k, \bls_k)\|$. Suppose Assumptions \ref{as41}--\ref{as45}, \ref{as51}, \ref{as61}--\ref{as62} hold and the parameters $(\eta_1,\eta_2,\delta_{\mu,b})$ satisfy \eqref{eq61}. Then, for sufficiently large $\tau$, we have
\begin{equation*}
\Psi^{\tau+1} \leq 2\{\Upsilon_\mu\texttt{poly}_{\mG\circ\mW}(b+4)\}^2\rho_\mu^b \cdot\Psi^\tau,
\end{equation*}
where $(\Upsilon_\mu, \rho_\mu)$ are defined in \cref{lm61} and $b$ is the overlap size. 
\end{theorem}

\begin{proof}
To ease notation, we drop the iteration index $\tau$ in the proof, and by \cref{thm61} we assume $\tau$ is large enough so that $\alpha_{\tau}=1$. We use $(\bx, \bl)$ and $(\bx',\bl')$ to denote two successive iterates, and let $(\btx_{\ell}, \btl_{\ell})\coloneqq \mD_{\ell}(\bx, \bl)$, $(\btx_{\ell}', \btl_{\ell}')\coloneqq \mD_{\ell}(\bx', \bl')$, $(\btxs_\ell, \btls_\ell) \coloneqq\mD_{\ell}(\bxs, \bls)$ be truncated iterates and solution of the subproblem $\ell$ (cf. \cref{df22}). With the above notation, for any $k\in\mV$, we consider the subproblem $\ell$ such that $k\in\mV_{\ell}$. Then, we have
\begin{align}\label{nequ:12}
\I_{\mV_\ell}\begin{bmatrix}
\btx_{\ell}' - \btxs_\ell\\    
\btl_{\ell}' - \btls_\ell
\end{bmatrix} & = \I_{\mV_\ell}\cbr{\begin{bmatrix}
\btx_{\ell} - \btxs_\ell\\    
\btl_{\ell} - \btls_\ell
\end{bmatrix} + \begin{bmatrix}
\towt_{\ell}(\bnz)\\
\tzewt_{\ell}(\bnz)
\end{bmatrix}} \nonumber\\
& \stackrel{\mathclap{\eqref{nequ:11}}}{=}\;\; \I_{\mV_\ell}\bigg\{\begin{bmatrix}
\btx_{\ell} - \btxs_\ell\\    
\btl_{\ell} - \btls_\ell
\end{bmatrix} -  {\underbrace{\begin{bmatrix}
\hHm_{\ell} & G_{\ell}^T\\
G_{\ell} & \bnz\\
\end{bmatrix}}_{K_\ell^\mu}}^{-1}\underbrace{\begin{bmatrix}
\mu (G[\hat{\mW}_{\ell}][\mW_{\ell}])^T\bc[\hat{\mW}_{\ell}]+\nabla_{\bx} \mL[\mW_{\ell}]\\
\bc[\mW_{\ell}\backslash \tmW_{\ell}]
\end{bmatrix} }_{\bg_\ell^\mu}\bigg\}. 
\end{align}
Here, $\I_{\mV_\ell}$ is a block-diagonal matrix with two blocks of sizes $|\mV_{\ell}|\times |\mW_\ell|$ and $|\mV_{\ell}|\times |\mW_\ell\backslash\tilde{\mW}_\ell|$. For each block, only the entries $(k,k)$ with $k\in \mV_\ell$ are 1, while other entries are 0. (Note that the matrices here are graph-indexed). This means \eqref{nequ:12} holds only for nodes in $\mV_{\ell}$, because we discard the variables on the overlap region. Furthermore, $K_\ell^\mu$ and $\bg_\ell^\mu$ are evaluated at $(\btx_\ell, \btl_\ell, \btd_\ell)$ with boundary variables $\btd_\ell = (\bx[\bar{\mW}_\ell], \bl[\hat{\mW}_\ell])$ (the evaluations do not solely depend on the truncated iterate $(\btx_\ell, \btl_\ell) = (\bx[\mW_{\ell}], \bl[\mW_{\ell}\backslash\tilde{\mW}_\ell])$; see \eqref{eq37} and \eqref{eq37_sec3} for the definitions of $\tilde{\mW}_{\ell}$, $\hat{\mW}_\ell$, and $\bar{\mW}_\ell$). We aim to apply Taylor expansion for the vector $\bg_\ell^\mu$. Let us define two segments
\begin{equation*}
\btx_{\ell}(\phi) = \btxs_{\ell} + \phi(\btx_{\ell} - \btxs_{\ell}) \quad\quad\text{and}\quad\quad \btl_{\ell}(\phi) = \btls_{\ell} + \phi(\btl_{\ell} - \btls_{\ell}),\quad\quad  \text{ for } \phi\in[0, 1].
\end{equation*}
Then, we denote $(\bg_\ell^\mu)^{\boldsymbol{*}} = ((\bg_\ell^\mu)^{\boldsymbol{*}}_1, (\bg_\ell^\mu)^{\boldsymbol{*}}_2)$ to be the vector evaluated at $(\btxs_\ell, \btls_\ell, \btd_\ell)$ and denote $H_\ell^\mu(\phi)$ and $G_\ell(\phi)$ to be the matrices evaluated at $(\btx_\ell(\phi), \btl_\ell(\phi),\btd_\ell)$. With these notation, we further have
\begin{align}\label{nequ:13}
\begin{bmatrix}
(\bg_\ell^\mu)_1\\[2pt]
(\bg_\ell^\mu)_2
\end{bmatrix} & = \begin{bmatrix}
(\bg_\ell^\mu)^{\boldsymbol{*}}_1\\[2pt]
(\bg_\ell^\mu)^{\boldsymbol{*}}_2
\end{bmatrix} + \begin{bmatrix}
(\bg_\ell^\mu)_1 - (\bg_\ell^\mu)^{\boldsymbol{*}}_1\\[2pt]
(\bg_\ell^\mu)_2 - (\bg_\ell^\mu)^{\boldsymbol{*}}_2
\end{bmatrix} \nonumber\\
& = \begin{bmatrix}
(\bg_\ell^\mu)^{\boldsymbol{*}}_1\\[2pt]
(\bg_\ell^\mu)^{\boldsymbol{*}}_2
\end{bmatrix}  + \int_{0}^{1}\begin{bmatrix}
H_\ell^\mu(\phi) + \mu\langle \nabla^2_{\bx[\mW_\ell]} \bc[\hat{\mW}_\ell](\phi), \bc[\hat{\mW}_\ell]\rangle & G_\ell^T(\phi)\\[2pt]
G_\ell(\phi) & \bnz
\end{bmatrix}\begin{bmatrix}
\btx_{\ell} - \btxs_\ell\\    
\btl_{\ell} - \btls_\ell
\end{bmatrix} d\phi,
\end{align}
where
\begin{equation*}
\langle \nabla^2_{\bx[\mW_\ell]} \bc[\hat{\mW}_\ell](\phi), \bc[\hat{\mW}_\ell]\rangle \coloneqq \sum_{i\in \hat{\mW}_\ell} \bc_i(\bx[\mW_\ell\cup \bar{\mW}_\ell])
\cdot \nabla^2_{\bx[\mW_\ell]}\bc_i(\{\btx_\ell(\phi), \bx[\bar{\mW}_\ell]\}).
\end{equation*}
Combining \eqref{nequ:12} and \eqref{nequ:13} together, we have
\begin{align*}
& \I_{\mV_\ell} \begin{bmatrix}
\btx_{\ell}' - \btxs_\ell\\    
\btl_{\ell}' - \btls_\ell
\end{bmatrix}  = \I_{\mV_\ell}\bigg\{ - \begin{bmatrix}
\hHm_{\ell} & G_{\ell}^T\\
G_{\ell} & \bnz\\
\end{bmatrix}^{-1}\begin{bmatrix}
(\bg_\ell^\mu)^{\boldsymbol{*}}_1\\[2pt]
(\bg_\ell^\mu)^{\boldsymbol{*}}_2
\end{bmatrix} \\
& \quad + \begin{bmatrix}
\hHm_{\ell} & G_{\ell}^T\\
G_{\ell} & \bnz\\
\end{bmatrix}^{-1}\int_0^1\begin{bmatrix}
\hat{H}_\ell^\mu - H_\ell^\mu(\phi) - \mu\langle \nabla^2_{\bx[\mW_\ell]} \bc[\hat{\mW}_\ell](\phi), \bc[\hat{\mW}_\ell]\rangle & G_\ell^T-G_\ell^T(\phi)\\
G_\ell - G_\ell(\phi) & \bnz
\end{bmatrix}\begin{bmatrix}
\btx_{\ell} - \btxs_\ell\\    
\btl_{\ell} - \btls_\ell
\end{bmatrix} d\phi \bigg\}\\
& \coloneqq \I_{\mV_\ell} (\mJ_1 + \mJ_2). 
\end{align*}
We obtain the nodewise upper bounds for the terms $\mJ_1$ and $\mJ_2$ separately as follows. We first deal with $\mJ_1$. Recall from the proof of \cref{lem41} that $\cmW_{\ell}\coloneqq (N_{\mG}(\tmW_{\ell})\cap \mW_{\ell})\cup \tmW_{\ell}$ is the set of internal boundary nodes of $\mW_{\ell}$ within the depth of two. For notational clarity, we spell out the evaluation point. For $(\bg_\ell^\mu)^{\boldsymbol{*}}_1$ and any node $j\in  \mW_\ell \backslash\cmW_\ell$,
\begin{align*}
[(\bg_\ell^\mu)^{\boldsymbol{*}}_1]_j & = [(\bg_\ell^\mu)_1]_j(\btxs_\ell, \btls_\ell, \btd_\ell) \\
& = \mu \{(G[\hat{\mW}_\ell][j])^T\bc[\hat{\mW}_\ell]\}(\btxs_\ell, \btls_\ell, \btd_\ell) + \nabla_{\bx} \mL[j](\btxs_\ell, \btls_\ell, \btd_\ell) = \nabla_{\bx} \mL[j](\btxs_\ell, \btls_\ell) = \bnz,
\end{align*}
where the third equality is because $G[\hat{\mW}_\ell][j] = \bnz$ at any evaluation point (since $j\notin N_\mG[\hat{\mW}_\ell]$) and the evaluation of $\nabla_{\bx} \mL[j]$ does not depend on $\btd_\ell$. On the other hand, for any node $j\in \cmW_{\ell}$, we have
\begin{align}\label{nequ:14}
\|[(\bg_\ell^\mu)^{\boldsymbol{*}}_1]_j\| & = \|[(\bg_\ell^\mu)_1]_j(\btxs_\ell, \btls_\ell, \btd_\ell)\| = \|[(\bg_\ell^\mu)_1]_j(\btxs_\ell, \btls_\ell, \btd_\ell) - [(\bg_\ell^\mu)_1]_j(\btxs_\ell, \btls_\ell, \btds_\ell)\| \nonumber\\
& \hskip-0.7cm \leq \|\nabla_{\bx_j} \mL(\btxs_\ell, \btls_\ell, \btd_\ell) - \nabla_{\bx_j} \mL(\btxs_\ell, \btls_\ell, \btds_\ell)\| \nonumber\\
& \hskip-0.7cm \quad + \mu \|\{(G[\hat{\mW}_\ell][j])^T\bc[\hat{\mW}_\ell]\}(\btxs_\ell, \btls_\ell, \btd_\ell) - \{(G[\hat{\mW}_\ell][j])^T\bc[\hat{\mW}_\ell]\}(\btxs_\ell, \btls_\ell, \btds_\ell)\| \nonumber\\
& \hskip-0.7cm  \stackrel{\mathclap{\eqref{eq53}}}{\leq}\; \hU\|\btd_\ell - \btds_\ell\| + \mu\int_{0}^{1} \nbr{\{(G[\hat{\mW}_\ell][j])^TG[\hat{\mW}_\ell][\bar{\mW}_\ell]\}(\btxs_\ell, \bx[\bar{\mW}_\ell](\phi))}\cdot \nbr{\bx[\bar{\mW}_\ell] - \bxs[\bar{\mW}_\ell]} d\phi \nonumber\\
& \hskip-0.7cm  \quad + \mu\int_{0}^{1}\nbr{\{\langle\nabla^2_{\bx_j\bx[\bar{\mW}_\ell]}\bc[\hat{\mW}_\ell], \bc[\hat{\mW}_\ell]\rangle\}(\btxs_\ell, \bx[\bar{\mW}_\ell](\phi))}\cdot \nbr{\bx[\bar{\mW}_\ell] - \bxs[\bar{\mW}_\ell]} d\phi \nonumber\\
& \hskip-0.7cm  \leq 1.3\mu\hU^2 \|\btd_\ell - \btds_\ell\| \leq 1.3\mu\hU^2\cdot|\hat{\mW}_\ell\cup \bar{\mW}_{\ell}|\cdot\Psi \leq 1.3\mu\hU^2\cdot|\mW_{\ell}\cup \bar{\mW}_\ell|\cdot\Psi \nonumber\\
& \hskip-0.7cm  \leq 1.3\mu\hU^2\cdot \texttt{poly}_{\mG\circ\mW}(b+4)\cdot \Psi,
\end{align}
where the fifth inequality is due to \eqref{eq53} and the facts that $\mu\geq 4$, $\|\bx[\bar{\mW}_\ell] - \bxs[\bar{\mW}_\ell]\|\leq \|\btd_\ell - \btds_\ell\|$, and $\|\bc[\hat{\mW}_\ell](\btxs_\ell, \bx[\bar{\mW}_\ell](\phi))\|$ is arbitrarily small when $\bx$ is sufficiently close to $\bxs$ (i.e., $\tau$ is large enough); the sixth inequality is by the definition of $\Psi(=\Psi^\tau)$; and the last inequality is due to Assumption \ref{as45} (note that $d_\mG(j_1,j_2)\leq \texttt{poly}_\mW(b)+4\leq \texttt{poly}_\mW(b+4)$ for any nodes $j_1,j_2\in \mW_{\ell}\cup \bar{\mW}_\ell$). We analyze $(\bg_\ell^\mu)^{\boldsymbol{*}}_2$ in a similar way. For any node $j\in \mW_{\ell}\backslash\tilde{\mW}_\ell$,
\begin{equation*}
[(\bg_\ell^\mu)^{\boldsymbol{*}}_2]_j = [(\bg_\ell^\mu)_2]_j(\btxs_\ell, \btls_\ell, \btd_\ell) = \bc[\mW_{\ell}\backslash\tilde{\mW}_\ell](\btxs_\ell) = \bnz.
\end{equation*}
Next, we deal with the term $\mJ_2$ and show that the nodewise upper bound of the integral term is only $o(\Psi)$, where $o(\cdot)$ denotes the small ``o" order in the usual sense. Specifically, by Assumptions \ref{as61}, \ref{as62}, \eqref{eq53}, and the fact that $\bc[\hat{\mW}_\ell](\btx_\ell, \btd_\ell)$ is arbitrarily small when $\bx$ is close to $\bxs$, we know
\begin{equation*}
\begin{bmatrix}
\hat{H}_\ell^\mu - H_\ell^\mu(\phi) - \mu\langle \nabla^2_{\bx[\mW_\ell]} \bc[\hat{\mW}_\ell](\phi), \bc[\hat{\mW}_\ell]\rangle & G_\ell^T-G_\ell^T(\phi)\\
G_\ell - G_\ell(\phi) & \bnz
\end{bmatrix} = o(1).
\end{equation*}
Combining $\mJ_1$ and $\mJ_2$ in the above four displays together and applying \cref{lm61}, we have
\begin{align*}
& \|\btx_{\ell,k}' - \btxs_{\ell,k}\| \leq \big\|\sum_{j\in\mW_{\ell}}[(K_\ell^\mu)_1^{-1}]_{k,j} [(\bg_\ell^\mu)^{\boldsymbol{*}}_1]_j +  \sum_{j\in\mW_{\ell}\backslash\tilde{\mW}_\ell}[(K_\ell^\mu)_2^{-1}]_{k,j} [(\bg_\ell^\mu)^{\boldsymbol{*}}_2]_j \big\| + o(\Psi)\\
& = \big\|\sum_{j\in\cmW_{\ell}} [(K_\ell^\mu)_1^{-1}]_{k,j} [(\bg_\ell^\mu)^{\boldsymbol{*}}_1]_j \big\| + o(\Psi)  \stackrel{\eqref{nequ:14}}{\leq} 1.5\mu\hU^2\cdot \texttt{poly}_{\mG\circ\mW}(b+4)\cdot \Psi\cdot \Upsilon_\mu\sum_{j\in\cmW_{\ell}}  \rho_\mu^{d_\mG(k,j)}\\
& \leq  1.5\mu\hU^2\cdot \texttt{poly}_{\mG\circ\mW}(b+4)\cdot \Psi\cdot \Upsilon_\mu\rho_\mu^{b-1} |\cmW_{\ell}| \\
& \stackrel{\mathclap{\eqref{nequ:8}}}{\leq}\;\; \Upsilon_{K} \cdot\texttt{poly}_{\mG\circ\mW}(b+4)\cdot\Psi\cdot \Upsilon_\mu\rho_\mu^{b-1} \texttt{poly}_{\mG\circ\mW}(b) \leq \{\Upsilon_\mu\texttt{poly}_{\mG\circ\mW}(b+4)\}^2\rho_\mu^b\Psi,
\end{align*}
where the third inequality applies \cref{lm61}, the fourth inequality is because $d_\mG(k,j) \geq b-1$ for any $k\in\mV_{\ell}$ and $j\in \cmW_\ell$, and the fifth and sixth inequalities are because of the definition in \cref{lm61}. The above derivations also hold for $\|\btl_{\ell,k}' - \btls_{\ell,k}\|$. We complete the proof by~noting that $\|(\btx_{\ell,k}', \btl_{\ell,k}') - (\btxs_{\ell,k}, \btls_{\ell,k})\| \leq \|\btx_{\ell,k}' - \btxs_{\ell,k}\| + \|\btl_{\ell,k}' - \btls_{\ell,k}\|$.
\end{proof}

From the proof of \cref{lm62}, we observe that the error $(\bxt_k - \bxs_k, \blt_k - \bls_k)$ consists of two terms. The first term, $\mJ_1$, represents a nodewise linear convergence rate achieved by the overlapping decomposition procedure, while the second term, $\mJ_2$, is of the order $o(\Psi^\tau)$ and represents a superlinear convergence rate achieved by the SQP framework. The linear rate of the decomposition procedure dominates the convergence rate of FOGD for sufficiently large $\tau$, and improves exponentially in terms of the overlap size. We note that the uniform linear convergence also matches that of the Schwarz method \cite{Na2022Convergence}.
We finally summarize the global and local convergence results of FOGD in the next theorem.

\begin{theorem}[\textbf{Convergence of FOGD}]\label{thm64}
\hskip-2pt Consider \cref{algo2} under Assumptions \ref{as41}--\ref{as45} and \ref{as51}, and suppose the parameters $(\eta_1, \eta_2, \delta_{\mu,b})$ satisfy \eqref{eq61}. Then, $\|\nabla \mLt\|\rightarrow 0$ as $\tau\rightarrow\infty$. Moreover, if Assumptions \ref{as61}--\ref{as62} hold locally, then for sufficiently large $\tau$, $\Psi^{\tau+1} \leq 2\{\Upsilon_\mu\texttt{poly}_{\mG\circ\mW}(b+4)\}^2\rho_\mu^b \cdot\Psi^\tau$, where the nodewise error $\Psi^{\tau}$ is defined in \cref{lm62}, the constants $(\Upsilon_\mu, \rho_\mu)$ are defined in \cref{lm61}, and $b$ is the overlap size. 

\end{theorem}

\begin{proof}
We combine \cref{thm52} with \cref{lm62} and complete the proof.
\end{proof}

\section{Numerical Experiments}\label{sec6}

We verify our theoretical results and illustrate the performance of FOGD by conducting numerical experiments on {\red two} PDE-constrained problems: a semilinear elliptic PDE problem \cite{Antil2023ALESQP} and {\red a boundary heating problem} \cite{Curtis2012note}.

\subsection{Semilinear Elliptic PDE Problem}\label{sec6:pde}

We let $\Omega \subseteq \mR^2$ be a two-dimensional domain and consider the problem
\begin{subequations}\label{eqn31}
\begin{align}
\min_{\bz,\bu}\;\; & f(\bz, \bu ) = \int_{\bomega\in\Omega} \cbr{ (\bz(\bomega) - z_d)^2 + \alpha\cdot\bu^2(\bomega) }d\bomega ,\\ 
\text{s.t.} \;\; & g(\bz,\bu) = 
\begin{cases} -\Delta \bz + \bz^p - \bu = 0\quad \text{in $\Omega$},\\ 
\bz = 0\quad \text{on $\partial\Omega$}, 
\end{cases} 
\end{align}
\end{subequations}
where $\bz, \bu: \Omega\rightarrow \mR$ are the state and control mappings, $z_d$ is the desired state, and $\Delta$ is the Laplace operator of form $\Delta\bz = \frac{\partial^2\bz}{\partial \bomega_1^2} + \frac{\partial^2\bz}{\partial \bomega_2^2}$. We let $\Omega = (0,39)^2$, $z_d = -5$, $\alpha = 0.5$, and $p=4$. To solve Problem \eqref{eqn31}, we discretize $\Omega$ by a $40\times 40$ mesh grid with evenly spaced nodes, denoted as $\Omega_h$. We then decompose the grid into five disjoint subgraphs, given by $\{(0,7), (8,15), \ldots,\\ (32,39)\}\times (0,39)$. The overlapping graph decomposition is illustrated in \cref{fig31}.

\begin{figure}[thb!]
\centering
\begin{subfigure}[b]{0.32\textwidth}
\centering
\includegraphics[width=\textwidth]{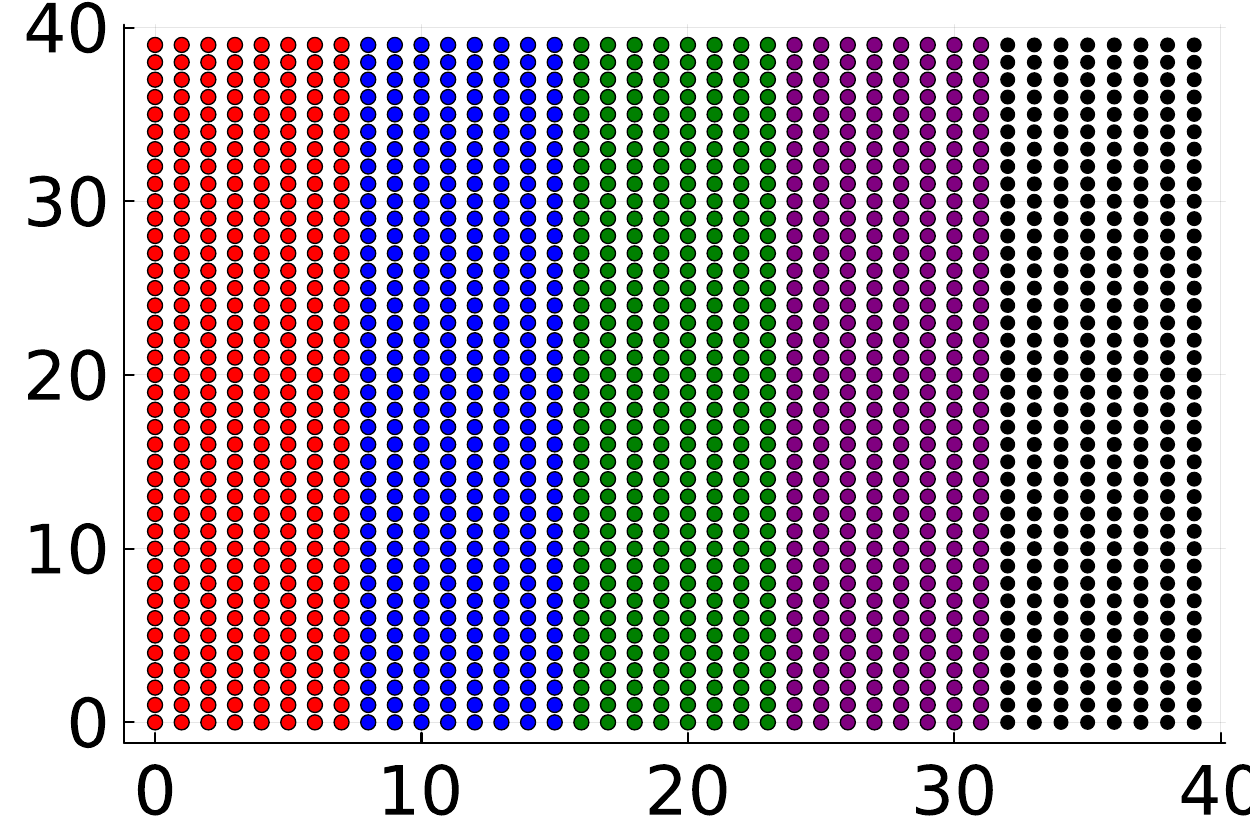}
\caption{Five disjoint subgraphs}
\label{fig31.a}
\end{subfigure}
\hfill
\begin{subfigure}[b]{0.32\textwidth}
\centering
\includegraphics[width=\textwidth]{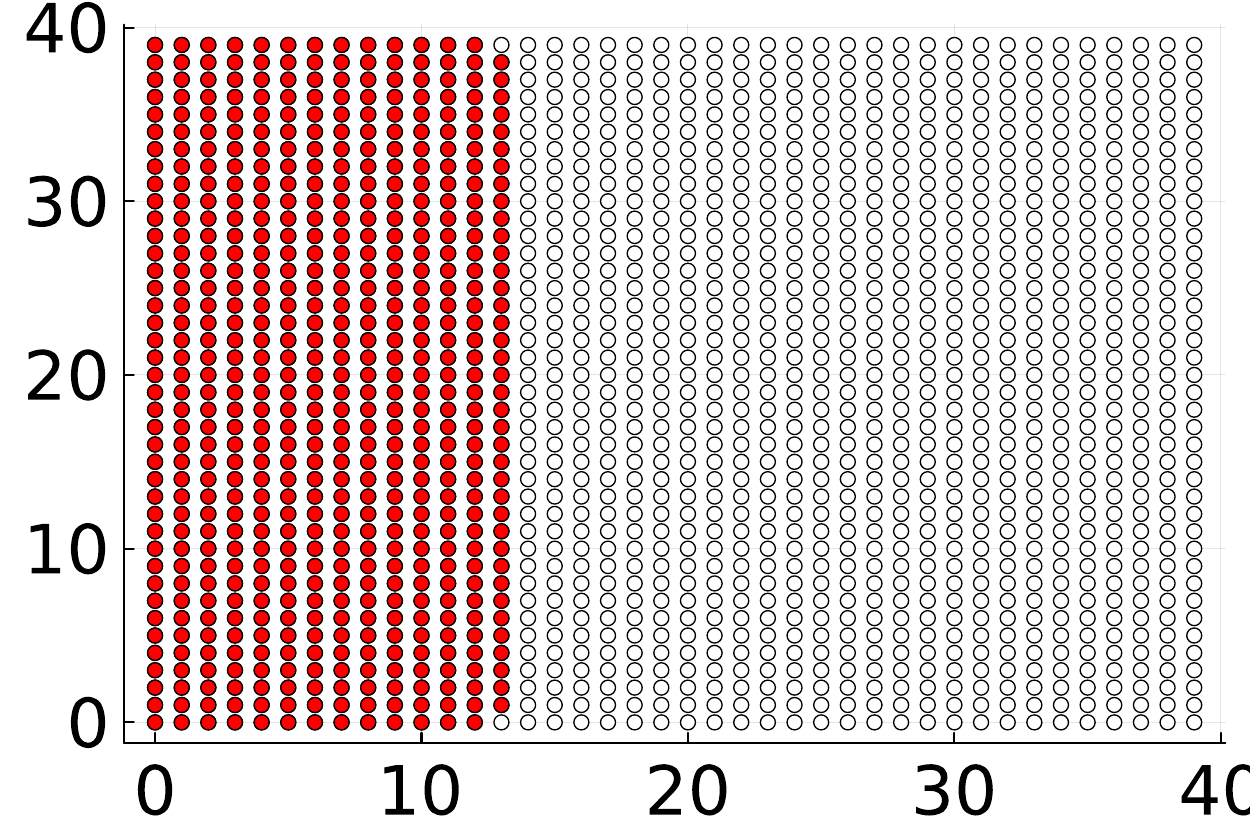}
\caption{$\mW_1$}
\label{fig31.b}
\end{subfigure}
\hfill
\begin{subfigure}[b]{0.32\textwidth}
\centering
\includegraphics[width=\textwidth]{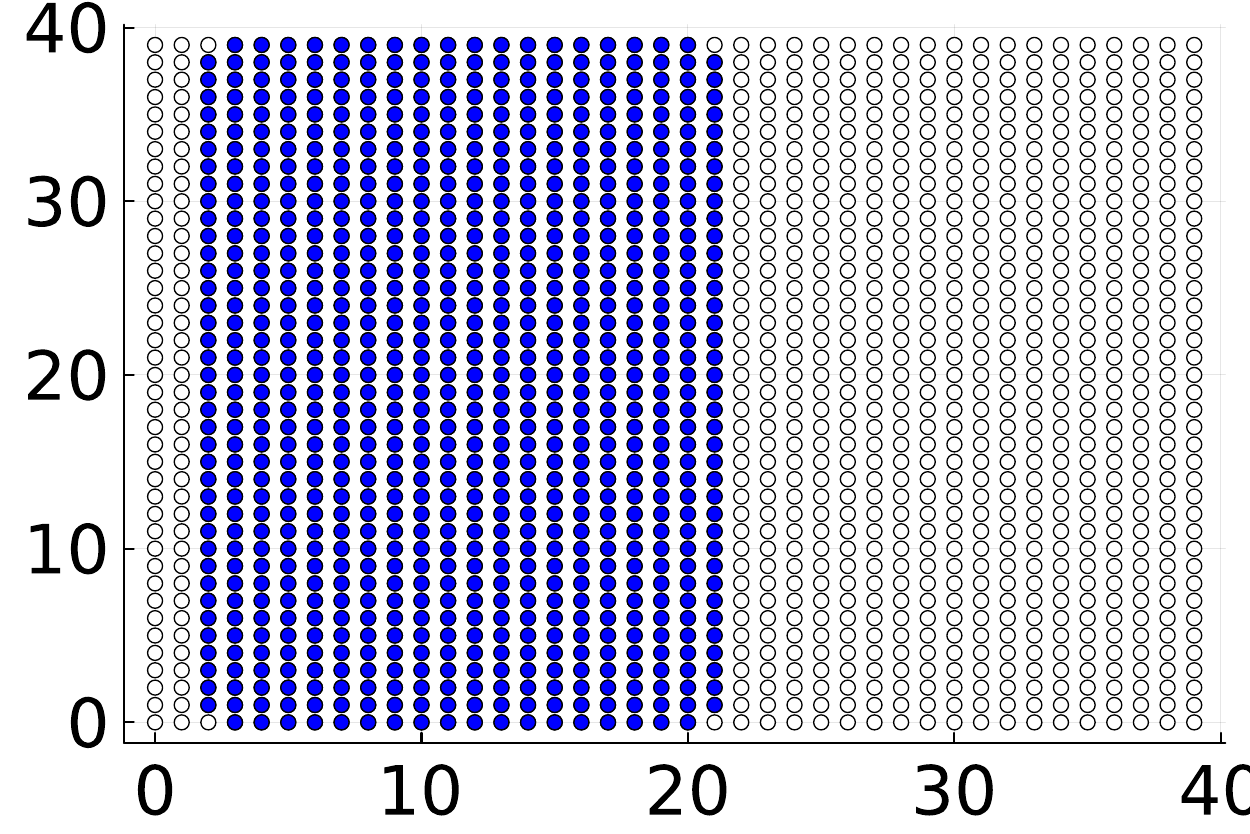}
\caption{$\mW_2$}
\label{fig31.c}
\end{subfigure}
\vskip-0.4cm
\begin{subfigure}[b]{0.32\textwidth}
\centering
\includegraphics[width=\textwidth]{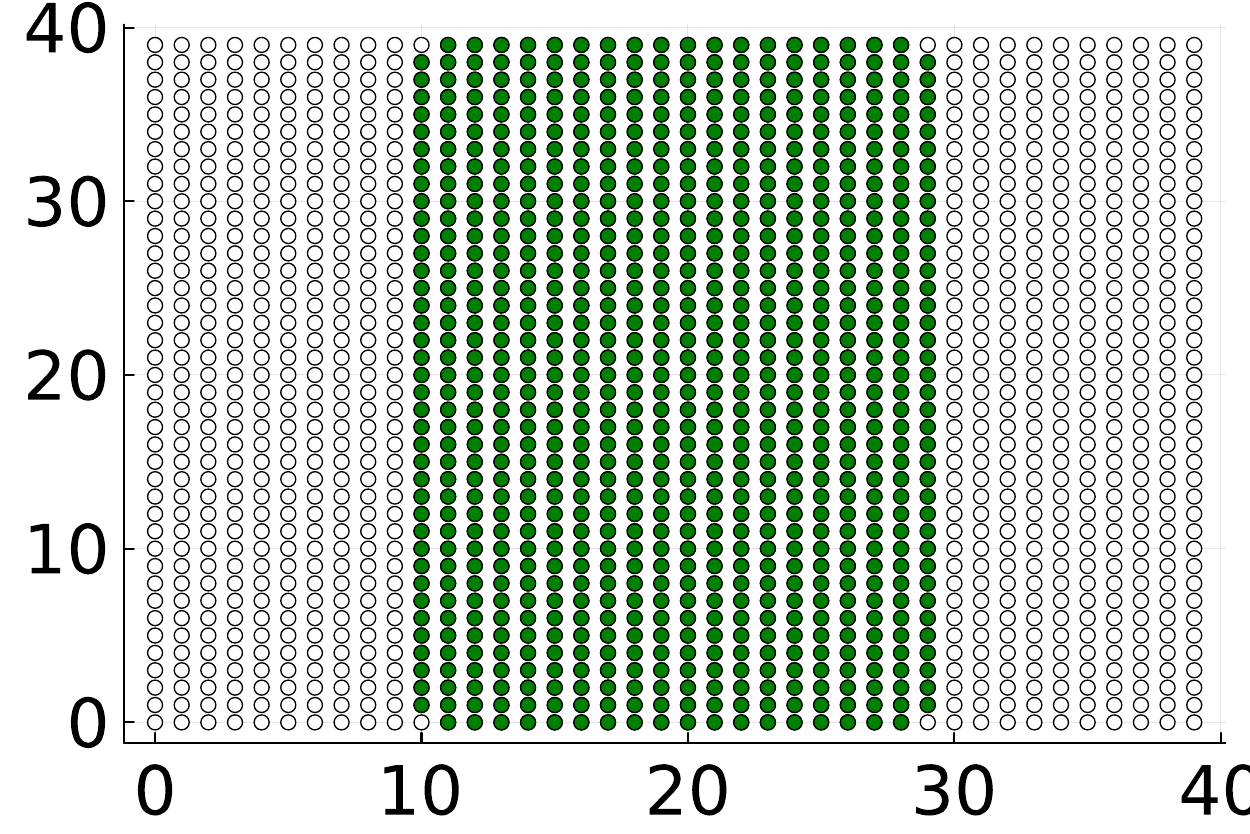}
\caption{$\mW_3$}
\label{fig31.d}
\end{subfigure}
\hfill
\begin{subfigure}[b]{0.32\textwidth}
\centering
\includegraphics[width=\textwidth]{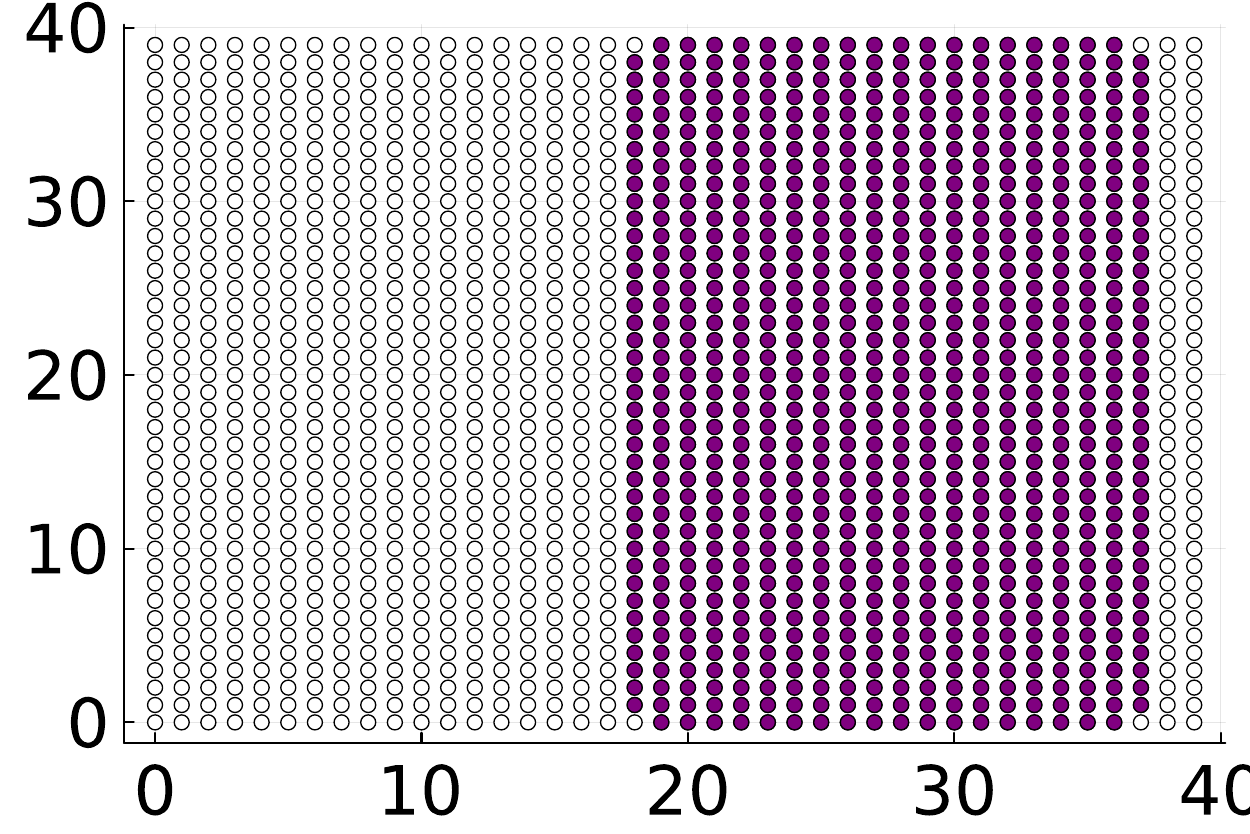}
\caption{$\mW_4$}
\label{fig31.e}
\end{subfigure}
\hfill
\begin{subfigure}[b]{0.32\textwidth}
\centering
\includegraphics[width=\textwidth]{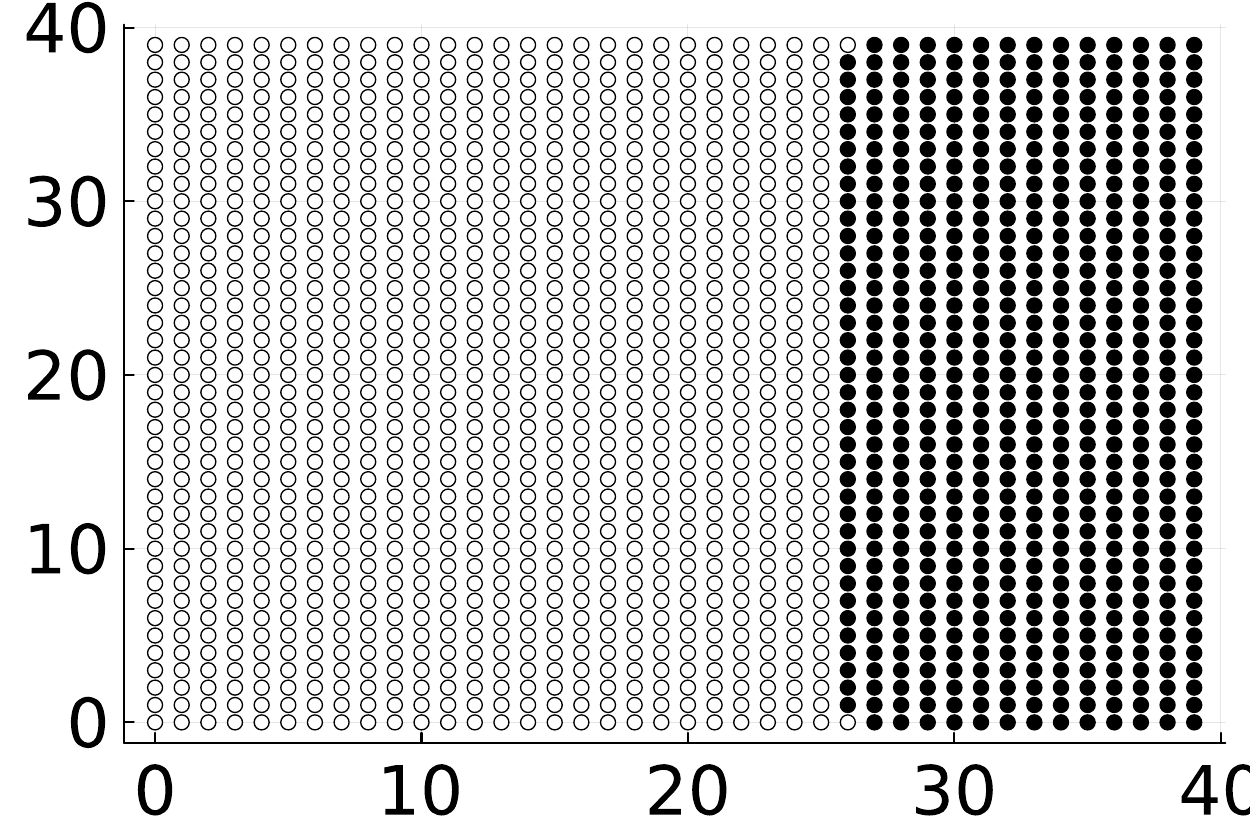}
\caption{$\mW_5$}
\label{fig31.f}
\end{subfigure}
\setlength{\belowcaptionskip}{-10pt}
\vskip-0.5cm
\caption{Illustration of the graph decomposition. The full graph is $40\times 40$ mesh grid. \cref{fig31.a} shows one way to decompose the full graph into five disjoint subgraphs, where different colors represent different subgraphs. \cref{fig31.b}--\cref{fig31.f} show the five overlapping subgraphs with the overlap size $b=6$.}\label{fig31} 
\end{figure}

\noindent $\bullet$ \textbf{FOGD.} We consider five different overlap sizes $b \in \{1,2,4,6,8\}$ and set $\eta_1 = 5$, $\eta_2 = 0.1$, $\beta = 0.1$, and $\mu = 1$. For backtracking line search, we first initialize the stepsize $\alpha_{\tau}$ to 1 and continue to decrease the stepsize by $\alpha_{\tau} \leftarrow 0.9\alpha_{\tau}$ until the Armijo condition \eqref{eq34} is satisfied. To demonstrate global convergence, we generate five initial iterates randomly~from~$(\bz^0, \bu^0, \bl^0)\overset{\text{iid}}{\sim} \text{Uniform}(-10^2,10^2)$. Then, we apply different values of $b$ to solve \eqref{eqn31} for each initialization. To demonstrate local convergence rate, we use the final iterate from one FOGD trajectory with overlap size $b=8$ as the reference solution. For each overlap size, we compute the nodewise error $\Psi^\tau$ relative to this reference solution along the global trajectory, and report the convergence behavior after the iterate first enters the local neighborhood $\Psi^\tau \leq 1$. We re-index the iterations after this entrance time by setting $\tilde{\tau}=0$ at the first iterate satisfying $\Psi^\tau \leq 1$.

{\red 
Then, we briefly discuss to what extent the above numerical setting satisfies the assumptions used in the analysis. Let $\Omega_h^\circ$ and $\partial\Omega_h$ denote the interior and boundary nodes of $\Omega_h$, respectively. Hence, the discrete constraint mapping can be written as \vspace{-0.15cm}
\begin{equation*}
\begin{bmatrix}
\{\bz_i\}_{i\in \partial\Omega_h}\\
\{\bc_i(\{\bz_j\}_{j\in N_{\Omega_h}[i]}, \{\bu_j\}_{j\in N_{\Omega_h}[i]})\}_{i\in \Omega_h^\circ}
\end{bmatrix}=\begin{bmatrix}
\bnz\\
\bnz
\end{bmatrix},
\end{equation*} \vskip-0.15cm
\noindent where $\bc_i(\{\bz_j\}_{j\in N_{\Omega_h}[i]}, \{\bu_j\}_{j\in N_{\Omega_h}[i]}) = 4z_i-\sum_{j \in N_{\Omega_h}(i)} z_j + z_i^p - u_i$
and $N_{\Omega_h}(i) = N_{\Omega_h}[i] \setminus \{i\}$ is the open nearest-neighbor set of node $i$. To verify Assumption \ref{as42}, we order the primal variable as $(\bz[\partial\Omega_h], \bu[\partial\Omega_h], \bz[\Omega_h^\circ], \bu[\Omega_h^\circ])$ and write the Jacobian \(G^\tau\) as \vspace{-0.1cm}
\begin{equation*}
G^\tau \coloneqq
\begin{bmatrix}
I & \bnz & \bnz & \bnz\\
A^\tau[\Omega_h^\circ][\partial\Omega_h] & \bnz & A^\tau[\Omega_h^\circ][\Omega_h^\circ] & -I
\end{bmatrix},
\end{equation*}
where \vspace{-0.1cm}
\begin{align*}
(A^\tau[\Omega_h^\circ][\partial\Omega_h])_{ij}
& \coloneqq
\begin{cases}
-1, & j \in N_{\Omega_h}(i)\cap \partial\Omega_h,\\
0, & \text{otherwise},
\end{cases}\\
(A^\tau[\Omega_h^\circ][\Omega_h^\circ])_{ij} & \coloneqq
\begin{cases}
4 + p \cdot (z_i^\tau)^{p-1}, & j=i,\\
-1, & j \in N_{\Omega_h}(i)\cap \Omega_h^\circ,\\
0, & \text{otherwise}.
\end{cases}
\end{align*}
In particular, we have
\begin{equation*}
G^\tau (G^\tau)^T \succeq
\begin{bmatrix}
I & \bnz\\
A^\tau[\Omega_h^\circ][\partial\Omega_h] & -I
\end{bmatrix}
\begin{bmatrix}
I & \bnz\\
A^\tau[\Omega_h^\circ][\partial\Omega_h] & -I
\end{bmatrix}^{T} \eqqcolon MM^T.
\end{equation*}
Noting that $A^\tau[\Omega_h^\circ][\partial\Omega_h]$ is actually independent of $\tau$, we conclude that there exists $\gamma_G \in (0, 1]$ such that $G^{\tau} (G^{\tau})^{T} \succeq \gamma_G I$ for any $\tau \geq 0$, verifying Assumption \ref{as42}. Assumption \ref{as45} is satisfied by the $40\times 40$ grid graph, whose neighborhoods and overlapping subgraphs grow polynomially with the graph distance and the overlap size $b$.
Assumptions \ref{as44}, \ref{as51}, and \ref{as62} should be understood locally near the strict local solution $(\bxs, \bls)$. After discretization, both the objective and constraint functions are polynomials of the discrete variables. Hence, once the iterates enter a neighborhood of $(\bxs,\bls)$, the iterates and line-search trial points remain in a compact set on which the blocks of $\hHt$, $H^\tau$, and $G^\tau$ are uniformly bounded, while $G_{ij}(\bx)$ and $H_{ij}(\bx,\bl)$ are Lipschitz continuous. Finally, Assumptions \ref{as41} and \ref{as61} are enforced algorithmically through the Hessian modification $\hat H^\tau$ in FOGD.

\vskip4pt
\noindent$\bullet$ \textbf{ADMM.}
For comparison, we also consider ADMM for the same discretized problem. Let $V_\ell$, $\ell=1,\ldots,5$, denote the five disjoint strips above, and let $\mW_\ell = N_{\Omega_h}[V_\ell]$ be the one-hop enlargement of $V_\ell$.
For each $\ell=1,\ldots,5$, we introduce local copies $(\bz^{(\ell)},\bu^{(\ell)})$ on $\mW_\ell$, global consensus variables $(\bar{\bz},\bar{\bu})$ on the full grid, the local objective contribution $f_\ell(\bz^{(\ell)},\bu^{(\ell)})$ on $V_\ell$, and the local feasible set $\mathcal X_\ell$ obtained by imposing the discretized PDE equations on the core nodes in $V_\ell$ using the local variables on $\mW_\ell$. This gives the consensus reformulation:
\begin{align*}
\min_{\{(\bz^{(\ell)},\bu^{(\ell)})\}_{\ell=1}^5,\bar{\bz},\bar{\bu}}\quad
& \sum_{\ell=1}^5 f_\ell(\bz^{(\ell)},\bu^{(\ell)}), \\
\text{s.t.}\quad\quad\quad
& (\bz^{(\ell)},\bu^{(\ell)}) \in \mathcal X_\ell,\qquad \ell=1,\ldots,5,\\
& \bz_i^{(\ell)}=\bar \bz_i,\quad \bu_i^{(\ell)}=\bar \bu_i,\qquad
i\in \mW_\ell,\ \ell=1,\ldots,5.
\end{align*}
ADMM then alternates between five independent local constrained nonlinear solves, a consensus averaging step for $(\bar{\bz},\bar{\bu})$, and dual updates associated with the consensus constraints. In the comparison below, ADMM uses the same five-strip partition and the same initializations as FOGD. For this problem, we use $\sigma$ to denote the augmented Lagrangian penalty parameter in ADMM that applies to both $\bz$ and $\bu$.

\begin{figure}[thb!]
\centering
\begin{subfigure}[b]{0.32\textwidth}
\centering
\includegraphics[width=\textwidth]{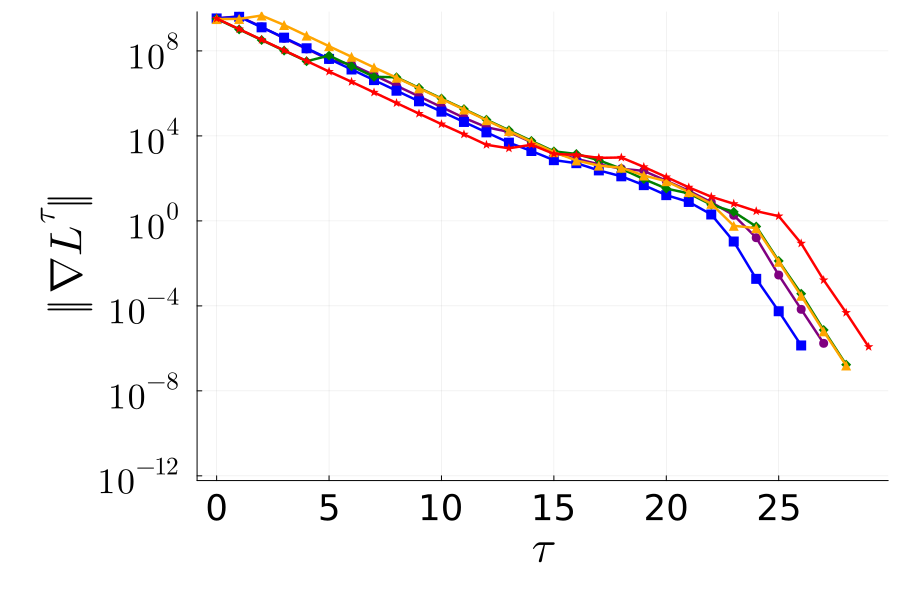}\caption{$b=1$}\label{fig33a}
\end{subfigure}
\hfill
\begin{subfigure}[b]{0.32\textwidth}
\centering
\includegraphics[width=\textwidth]{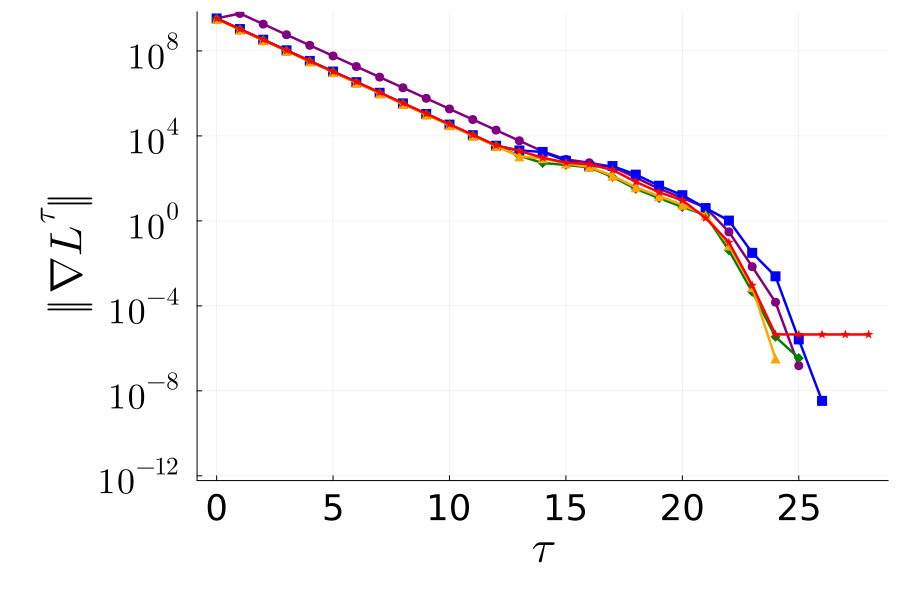}\caption{$b=2$}\label{fig33b}
\end{subfigure}
\hfill
\begin{subfigure}[b]{0.32\textwidth}
\centering
\includegraphics[width=\textwidth]{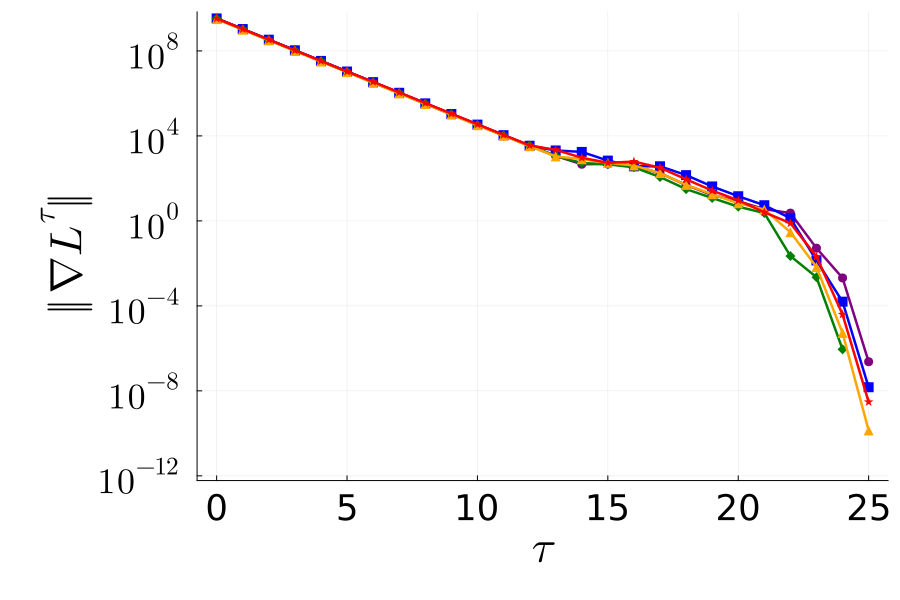}\caption{$b=4$}\label{fig33c}
\end{subfigure}
\begin{subfigure}[b]{0.32\textwidth}
\centering
\includegraphics[width=\textwidth]{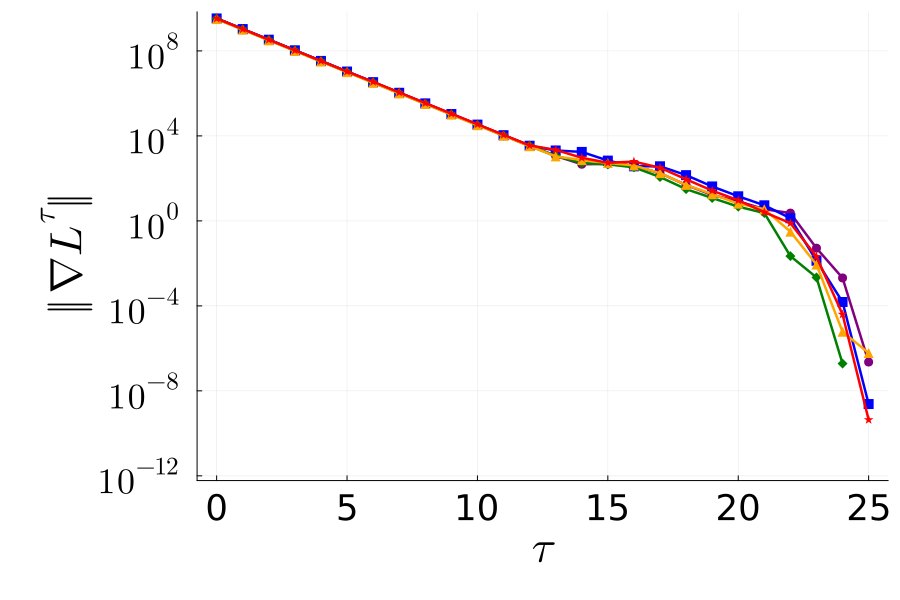}\caption{$b=6$}\label{fig33d}
\end{subfigure}
\hfill
\begin{subfigure}[b]{0.32\textwidth}
\centering
\includegraphics[width=\textwidth]{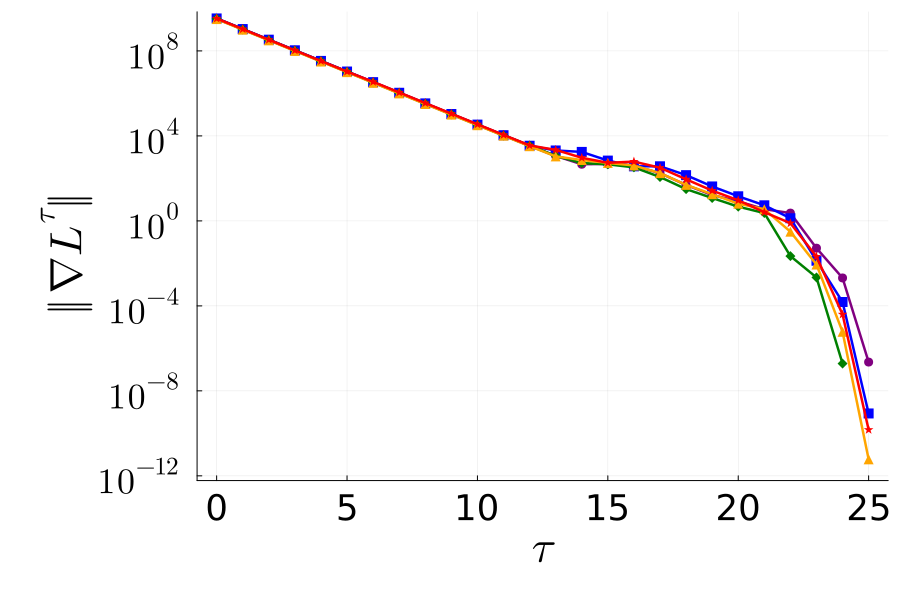}\caption{$b=8$}\label{fig33e}
\end{subfigure}
\hfill
\begin{subfigure}[b]{0.32\textwidth}
\centering
\includegraphics[width=\textwidth]{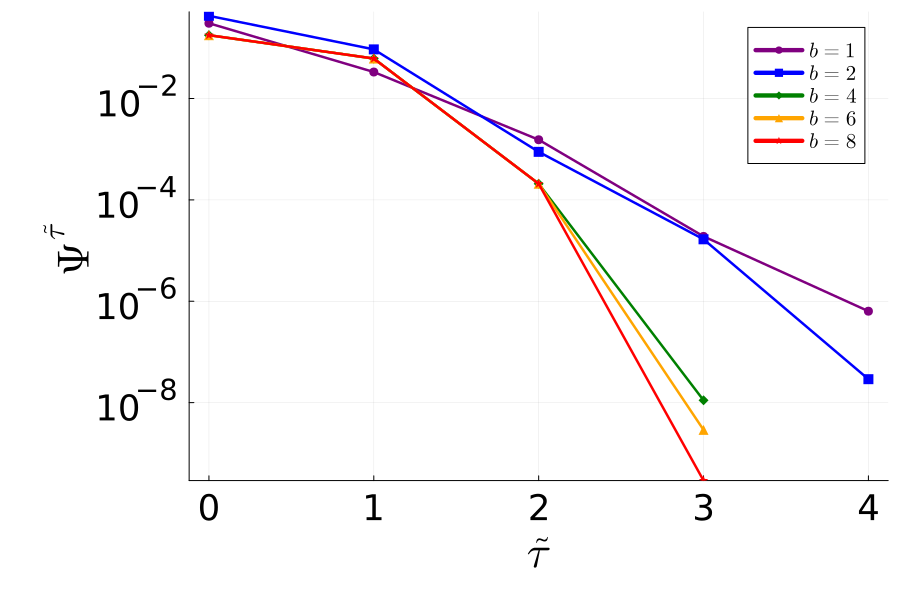}\caption{Local convergence}\label{fig33f}
\end{subfigure}
\setlength{\belowcaptionskip}{-10pt}
\vskip-10pt
\caption{Global and local convergence plots of FOGD for the semilinear elliptic PDE problem. The global convergence plots are displayed in \cref{fig33a}--\ref{fig33e}, where we plot the KKT residual of each iteration. In each figure, five lines with different colors represent five initializations, and they all converge, indicating that FOGD converges for any initializations. The local convergence plot is displayed in \cref{fig33f}, where we show the nodewise error $\Psi^{\tilde{\tau}}$ for different overlap sizes $b$, with $\tilde{\tau}=0$ marking the first iterate entering the local neighborhood. On this plot, five lines represent five overlap sizes. We observe that $\Psi^{\tilde{\tau}}$ decays at least linearly with respect to $\tilde{\tau}$, and a larger $b$ leads to a faster convergence rate.}\label{fig33}
\end{figure}

\begin{table}[thb!]
\centering
\footnotesize
\renewcommand{\arraystretch}{1.08}
\caption{Global convergence summary for FOGD and ADMM on Problem \eqref{eqn31}. The reported quantities are averages over five random initializations. For FOGD, parameter refers to the overlap size $b$, while for ADMM it refers to the penalty parameter $\sigma$.}
\label{tab61}
\begin{tabular}{ccccc}	
\hline
Method & Parameter & Final KKT ($10^{-7}$) & Flops/Iter ($10^{5}$) & Time (s) \\
\hline
\multirow{5}{*}{FOGD}
& 1   & 4.97   & 9.65 & 0.69 \\
& 2   & 9.51   & 9.62 & 0.75 \\
& 4   & 1.49   & 9.61 & 0.86 \\
& 6   & 2.05   & 9.61 & 1.11 \\
& 8   & 0.84   & 9.61 & 1.49 \\
\hline
\multirow{5}{*}{ADMM}
& 0.01  & 307.76 & 49.99 & 301.97 \\
& 0.1   & 92.57  & 49.99 & 287.71 \\
& 1     & 17.52  & 49.97 & 154.70 \\
& 10    & 9.65   & 49.87 & 37.20 \\
& 100   & 9.98   & 49.99 & 237.49 \\
\hline
\end{tabular}		
\end{table}
}

\vskip4pt
\noindent$\bullet$ \textbf{Results.}
Our results are summarized in \cref{fig33,tab61}. For the global convergence 
experiments, both FOGD and ADMM are run for at most 5000 iterations, with each run terminated earlier if a prescribed stopping criterion is satisfied. As shown in \cref{fig33a}--\ref{fig33e} and \cref{tab61}, FOGD with different overlap sizes converges for all five random initializations, validating our global convergence result (cf. \cref{thm52}).
From \cref{fig33f}, we observe that the nodewise error $\Psi^{\tau}$ decays at least linearly with respect to $\tau$, and that the convergence rate increases as the overlap size $b$ increases. This observation is consistent with our local convergence result (cf. \cref{lm62}).

{\red
We also compare FOGD and ADMM in terms of global convergence behavior, flops per iteration, and running time. As reported in \cref{tab61}, FOGD consistently achieves smaller KKT residuals with fewer flops per iteration and shorter running times across the tested overlap sizes. In contrast, ADMM is considerably more sensitive to the choice of $\sigma$. Among the tested penalty parameters, the best ADMM runs are obtained with $\sigma=10$ and $100$; however, these runs still require substantially more computation time, and their flops per iteration are about five times larger than those of FOGD. Moreover, even under these better parameter choices, the final KKT residuals achieved by ADMM are not better than those obtained by FOGD with moderate or large overlap sizes. Overall, for this problem, FOGD demonstrates both more robust global convergence behavior and superior computational efficiency compared to ADMM.
}

{\red 

\subsection{Boundary Heating Problem}\label{sec6:heating}

In this example, we let $\Omega \subseteq \mR^3$ be a three-dimensional domain and consider the following boundary heating problem in \cite{Curtis2012note}:
\begin{subequations}\label{eqn32}
\begin{align}
\min_{\bz, \bu}\;\; & f(\bz, \bu) = \sum_{j=1}^{N_S}\int_{\Omega_j} (\bz(\bomega) - z_j^{\min})^2 d\bomega + \kappa \int_{\partial\Omega} \bu(\xi)^2 da(\xi),\\
\text{s.t.} \;\; & g(\bz, \bu) = \begin{cases}
-\Delta \bz = 0\quad \text{in $\Omega$},\\
\frac{\partial \bz}{\partial n} = \chi\cdot (\bu - \bz^4)\quad \text{on $\partial\Omega$},
\end{cases}
\end{align}
\end{subequations}
where $\bz:\Omega\rightarrow \mR$ is the temperature field, $\bu:\partial\Omega\rightarrow \mR$ is the boundary heating input, $\Omega_j \subseteq \Omega$, $j=1,\ldots,N_S$, are the target sub-regions, $z_j^{\min}$ are the associated target temperatures, $\chi > 0$ is Stefan's constant, $\frac{\partial \bz}{\partial n}$ denotes the outward normal derivative on $\partial\Omega$, and $da$ is the surface measure on $\partial\Omega$. 
We let $\chi = 1$, $N_S=2$, $\Omega = (0,1)^3$, $\Omega_1 = [0.1,0.2] \times [0.05,0.3] \times [0,0.1]$, $\Omega_2 = [0.8,1] \times [0.75,1] \times [0.7,1]$, $z_1^{\min} = 2.5$, $z_2^{\min} = 2$, and $\kappa = 10^{-3}$.
To solve~Problem \eqref{eqn32}, we discretize $\Omega$ by a uniform Cartesian grid with mesh size $h = 1/30$, denoted by $\Omega_h = h\cdot \{0,\ldots,30\}^3$. Thus, the full grid contains $31\times 31\times 31$ nodes. We let $\Omega_h^\circ$ and $\partial\Omega_h$ denote the interior and boundary nodes of $\Omega_h$, respectively. The discrete state variables are $\bz=\{\bz_i\}_{i\in\Omega_h}$, and the discrete control variables are $\bu=\{\bu_i\}_{i\in\partial\Omega_h}$, since the control acts~only~on the boundary.

We approximate the objective in \eqref{eqn32} by a simple node-based quadrature rule. Specifically, each grid node is assigned a volume weight for the state-tracking term, while each boundary node is assigned a surface-area weight for the boundary-control regularization term. This yields
\begin{equation*}
f_h(\bz,\bu) = \sum_{j=1}^{N_S}\sum_{i\in\Omega_h} w_i^{\Omega}\mathbf 1_{\{\bomega_i\in\Omega_j\}} (\bz_i-\bz_j^{\min})^2 + \kappa\sum_{i\in\partial\Omega_h}w_i^{\partial}\bu_i^2,
\end{equation*}
where $\bomega_i$ is the coordinate of node $i$, $w_i^\Omega$ is its volume weight, and $w_i^\partial$ is its boundary surface-area weight.
We discretize the PDE constraints with standard finite differences. At an interior node and a boundary node, respectively, we impose
\begin{align*}
\bc_i^{\mathrm{int}}(\{\bz_j\}_{j\in N_{\Omega_h}[i]}) & =
\frac{1}{h^2} \left( 6z_i-\sum_{j\in N_{\Omega_h}(i)} z_j \right) =0,
&& i\in\Omega_h^\circ,\\
\bc_i^{\mathrm{bd}}(\{\bz_j\}_{j\in N_{\Omega_h}[i]}, \{\bu_j\}_{j\in N_{\Omega_h}[i]}) & = \frac{1}{|\mathcal A(i)|} \sum_{q\in\mathcal A(i)}
\frac{z_i-z_{i^{-}(q)}}{h} - \chi(u_i-z_i^4) =0,
&& i\in\partial\Omega_h.	
\end{align*}
Here, $N_{\Omega_h}(i) = N_{\Omega_h}[i]\backslash\{i\}$ is the open nearest-neighbor set, $\mathcal A(i)$ collects the boundary-face directions incident to node $i$, and $i^{-}(q)$ is the inward neighbor in direction $q$. The boundary equation averages the corresponding one-sided normal differences. Hence, the discretization gives one equality constraint per grid node: an interior equation on $\Omega_h^\circ$ and a boundary heating equation on $\partial\Omega_h$. We form the graph by connecting nearest-neighbor grid nodes and split it into three slabs in the first coordinate direction, with $V_\ell=I_\ell\times \{0,\ldots,30\}^2$ for $\ell=1,2,3$, where $I_1=\{0,\ldots,9\}$, $I_2=\{10,\ldots,19\}$, and $I_3=\{20,\ldots,30\}$.

\vskip4pt
\noindent$\bullet$ \textbf{FOGD.}
We use the same multi-start and Armijo backtracking setup as in Section \ref{sec6:pde}, with $b\in\{1,2,4,6,8\}$, $\eta_1=\eta_2=0.05$, $\beta=0.05$, $\mu=3$, initial stepsize $1$, and shrinkage factor $0.9$. The five initial points are sampled independently from $\bz_i^0\sim\mathrm{Uniform}(0,2)$, $\bu_i^0\sim\mathrm{Uniform}(0,10)$, and $\bl_i^0\sim\mathrm{Uniform}(-1,1)$, with $\bu_i^0$ defined only on $\partial\Omega_h$ while the other variables defined on $\Omega_h$. Each initial point is solved using all five overlaps. For local convergence, we report the nodewise error relative to the final iterate of the same trajectory after it enters the local neighborhood. Same as Section \ref{sec6:pde}, we re-index the iterations after this entrance time.

We next discuss why the assumptions hold for this example for the same reasons as in Section \ref{sec6:pde}. Let us order the constraints as $(\{\bc_i^{\mathrm{bd}}\}_{i\in\partial\Omega_h},\{\bc_i^{\mathrm{int}}\}_{i\in\Omega_h^\circ})$ and the variables as $(\bu,\bz[\Omega_h^\circ],\bz[\partial\Omega_h])$. Then the Jacobian has the block form
\begin{equation*}
G^\tau = \begin{bmatrix}
\\[-0.6cm]
-\chi I & D[\partial\Omega_h][\Omega_h^\circ] & \vdots & A^\tau[\partial\Omega_h][\partial\Omega_h]\\
\bnz & L[\Omega_h^\circ][\Omega_h^\circ] & \vdots & L[\Omega_h^\circ][\partial\Omega_h]
\end{bmatrix} \eqqcolon \begin{bmatrix}
\\[-0.62cm]
M & \vdots & Q^\tau
\end{bmatrix}.
\end{equation*}
Here, $L[\Omega_h^\circ][\Omega_h^\circ]$ and $L[\Omega_h^\circ][\partial\Omega_h]$ denote the interior-interior and interior-boundary blocks of the seven-point finite-difference operator, respectively, while $D[\partial\Omega_h][\Omega_h^\circ]$ collects the dependence of the one-sided boundary differences on the interior state variables. The block $A^\tau[\partial\Omega_h][\partial\Omega_h]$ contains the boundary-state derivatives of the boundary equations. More specifically, for $i\in\partial\Omega_h$, its diagonal entry is $1/h+4\chi(z_i^\tau)^3$, while its off-diagonal entries are $-1/(|\mathcal A(i)|h)$ whenever an inward neighbor $i^{-}(q)$ is also a boundary node. 
Since $\chi=1$ and $L[\Omega_h^\circ][\Omega_h^\circ]$ is the positive definite discrete Dirichlet Laplacian, the block matrix $M$ is invertible and independent of $\tau$. Consequently, $G^\tau(G^\tau)^T=MM^T+Q^\tau(Q^\tau)^T\succeq MM^T\succeq \gamma_G I$ for some $\gamma_G>0$, verifying Assumption \ref{as42}. 
Assumption \ref{as45} follows from the polynomial growth property of the three-dimensional Cartesian grid graph. Finally, Assumptions \ref{as44}, \ref{as51}, and \ref{as62} hold locally near the strict local solution because the discretized objective and constraints are smooth finite-dimensional functions involving only polynomial nonlinearities. Assumptions \ref{as41} and \ref{as61} are enforced by the Hessian modification $\hat H^\tau$.

\vskip4pt
\noindent$\bullet$ \textbf{ADMM.}
For comparison, we consider ADMM applied to a consensus reformulation of the same finite-difference discretized problem over the same three subgraphs. As in Section \ref{sec6:pde}, each subproblem introduces local copies of the state variables and boundary-control variables, and consensus constraints are imposed on the shared nodes. ADMM uses the same decomposition and initializations as FOGD. We use $\sigma$ to denote the ADMM penalty parameter.

\vskip4pt
\noindent$\bullet$ \textbf{Results.}
Our results are summarized in \cref{fig34,tab62}. For global convergence, both FOGD and ADMM are run for at most $3000$ iterations, with each run terminated earlier if a prescribed stopping criterion is satisfied. As shown in 
\cref{fig34a}--\ref{fig34e}, FOGD consistently reduces the KKT residual from all five random initializations for every chosen overlap size. The final KKT residuals reported in \cref{tab62} further demonstrate that all tested FOGD configurations reach small KKT residuals, which is consistent with the global convergence guarantee in \cref{thm52}. 
From \cref{fig34f}, we observe that after the iterates enter the local neighborhood, the nodewise error $\Psi^{\tau}$ exhibits at least linear convergence. The curves are not strictly ordered by the overlap size: the $b=1$ run enters the neighborhood much later than the $b=2$ and $b=4$ runs, so its re-indexed curve mainly reflects a short late-stage segment relative to its own final iterate. In contrast, the $b=6$ and $b=8$ runs enter the local neighborhood much earlier and reach the $10^{-3}$ level in substantially fewer steps. This behavior remains consistent~with \cref{lm62}.

We also compare FOGD and ADMM based on the results in \cref{tab62}. The final KKT column reveals a clear performance gap: every FOGD configuration attains a smaller final KKT residual than every ADMM configuration, even though ADMM is tested over several penalty parameters. Among the ADMM runs, $\sigma=0.1$ yields the smallest final KKT residual and the shortest average runtime; nevertheless, its final KKT residual is still approximately one order of magnitude larger than the FOGD residual at $b=1$. For FOGD, increasing the overlap size increases the flops per iteration, as expected, while the configurations with $b=4,6,8$ all attain final KKT residuals on the order of $10^{-4}$. In particular, $b=8$ achieves the shortest average runtime among the tested FOGD configurations. Overall, the results indicate that FOGD is less sensitive to parameter choices and more effective for this boundary heating problem.

\begin{figure}[t!]
\centering
\begin{subfigure}[b]{0.32\textwidth}
\centering
\includegraphics[width=\textwidth]{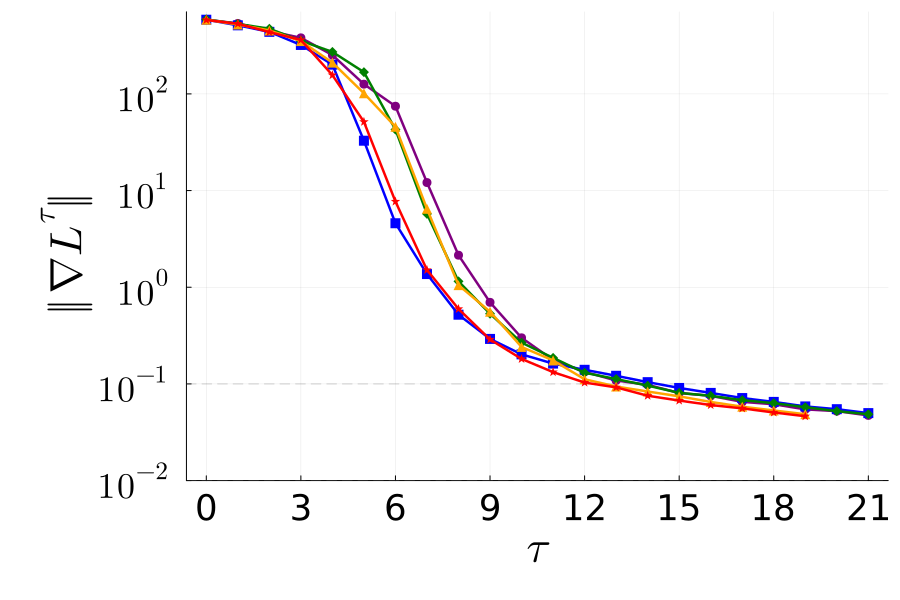}
\caption{$b=1$}
\label{fig34a}
\end{subfigure}
\hfill
\begin{subfigure}[b]{0.32\textwidth}
\centering
\includegraphics[width=\textwidth]{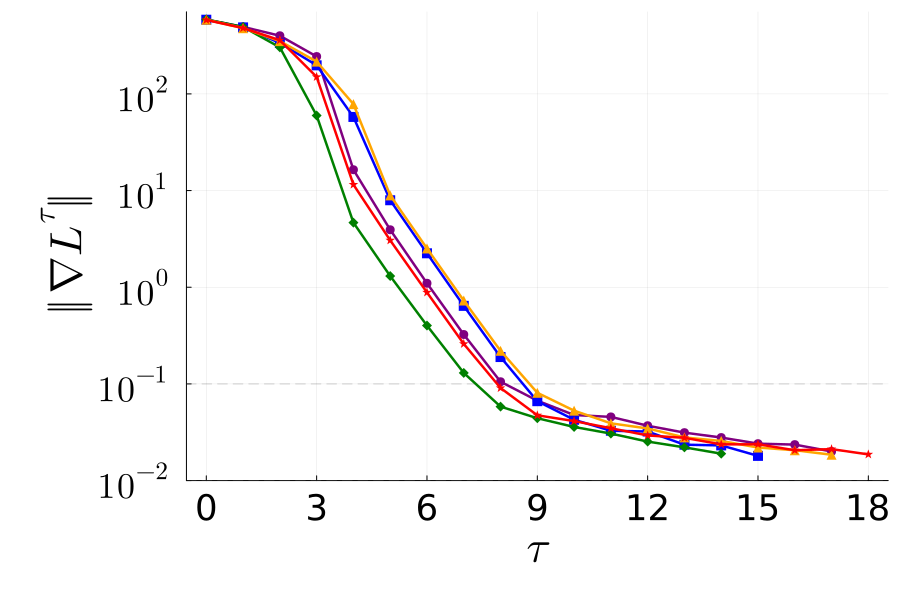}
\caption{$b=2$}
\label{fig34b}
\end{subfigure}
\hfill
\begin{subfigure}[b]{0.32\textwidth}
\centering
\includegraphics[width=\textwidth]{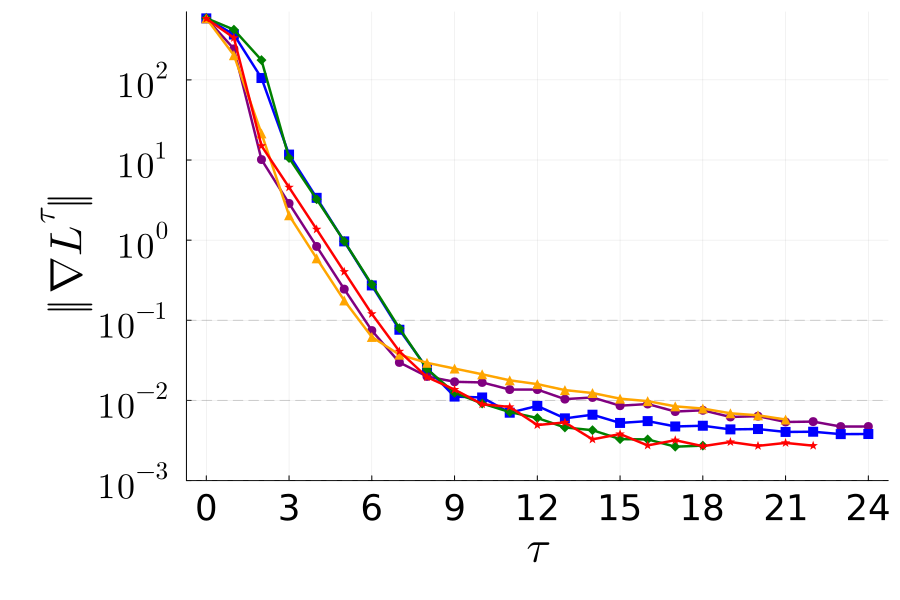}
\caption{$b=4$}
\label{fig34c}
\end{subfigure}
\begin{subfigure}[b]{0.32\textwidth}
\centering
\includegraphics[width=\textwidth]{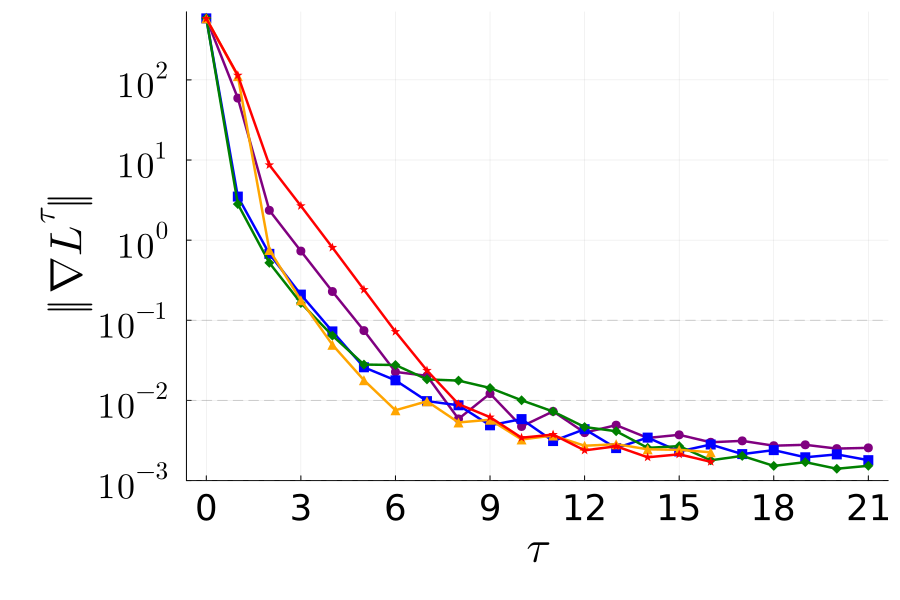}
\caption{$b=6$}
\label{fig34d}
\end{subfigure}
\hfill
\begin{subfigure}[b]{0.32\textwidth}
\centering
\includegraphics[width=\textwidth]{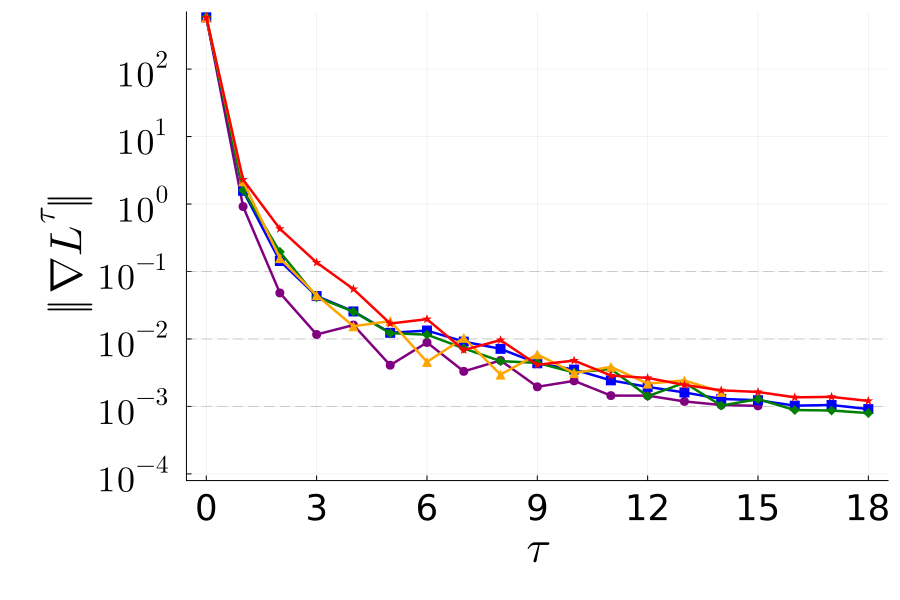}
\caption{$b=8$}
\label{fig34e}
\end{subfigure}
\hfill
\begin{subfigure}[b]{0.32\textwidth}
\centering
\includegraphics[width=\textwidth]{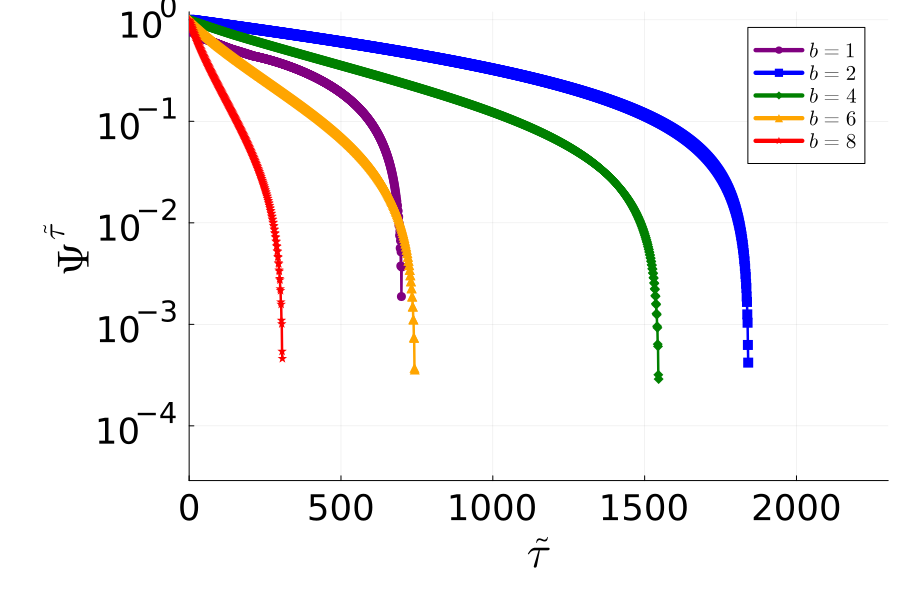}
\caption{Local convergence}
\label{fig34f}
\end{subfigure}
\setlength{\belowcaptionskip}{-10pt}
\vskip-10pt
\caption{Global and local convergence plots of FOGD for the boundary heating problem. The global convergence plots are displayed in \cref{fig34a}--\ref{fig34e}, where we plot the KKT residual of each iteration. In each figure, five lines with different colors represent five initializations, and they all converge, indicating that FOGD converges for any initializations. The local convergence plot is displayed in \cref{fig34f}, where we show the nodewise error $\Psi^{\tilde{\tau}}$ for different overlap sizes $b$, with $\tilde{\tau}=0$ marking the first iterate entering the local neighborhood. We observe that $\Psi^{\tilde{\tau}}$ decays at least linearly with respect to $\tilde{\tau}$, and a larger $b$ leads to a faster convergence rate.}\label{fig34}		
\end{figure}

\begin{table}[t!]
\centering
\footnotesize
\renewcommand{\arraystretch}{1.08}
\caption{Global convergence summary for FOGD and ADMM on Problem \eqref{eqn32}. The reported quantities are averages over five random initializations. For FOGD, parameter refers to the overlap size $b$, while for ADMM it refers to the penalty parameter $\sigma$.}
\label{tab62}
\begin{tabular}{ccccc}	
\hline
Method & Parameter & Final KKT ($10^{-5}$) & Flops/Iter ($10^{13}$) & Time ($10^3$ s) \\
\hline
\multirow{5}{*}{FOGD}
& 1   & 44.04 & 3.72  & 7.07 \\
& 2   & 28.05 & 5.19  & 8.73 \\
& 4   & 9.98  & 9.37  & 12.28 \\
& 6   & 9.97  & 15.52 & 10.48 \\
& 8   & 9.96  & 24.03 & 4.66 \\
\hline
\multirow{5}{*}{ADMM}
& 0.01  & 1964.73   & 3.72 & 8.79 \\
& 0.1   & 459.63    & 3.72 & 7.37 \\
& 1     & 1085.48   & 3.72 & 10.37 \\
& 10    & 10758.70  & 3.72 & 8.78 \\
& 100   & 107482.21 & 3.72 & 10.86 \\
\hline
\end{tabular}		
\end{table}

}

\section{Conclusion}\label{sec7}

In this paper, we introduced a fast overlapping graph decomposition (FOGD) method for solving graph-structured nonlinear programs. We embedded the OGD technique into the SQP framework to compute approximate Newton directions in a parallel environment, and selected suitable stepsizes by performing line search on an exact augmented Lagrangian merit function. We established the global and local linear convergence guarantees for FOGD, and showed that the linear rate improves exponentially in terms of the overlap size. FOGD balances the flexibility and efficiency of centralized optimization solvers, addresses the slow convergence disadvantage of existing parallel methods with exclusive decomposition, and complements the overlapping Schwarz schemes that have expensive computational costs on local machines and lack global convergence guarantees. Numerical results on {\red two} PDE-constrained problems, a semilinear elliptic problem and {\red a boundary heating problem}, verified our theoretical findings.

Further extensions of FOGD are a promising research direction, inspired by empirical evidence that the Schwarz schemes have exhibited superior performance compared to the popular parallel solver, the alternating direction method of multipliers (ADMM) \cite{Shin2020Decentralized}. It would be worthwhile to incorporate OGD techniques into ADMM and other advanced SQP frameworks to enhance their performance under exclusive graph decomposition.

\appendix
\numberwithin{equation}{section}
\numberwithin{theorem}{section}

\section{Glossary of Main Notation}
To improve readability, especially in the proofs of Sections \ref{sec3} and \ref{sec4}, Table \ref{p1_tb1} provides a glossary of the main notation used throughout the paper. It collects the graph-, partition-, and boundary-related symbols that appear most frequently in our analysis.

\begin{table}[t]
\centering
\footnotesize
\setlength{\tabcolsep}{4pt}
\renewcommand{\arraystretch}{1.15}
\begin{tabularx}{\textwidth}{|>{\centering\arraybackslash}m{2.8cm}||>{\raggedright\arraybackslash}X|}
\hline
$N^b_{\mG}[\mW]$ &
$\{i\in \mV:d_{\mG}(i,\mW)\leq b\}$, the $b$-hop neighborhood of $\mW$.\\
\hline
$N_{\mG}[\mW]$ &
$N_{\mG}[\mW] = N^1_{\mG}[\mW]$, with the superscript omitted when $b=1$.\\
\hline
\multirow{2}{*}{$N_{\mG}(\mW)$} &
$N_{\mG}(\mW) \coloneqq N_{\mG}[\mW]\backslash \mW$, the neighborhood of $\mW$ excluding itself (i.e., the open neighborhood).\\
\hline
\multirow{2}{*}{$d_{\mG}(i,\mW)$} &
the number of edges in the shortest path connecting node $i$ to the set $\mW$ on the graph $\mG$.\\
\hline
\multirow{2}{*}{$\{\mV_{\ell}\}_{\ell\in[M]}$} &
a disjoint partition of the full node set $\mV$, where $M$ is the number of subsets (or subdomains).\\
\hline
$\mV_{\ell}$ &
the $\ell$-th disjoint subset of the full node set $\mV$.\\
\hline
\multirow{2}{*}{$\mW_{\ell}$} &
an overlapping subdomain associated with $\mV_{\ell}$, satisfying
$N^b_{\mG}[\mV_{\ell}] \subseteq \mW_{\ell}$; in the simple case, one may take
$\mW_{\ell}=N^b_{\mG}[\mV_{\ell}]$.\\
\hline
$\tmW_{\ell}$ &
$\tmW_{\ell}\coloneqq N_{\mG}\!\left(N_{\mG}(\mW_{\ell})\right)\cap \mW_{\ell}$,
the set of internal boundary nodes of $\mW_{\ell}$.\\
\hline
$\hat{\mW}_{\ell}$ &
$\hat{\mW}_{\ell}\coloneqq N_{\mG}(\mW_{\ell})\cup \tmW_{\ell}$,
the set of internal and external boundary nodes of $\mW_{\ell}$.\\
\hline
\multirow{2}{*}{$\bar{\mW}_{\ell}$} &
$\bar{\mW}_{\ell} \coloneqq N_{\mG}(\mW_{\ell})\cup N_{\mG}\!\left(N_{\mG}[\mW_{\ell}]\right)$,
the set of external boundary nodes of $\mW_{\ell}$ within depth two.\\
\hline
\multirow{2}{*}{$(\bto_{\ell},\btze_{\ell})$ and $(\bo,\bze)$} &
subdomain and full-domain primal-dual variables used in the decomposition or composition operators; see Definition \ref{df22}.\\
\hline	
\end{tabularx}
\caption{Glossary of Main Notation.}
\label{p1_tb1}
\end{table}

\section{Proof of \cref{lm41}}\label{alm41}
Throughout the proof, we omit the iteration index $\tau$ for the simplicity of notation. We first prove (a). We note that $G_{\ell}$ is a submatrix of $G = \nabla \bc$, and for any nodes $i\in \mW_\ell\backslash \tmW_{\ell}$ and $j\in \mV\backslash\mW_{\ell}$, $[G_\ell]_{i,j} = G_{i,j} = \bnz$ since $d_G(i,j)>1$. Thus, we permute $G$ as
\begin{equation*}
G_P \coloneqq \begin{bmatrix}
G_{\ell} & \bnz\\
J_1 & J_2
\end{bmatrix},
\end{equation*}
where $J_1$ and $J_2$ are nonzero matrices. Since $G_{\ell}G_{\ell}^T$ is the diagonal submatrix of $G_PG_P^T$, and $G_PG_P^T\succeq \gamma_G I$ by Assumption \ref{as42}, we complete the proof of (a). For the proof of (b), we apply Assumptions \ref{as41}--\ref{as45}, \eqref{as44-e}, and \cite[Lemma 5.6]{Shin2022Exponential}, and have
\begin{equation}\label{nequ:1}
\hH + \mu G ^TG \succeq \frac{\gamma_{H}}{2} I \quad \text{ provided }\quad \mu \geq \frac{2\hU^2/\gRH+\gRH+\hU}{\gamma_G}.
\end{equation}
The above condition on $\mu$ is implied by the statement of the lemma. Because $\hH[\mW_{\ell}][\mW_{\ell}] + \mu (G[\mV][\mW_{\ell}])^TG[\mV][\mW_{\ell}]$ is the diagonal submatrix of $\hH + \mu G^TG$, we know \eqref{nequ:1} holds for this diagonal submatrix as well. Furthermore, for any $\bto_\ell$ such that $G_\ell\bto_\ell = \bnz$, we have
\begin{align*}
\bto_\ell^T\hH_\ell^\mu \bto_\ell & = \bto_\ell^T\hH[\mW_{\ell}][\mW_{\ell}] \bto_\ell + \mu \|G[\hat{\mW}_{\ell}][\mW_{\ell}] \bto_\ell\|^2 \\
& = \bto_\ell^T\hH[\mW_{\ell}][\mW_{\ell}] \bto_\ell + \mu \|G[\hat{\mW}_{\ell}][\mW_{\ell}] \bto_\ell\|^2 + \mu\|G_\ell\bto_\ell\|^2\\
& = \bto_\ell^T\hH[\mW_{\ell}][\mW_{\ell}] \bto_\ell + \mu \|G[\hat{\mW}_{\ell}][\mW_{\ell}] \bto_\ell\|^2 + \mu\|G_\ell\bto_\ell\|^2 + \mu\|G[\mV\backslash N_{\mG}[\mW_{\ell}]][\mW_{\ell}] \bto_{\ell}\|^2\\
& = \bto_\ell^T\hH[\mW_{\ell}][\mW_{\ell}] \bto_\ell + \mu\|G[\mV][\mW_{\ell}]\bto_\ell\|^2 \geq \gamma_{H} \|\bto_{\ell}\|^2/2,
\end{align*}
where the third equality is due to the fact that $G[\mV\backslash N_{\mG}[\mW_{\ell}]][\mW_{\ell}] = \bnz$. This completes the proof of (b). With (a), (b), and \cite[Theorem 16.2]{Nocedal2006Numerical}, we know $\mS\mP_{\ell}^{\mu}(\bd_{\ell})$ has a unique global~solution for any boundary variables $\bd_{\ell}$. We complete the proof.

\section{Proof of \cref{lm51}}\label{alm51}
By Assumptions \ref{as45}, \ref{as51}, and \cite[Lemma 5.10]{Shin2022Exponential}, we know for the constant $\hU$ in \eqref{as44-e} that
\begin{equation*}
\|H(\bx,\bl)\|\le \hU \quad \text{ and }\quad \|G(\bx)\|\le \hU,\quad\quad \forall\,\bx\in\mX,\,\bl\in\Lambda.
\end{equation*}
This proves \eqref{eq53}. Furthermore, let $Z$ be the matrix whose columns are orthonormal and span the null space $\{\bo:G\bo=\bnz\}$. Then, we can verify that
\begin{align*}
\mathcal{B} & = \begin{bmatrix}
\hat{H} & G^T\\
G & \boldsymbol{0}\\
\end{bmatrix}^{-1} = \begin{bmatrix}
\mathcal{B}_1 & \mathcal{B}_2^T\\
\mathcal{B}_2 & \mathcal{B}_3\\
\end{bmatrix},\\
\mathcal{B}_1 &= Z(Z^T \hH Z)^{-1}Z^T,\quad \mathcal{B}_2 = (GG^T)^{-1}G(I-\hH Z(Z^T\hH Z)^{-1}Z^T),\\
\mathcal{B}_3 &= (GG^T)^{-1}G(\hH Z(Z^T\hH Z)^{-1}Z^T\hH-\hH)G^T(GG^T)^{-1}.
\end{align*}
By Assumptions \ref{as41}--\ref{as45} and \eqref{as44-e}, we know
\begin{equation*}
\|\mathcal{B}_1\| \leq \frac{1}{\gamma_{H}},\quad \|\mathcal{B}_2\| \leq \frac{1}{\sqrt{\gamma_G}}\left(1 + \frac{\hU}{\gH}\right),\quad \|\mathcal{B}_3\|\leq \frac{1}{\gamma_G}\left( \hU + \frac{\hU^2}{\gH}\right).
\end{equation*}
Therefore, $\|\mathcal{B}\| \leq \max\{\|\mathcal{B}_1\|, \|\mathcal{B}_3\|\} + \|\mathcal{B}_2\|\leq \frac{2\hU^2}{\gamma_G\gamma_{H}} + \frac{2\hU}{\sqrt{\gamma_G}\gamma_{H}} \leq \frac{4\hU^2}{\gamma_{H}\gamma_G}$. This proves \eqref{eq54}.

\section{Proof of \cref{thm61}}\label{athm61}
\hskip-3pt For two positive sequences $\{a_\tau\}$ and $\{b_\tau\}$, we denote $a_\tau = o(b_\tau)$ if $a_\tau/b_\tau\rightarrow 0$. By \eqref{eq34}, it suffices to show for sufficiently large $\tau$ that
\begin{equation*}
\mLe(\bxt+\tD\bxt, \blt+\tD\blt) \le \mLe(\bxt, \blt) + \beta \begin{bmatrix}
\nbx \mLe^{\tau} \\
\nbl \mLe^{\tau} \\
\end{bmatrix}^T
\begin{bmatrix}
\tD \bxt\\
\tD \blt\\
\end{bmatrix}.
\end{equation*}
By thrice differentiability of $\{f_i, \bc_i\}_{i\in\mV}$, we know $\nabla^2 \mLe$ is continuously differentiable. Therefore, by direct calculation, we apply Taylor expansion and have
\begin{align*}
\mLe(\bxt+\tD\bxt, \blt+\tD\blt)
& \le \mLe(\bxt, \blt) + \begin{bmatrix}
\nbx \mLe^{\tau} \\
\nbl \mLe^{\tau} \\
\end{bmatrix}^T
\begin{bmatrix}
\tD \bxt\\
\tD \blt\\
\end{bmatrix} \\
& \quad + \frac{1}{2} \begin{bmatrix}
\tD \bxt\\
\tD \blt\\
\end{bmatrix}^T \mHt
\begin{bmatrix}
\tD \bxt\\
\tD \blt\\
\end{bmatrix} + o(\|(\tD \bxt, \tD \blt)\|^2),
\end{align*}
where
\begin{equation*}
\mHt \coloneqq \begin{bmatrix}
\Ht + \eta_1(\Gt)^T \Gt + \eta_2(\Ht)^2 & (I + \eta_2 \Ht)(\Gt)^T\\
\Gt(I+\eta_2 \Ht) & \eta_2 \Gt (\Gt)^T
\end{bmatrix}.
\end{equation*}
Combining the above three displays, we know $\alpha_{\tau}=1$ as long as
\begin{equation}\label{eqD4}
(1-\beta)\begin{bmatrix}
\nbx \mLe^{\tau} \\
\nbl \mLe^{\tau} \\
\end{bmatrix}^T
\begin{bmatrix}
\tD \bxt\\
\tD \blt\\
\end{bmatrix} + \frac{1}{2} \begin{bmatrix}
\tD \bxt\\
\tD \blt\\
\end{bmatrix}^T \mHt
\begin{bmatrix}
\tD \bxt\\
\tD \blt\\
\end{bmatrix} + o(\|(\tD \bxt, \tD \blt)\|^2) \le 0.
\end{equation}
We note that
\begin{align*}
& \begin{bmatrix}
\nbx \mLe^{\tau} \\
\nbl \mLe^{\tau} \\
\end{bmatrix}^T
\begin{bmatrix}
\tD \bxt\\
\tD \blt\\
\end{bmatrix} + \begin{bmatrix}
\tD \bxt\\
\tD \blt\\
\end{bmatrix}^T \mHt
\begin{bmatrix}
\tD \bxt\\
\tD \blt\\
\end{bmatrix}\\
&\quad \overset{\substack{\eqref{eq31}\\\eqref{eq35}}}{=} \begin{bmatrix}
\tD \bxt\\
\tD \blt\\
\end{bmatrix}^T \left\{ \mHt - \begin{bmatrix}
\hHt + \eta_1(\Gt)^T \Gt + \eta_2 \Ht\hHt & (I + \eta_2 \Ht)(\Gt)^T\\
\Gt(I+\eta_2 \hHt) & \eta_2 \Gt (\Gt)^T
\end{bmatrix} \right\}\begin{bmatrix}
\Delta \bxt\\
\Delta \blt\\
\end{bmatrix}\\
& \quad \quad\; + \begin{bmatrix}
\tD \bxt\\
\tD \blt\\
\end{bmatrix}^T\mHt\begin{bmatrix}
\tD \bxt - \Delta \bxt\\
\tD \blt - \Delta \blt\\
\end{bmatrix}\\
&\quad =  \begin{bmatrix}
\tD \bxt\\
\tD \blt\\
\end{bmatrix}^T\begin{bmatrix}
(I + \eta_2H^\tau) (H^\tau - \hHt) & \bnz\\
\eta_2G^\tau(H^\tau - \hHt) & \bnz
\end{bmatrix}\begin{bmatrix}
\Delta \bxt\\
\Delta \blt\\
\end{bmatrix} +  \begin{bmatrix}
\tD \bxt\\
\tD \blt\\
\end{bmatrix}^T\mHt\begin{bmatrix}
\tD \bxt - \Delta \bxt\\
\tD \blt - \Delta \blt\\
\end{bmatrix}\\
&\quad \leq o(\|(\Delta \bxt, \Delta \blt)\|^2) + (1+\delta_{\mu,b})\delta_{\mu,b} \|\mHt\|\cdot \|(\Delta \bxt, \Delta \blt)\|^2\\
& \quad \leq  o(\|(\Delta \bxt, \Delta \blt)\|^2) + 2\delta_{\mu,b}\{2\hU+(\eta_1+2\eta_2)\hU^2\}\|(\Delta \bxt, \Delta \blt)\|^2,
\end{align*}
where the second inequality from the end is due to Assumption \ref{as61} and \eqref{as44-e}, and the last inequality is due to $\|\mHt\| \leq \max(\hU+(\eta_1+\eta_2)\hU^2, \eta_2\hU^2) + (\hU+\eta_2\hU^2)$ and $\delta_{\mu,b}\leq 1$. Plugging the above result into \eqref{eqD4}, we obtain
\begin{align*}
& (1-\beta)\begin{bmatrix}
\nbx \mLe^{\tau} \\
\nbl \mLe^{\tau} \\
\end{bmatrix}^T
\begin{bmatrix}
\tD \bxt\\
\tD \blt\\
\end{bmatrix} + \frac{1}{2} \begin{bmatrix}
\tD \bxt\\
\tD \blt\\
\end{bmatrix}^T \mHt
\begin{bmatrix}
\tD \bxt\\
\tD \blt\\
\end{bmatrix} + o(\|(\tD \bxt, \tD \blt)\|^2)\\
& \leq (0.5-\beta)\begin{bmatrix}
\nbx \mLe^{\tau} \\
\nbl \mLe^{\tau} \\
\end{bmatrix}^T
\begin{bmatrix}
\tD \bxt\\
\tD \blt\\
\end{bmatrix} + \delta_{\mu,b}\{2\hU+(\eta_1+2\eta_2)\hU^2\}\|(\Delta \bxt, \Delta \blt)\|^2 + o(\|(\Delta \bxt, \Delta \blt)\|^2)\\
& \leq \cbr{(0.5-\beta)\cbr{-\frac{\eta_2\gamma_G}{8}+1.05\eta_1\hU^2\delta_{\mu,b}} + 1.05\eta_1\hU^2\delta_{\mu,b}}\|(\Delta \bxt, \Delta \blt)\|^2 + o(\|(\Delta \bxt, \Delta \blt)\|^2)\\
& = \cbr{-\frac{(0.5-\beta)\eta_2\gamma_G}{8} +(1.5-\beta)1.05\eta_1\hU^2\delta_{\mu,b} }\|(\Delta \bxt, \Delta \blt)\|^2 + o(\|(\Delta \bxt, \Delta \blt)\|^2),
\end{align*}
where the second inequality is due to \eqref{eq57}, \eqref{nequ:9}, \eqref{nequ:10}, and the fact that $2\hU + (\eta_1+2\eta_2)\hU^2 \leq 2\hU+ 2/7 + \eta_1\hU^2 \leq 2.5\hU^2+\eta_1\hU^2 \leq 1.05\eta_1\hU^2$. By \eqref{eq61} and noting that the condition of $\delta_{\mu,b}$ implies the condition \eqref{eq55}, we complete the proof.


\bibliographystyle{siamplain}
\bibliography{references}

\vspace{0.1cm}
\begin{flushright}
\scriptsize \framebox{\parbox{0.95\textwidth}{Government License: The submitted manuscript has been created by UChicago Argonne, LLC, Operator of Argonne National Laboratory (``Argonne"). Argonne, a U.S. Department of Energy Office of Science laboratory, is operated under Contract No. DE-AC02-06CH11357.  The U.S. Government retains for itself, and others acting on its behalf, a paid-up nonexclusive, irrevocable worldwide license in said article to reproduce, prepare derivative works, distribute copies to the public, and perform publicly and display publicly, by or on behalf of the Government. The Department of Energy will provide public access to these results of federally sponsored research in accordance with the DOE Public Access Plan. http://energy.gov/downloads/doe-public-access-plan. }}
\normalsize
\end{flushright}

\end{document}